\documentclass[10pt,a4paper]{article}
\usepackage[T1]{fontenc}
\usepackage[utf8]{inputenc}
\pdfoutput=1
\usepackage{lmodern}
\usepackage{textcomp}
\usepackage{microtype}
\usepackage[font=small,labelfont=bf]{caption}
\usepackage{titlesec}
\usepackage{tocloft}
\usepackage{amsmath,amssymb}
\usepackage{enumerate} 
\usepackage{amsfonts}
\usepackage{graphicx}
\usepackage{amsthm}
\usepackage{yfonts}
\usepackage{tikz-cd}
\usepackage{amssymb}
\usepackage{mathrsfs}
\usepackage[parfill]{parskip}
\usepackage{hyperref}
\usepackage[nameinlink]{cleveref}
\usepackage{MnSymbol }
\usepackage{verbatim}
\usepackage{array}
\titleformat{\subsection}[runin]
  {\normalfont\large\bfseries}{\thesubsection}{0em}{}
  \titleformat{\subsubsection}[runin]
  {\normalfont\normalsize\bfseries}{\thesubsubsection}{0em}{}
\newcommand\restr[2]{{
  \left.\kern-\nulldelimiterspace 
  #1
  \vphantom{\big|}
  \right|_{#2}
  }}
\newtheorem{theorem}{Theorem}
\newtheorem{conjecture}{Conjecture}[subsection]
\newtheorem{corollary}[theorem]{Corollary}
\newtheorem{lemma}[theorem]{Lemma}
\newtheorem{prop}[theorem]{Proposition}
\theoremstyle{definition}
\newtheorem*{rem}{Remark}

\newtheorem*{ex}{Example}
\newcommand{\C}{\mathbb{C}}

\newcommand{\fm}{\mathfrak{m}}
\newcommand{\fn}{\mathfrak{n}}
\newcommand{\gl}{\mathrm{GL}}
\newcommand{\fs}{\mathfrak{s}}
\newcommand{\Irr}{\mathfrak{Irr}}
\newcommand{\Rep}{\mathfrak{Rep}}
\newcommand{\Ho}{\mathrm{Hom}}
\newcommand{\cus}{\mathrm{cusp}}
\newcommand{\scus}{\mathrm{scusp}}
\newcommand{\ol}[1]{\overline{#1}}

\newcommand{\De}{\Delta}
\newcommand{\He}{\mathscr{H}}
\newcommand{\Ms}{\mathcal{MS}}
\newcommand{\lp}{\preceq}

\newcommand{\ind}{\mathrm{Ind}}
\newcommand{\hra}{\hookrightarrow}

\newcommand{\tfm}{{\widetilde{\fm}}}

\newcommand{\mo}{\mathfrak{o}}

\newcommand{\iso}{\overset{\cong}{\rightarrow}}
\newcommand{\ql}{\ol{\mathbb{Q}}_\ell}
\newcommand{\fl}{{\ol{\mathbb{F}}_\ell}}
\newcommand{\zl}{{\ol{\mathbb{Z}}_\ell}}

\newcommand{\trho}{{\Tilde{\rho}}}
\newcommand{\rl}{\mathrm{r}_\ell}
\newcommand{\Z}{\mathrm{Z}}
\newcommand{\I}{\mathrm{I}}

\newcommand{\ain}[3]{#1\in\{#2,\ldots,#3\}}
\newcommand{\bbrh}{\mathbb{N}(\mathbb{Z}[\rho])}

\newcommand{\sch}{\mathscr{S}_{\fl}(M_n(\mathrm{D}))}

\newcommand{\schr}{\mathscr{S}_{R}(M_n(\mathrm{D}))}

\newcommand{\fk}{\mathfrak{k}}

\newcommand{\TD}{{\widetilde{\De}}}
\newcommand{\RR}{\mathfrak{R}}
\newcommand{\bd}{\mathrm{D}}

\newcommand{\ZZ}{\mathbb{Z}}
\newcommand{\tpsi}{\widetilde{\psi}}

\newcommand{\ra}{\rightarrow}
\newcommand{\sra}{\twoheadrightarrow}

\newcommand{\dd}{\,\mathrm{d}_l}
\newcommand{\ddd}{\,\mathrm{d}}
\newcommand{\dr}{\,\mathrm{d}_r}
\newcommand{\bs}{\backslash}
\newcommand{\Cc}{\mathrm{C}}

\newcommand{\Dr}{\mathcal{D}_{l,\rho}}
\newcommand{\Dl}{\mathcal{D}_{r,\rho}}
\newcommand{\Dlm}{\mathcal{D}_{r,\rho,max}}
\newcommand{\Drm}{\mathcal{D}_{l,\rho,max}}
\newcommand{\drm}{d_{l,\rho}}
\newcommand{\dlm}{d_{r,\rho}}

\newcommand{\soc}{\mathrm{soc}}
\newcommand{\cp}{C_{c}^\infty(\op P)}
\newcommand{\op}{{\overline{P}}}
\newcommand{\wf}{W_\mathrm{F}}
\newcommand{\vc}{\mathrm{CV}}
\newcommand{\jl}{\mathrm{J}_\ell}

\newcommand{\irs}{\mathfrak{Irr}^\square}
\newcommand{\pa}{\mathrm{Pairs}_r}

\newcommand{\vt}{\widetilde{V}}
\newcommand{\ic}{\mathrm{IC}}
\newcommand{\scu}{\mathfrak{C}^\square}
\newcommand{\st}{\mathrm{St}}
\newcommand{\com}{^{\complement}}

\title{On modular representations of inner forms of $\mathrm{GL}_n$ over a local non-archimedean field}\author{Johannes Droschl}\date{}
\begin{document}

\maketitle

\begin{center}
    \textbf{Abstract}
\end{center}

 Let $\mathrm{F}$ be a local non-archimedean field of residue characteristic $p$ and $\ol{\mathbb{F}}_\ell$ an algebraic closure of a finite field of characteristic $\ell \neq p$. We extend the results of \cite{LMbinary} concerning $\square$-irreducible representations of inner forms of $\mathrm{GL}_n(\mathrm{F})$ to representations over $\fl$. As applications, we compute the Godement-Jacquet $L$-factor for any smooth irreducible representation over $\fl$ and show that the local factors of a representation agree with the ones of its $\mathrm{C}$-parameter defined in \cite{Ccor}. Moreover, we reprove that the classification of irreducible representations via multisegments due to Vignéras and M\'inguez-S\'echerre is indeed exhaustive without using the results of \cite{Ariki}.
Finally, we characterize the irreducible constituents of certain parabolically induced representations, as was already done by Zelevinsky over $\C$. 

\tableofcontents
\section{Introduction}
Let $\mathrm{F}$ be a local non-archimedean field of residue characteristic $p$ and $G$ a reductive group over $\mathrm{F}$. The classification of irreducible, smooth admissible representations of $G(\mathrm{F})$ over an algebraically closed field $R$ is of great importance in the Langlands program. This has been achieved by Bernstein and Zelevinsky in \cite{Zel}, \cite{Bernshtein1976REPRESENTATIONSOT}, \cite{ASENS_1977_4_10_4_441_0} for the case $G(\mathrm{F})=\mathrm{GL}_n(\mathrm{F})$ and $R=\mathbb{C}$ and was later extended by Tadi\'c in \cite{Tadic1990} to the case $G(\mathrm{F})=\mathrm{GL}_n(\mathrm{D})$, where $\mathrm{D}$ is a central division algebra over $\mathrm{F}$ of dimension $d^2$ and $\text{char }\mathrm{F}=0$. For $\text{char }\mathrm{F}\neq 0$, the classification was completed in \cite{Bad}.

The main focus of this paper will lie on the case $G(F)=\mathrm{GL}_n(\mathrm{D})$  and $R$ of characteristic $\ell\coloneq \mathrm{char}\, R$ possibly non-zero. 
Moreover, we restrict ourselves to the case $\ell\neq p$ to ensure the existence of Haar measures on reductive groups over $\mathrm{F}$. The theory of such representations has been explored and developed in \cite{alma991023733359705251}, \cite{Vigneras1998InducedRO}.
Passing from the $\ell=0$ to the $\ell\neq 0$ case, many results carry over, however, also several new phenomena appear. For example the notions of \emph{cuspidal} and \emph{supercuspidal} no longer coincide. 
Vignéras \cite{Vigneras1998InducedRO} and M\'inguez-S\'echerre \cite{M-S} generalized the classification of Bernstein and Zelevinsky and Tadic to the case of $\ell\neq p$. 
The irreducible representations are classified by purely combinatorial objects called \emph{multisegments}, formal finite sums of segments, which we will define in a moment. To each multisegment $\fm$ an irreducible representation $\Z(\fm)$ of $\mathrm{GL}_n(\mathrm{D})$, for a suitable $n$, can be assigned and every irreducible representation of $\mathrm{GL}_n(\mathrm{D})$ is isomorphic to a representation coming from such a multisegment. However, in the case $\ell\neq 0$ it might happen that two multisegments give rise to isomorphic representations. Because of this one has to restrict the map $\fm\mapsto \Z(\fm)$ to so-called \emph{aperiodic} multisegments to obtain a bijection. 

\sloppy Before we come to the main results of the paper, we need to introduce some notation.
For the rest of the paper a representation will be a smooth representation of finite length of $G_n\coloneq \mathrm{GL}_n(\mathrm{D})$ over some fixed algebraically closed field ${R\in\{\fl,\ql\}}$.
We denote the image of a representation $\pi$ of  $G_n$ in the Grothendieck group of representations of  $G_n$ as $[\pi]$. We denote the socle of $\pi$, \emph{i.e.} its maximal semi-simple subrepresentation, by $\soc(\pi)$. If $\pi_1,\pi_2$ are representations of $G_{n_1}$ and $G_{n_2}$, we 
 let $P_{(n_1,n_2)}$ be the parabolic subgroup of $\mathrm{GL}_{n_1+n_2}(\mathrm{D})$ defined over $\mathrm{F}$ containing the upper triangular matrices with Levi-component the block diagonal matrices $G_{n_1}\times G_{n_2}$.
 We then denote by $\pi_1\times \pi_2$ the normalized parabolically induced representation of $\pi_1\otimes \pi_2$ and by $r_{(n_1,n_2)}$ normalized parabolic restriction with respect to $P_{(n_1,n_2)}$. For a character $\chi$ of $\mathrm{F}^*$, we will write by abuse of notation also $\chi$ for the character $g\mapsto \chi(\det\nolimits'(g))$ of $G_n$, where $\det\nolimits'$ denotes the reduced norm. Write $\Irr_n$ for the set of isomorphism classes of irreducible representations of $G_n$ and $\Irr\coloneq \bigcup_{n\ge 0} \Irr_n$.

Let $\rho$ be a cuspidal irreducible representation of $G_m$. Then there exists an unramified character $v_\rho$ such that the representation $\rho\times \rho \chi$ is reducible if and only if $\chi \cong v_\rho^{\pm 1}$ and we set \[{o(\rho)\coloneq \#\{[\rho v_\rho^k]:k\in \mathbb{Z}\}}.\] In \cite[Section 4]{MST} the authors proved that $o(\rho)$ is the order of $q(\rho)$ in $R^*$, where $q(\rho)$ is an explicit power of $q$ depending on $\rho$. 
Set \[e(\rho)\coloneq \begin{cases}
    o(\rho)&\text{ if }o(\rho)>1,\\
    \ell&\text{ if } o(\rho)=1.
\end{cases}\]
For $a\le b\in\mathbb{Z}$ and $\rho$ a cuspidal representation of $G_m$, we define a segment as a sequence \[[a,b]_\rho\coloneq (\rho v_\rho^a,\ldots,\rho v_\rho^b).\]
 Let the cuspidal support of $[a,b]_\rho$ be $[\rho v_\rho^a]+\ldots+[\rho v_\rho^b]$, its length ${l([a,b]_\rho)\coloneq b-a+1}$ and its degree $\deg([a,b]_\rho)\coloneq m\cdot l([a,b]_\rho)$.
To a segment $[a,b]_\rho$ of length $n$ we can also associate an irreducible subrepresentation \[\Z([a,b]_\rho)\] of
\[\rho v_\rho^a\times\ldots\times \rho v_\rho^b,\]
which corresponds to the trivial character of the affine Hecke-algebra ${\mathscr{H}_R(n,q(\rho))}$ under the map of \cite[§4.4]{MST}.
 Two segments $[a,b]_\rho$ and $[a',b']_{\rho'}$ are called equivalent if they have the same length and $[\rho v_\rho^{a+i}]=[\rho' v_{\rho'}^{a'+i}]$ for all $i\in\mathbb{Z}$ and equivalent segments give isomorphic representations. We let for $\rho$ a cuspidal representation $\mathcal{S}(\rho)$ be the set of segments of the form $[a,b]_\rho$ and we call it the set of $\rho$-segments.
A multisegment $\fm=\De_1+\ldots+\De_k$ is a formal sum of equivalence classes of segments and we extend the notion of cuspidal support, length, degree, and equivalence linearly to multisegments.
 We call a multisegment \emph{aperiodic} if it does not contain a sub-multisegment equivalent to
\[[a,b]_\rho+\ldots+[a+e(\rho)-1,b+e(\rho)-1]_\rho.\]
Denote by $\Ms_{ap}$ the set of equivalence classes of aperiodic segments.
For some fixed cuspidal representation $\rho$ we denote by $\Ms(\rho)$ the set of all multisegments of the form
\[\fm=[a_1,b_1]_\rho+\ldots+[a_k,b_k]_\rho\] and $\Ms(\rho)_{ap}$ the aperiodic multisegments contained in $\Ms(\rho)$. In \cite[Section 9.5]{M-S} the authors define for a multisegment $\fm=\De_1+\ldots+\De_k$ a particular irreducible subquotient $\Z(\fm)$ of \[\I(\fm)\coloneq [\Z(\De_1)\times\ldots.\times \Z(\De_k)]\]
such that $\Z(\fm)\cong \Z(\fn)$ implies $\fm=\fn$ if both are aperiodic.

Following \cite{LMbinary}, we call an irreducible representation $\pi$ of $G_n$ \emph{$\square$- irreducible} if $\pi\times\pi$ is irreducible. We study these representations in \Cref{S:sir} using the theory of intertwining operators introduced in \cite{Dat}, which we recall in \Cref{S:intertwiningoperators}. 
The main proposition of this section is then the following.
\begin{prop}[\emph{cf}. \cite{LMbinary}, \cite{KKKOa}, \cite{KKKOb}]
Let $\pi\in \Irr$.
The following are equivalent.
\begin{enumerate}
    \item $\pi$ is $\square$-irreducible.
    \item For all $\sigma\in \Irr$, $\mathrm{soc}(\pi\times\sigma)$ is irreducible and appears with multiplicity $1$ in $\pi\times\sigma$.
     \item For all $\sigma\in \Irr$, $\mathrm{soc}(\sigma\times\pi)$ is irreducible and appears with multiplicity $1$ in $\sigma\times\pi$.
\end{enumerate}
If any of the above statements hold true, the maps
\[\sigma\mapsto \mathrm{soc}(\sigma\times\pi), \, \sigma\mapsto \mathrm{soc}(\pi\times\sigma)\] are injective.
 \end{prop} Note that over $\ql$, this was already proven in \cite{LMbinary}. We adapt their argument, which is originally due to \cite{KKKOa}, \cite{KKKOb}.
The proposition allows us to give the following definition. Let $\pi'\in \Irr$. If $\pi'\cong \mathrm{soc}(\sigma\times\pi)$ respectively $\pi'\cong \mathrm{soc}(\pi\times\sigma)$, we define the right-$\pi$-derivative $\mathcal{D}_{r,\pi}(\pi')\coloneq \sigma$ respectively the left-$\pi$-derivative $\mathcal{D}_{l,\pi}(\pi')\coloneq \sigma$. If $\pi'$ does not lie in the image of these maps, we set the derivatives to $0$.

One of the main differences between the $\fl$-case and the complex case is that not all cuspidal representations are $\square$-irreducible. Recall that
if $\rho$ is a cuspidal representation, then $\rho$ is $\square$-irreducible if and only if $o(\rho)>1$. The behavior of $\square$-irreducible and $\square$-reducible cuspidal representations differs quite drastically. For example, if $\rho$ is a $\square$-reducible cuspidal representation, its $L$-factor will always be trivial, and $\rho\times\rho$ it is of length $2$ and semisimple if $\ell>2$ and of length $3$ if $\ell=2$.
We say an irreducible representation has $\square$-irreducible cuspidal support if its cuspidal support consists of $\square$-irreducible cuspidal representations. 

In \Cref{S:Der} we study the specific situation where $\rho$ is a $\square$-irreducible cuspidal representation. 
The notion of $\rho$-derivatives was first introduced independently in \cite{Mder} and \cite{Jader} and further developed in \cite{ML} in the case $G_n$ and $R=\C$. Moreover, the theory was extended to classical groups in \cite{MA}. Here we extend the notion of derivatives to the case $\ell\neq 0$ and observe that most of the results carry over.
\begin{corollary}
 If $\Dl(\pi)\cong \Dl(\pi')\neq 0$ or $\Dr(\pi)\cong \Dr(\pi')\neq 0$, then $\pi\cong \pi'$. Furthermore,
\[\Dl(\pi)^\lor\cong \mathcal{D}_{l,\rho^\lor}(\pi^\lor).\]
\end{corollary}
For $\rho$ a $\square$-irreducible cuspidal representation, we moreover define four maps \[\soc(\rho,{}\cdot{}),\soc({}\cdot{},\rho),\Dl,\Dr\colon \Ms_{ap}(\rho)\ra \Ms_{ap}(\rho)\] using a combinatorial description such that
\[\Dr(\soc(\rho,\fm))=\fm=\Dl(\soc(\fm,\rho)).\]
We then can compute the $\rho$-derivatives of the corresponding representation $\Z(\fm)$ simply as follows.
\begin{prop}
Let $\fm$ be an aperiodic multisegment and $\rho$ a $\square$-irreducible cuspidal representation.
    \[\Dl(\Z(\fm))=\Z(\Dl(\fm)),\, \Dr(\Z(\fm))=\Z(\Dr(\fm)).\]
\end{prop}
As a corollary we obtain that $\Z(\soc(\rho,\fm))=\mathrm{soc}(\rho\times \Z(\fm))$. In particular, this is used to give a new proof of the following theorem.
\begin{theorem}\label{T:I3}
Let $\pi$ be an irreducible smooth representation of $G_n$ whose cuspidal support is $\square$-irreducible cuspidal. Then there exists a multisegment $\fm$ such that $\Z(\fm)\cong \pi$.
\end{theorem}
\begin{rem}
There are two main reasons why such a new proof could be of interest. The first is that the original proof is quite involved and relies heavily on the techniques from geometric representation theory. Secondly, the existing proof does not give much information on the inverse $\Z^{-1}$ of $\Z$, which is something that we will need in the computations of the $L$-factors. Let us now expand on both of these points. We start by recalling the existing proof of the surjectivity of $\Z$ in \cite{M-S}. One first reduces the claim to the situation where all representations and multisegments involved have cuspidal support $\fs$ contained in the cuspidal line spanned by $\rho$.
Let $\Ms(\fs)$ be the set of aperiodic multisegments with cuspidal support $\fs$ and $\Irr(\fs)$ the set of isomorphism classes of irreducible representations with cuspidal support $\fs$.
It is then possible to define a central character $\chi_{\fs}$ of the affine Hecke-algebra $\mathscr{H}_R(n,q(\rho))$ with the property that there exists a bijection between $\Irr(\fs)$ and \[\mathfrak{Irr}(\mathscr{H}_R(n,q(\rho)),\fs),\] the set of isomorphism classes modules of $\mathscr{H}_R(n,q(\rho))$ with central character $\chi_{\fs}$.
The proof of the classification of simple modules of the Hecke algebra has two very different flavors, depending on whether $o(\rho)>1$ or $o(\rho)=1$. If $o(\rho)=1$, the central character of $\mathscr{H}_R(n,1)$ is always trivial and hence the representation theory of the Hecke-algebra reduces in this case to the representation theory of the group algebra $R[\mathrm{S}_n]$ or said differently, to the modular representation theory of the symmetric group. Although some questions on modular representations of the symmetric groups remain to this day unsolved, the classification of irreducible representations falls not among them. It is a classical result that the irreducible representations are classified by $\ell$-regular partitions and there is an explicit construction of these, see \Cref{S:repsym}. Thus one has an explicit description of the inverse map $\Z^{-1}$ in this case.

On the other hand, if we assume that $o(\rho)>1$, the situation is much more involved. Here one uses \cite[Theorem B]{Ariki} to show that $\Ms(\fs)$ and $\mathfrak{Irr}(\mathscr{H}_R(n,q(\rho)),\fs)$ have the same cardinality and hence obtains as a consequence that the irreducible representations are classified by aperiodic multisegments. In the proof of Theorem B of \cite{Ariki} the author first uses the geometric representation theory of \cite{Chriss1997RepresentationTA} to show that in the case $\ell\neq 0$, $\mathfrak{Irr}(\mathscr{H}_R(n,q(\rho)),\fs)$ is parametrized by simple modules with a particular central character $\Tilde{\chi}_\fs$ of an affine Hecke-algebra $\mathscr{H}_\C(n,\xi)$ over $\C$ at a root of unity $\xi$. Using geometric representation theory again, the author then argues that those are classified by the aperiodic multisegments. 
Thus we proved that $\Irr(\fs)$ is finite and of the same cardinality as $\Ms(\fs)$. Since the map $\fm\mapsto \Z(\fm)$ induces an injective map between these two sets, it has to be bijective. This gives unfortunately not the kind of information on the inverse map $\Z^{-1}$ we need later on. However, since a cuspidal representation is $\square$-irreducible if and only if $o(\rho)>1$, this second, harder, case is covered by our new surjectivity proof, and thus gives us a more representation-theoretic description of $\Z^{-1}$.
\end{rem}
Let us now return to the remaining contents of this paper. We continue by discussing how the Aubert-Zelevinsky dual of an irreducible representation interacts with taking derivatives.
Recall that the Aubert-Zelevinsky dual, see for example \cite{Zel}, \cite{Aub}, \cite{MoeglinetJ1986}, \cite{MSI}, is a map
\[(\cdot)^{*}\colon \Irr\ra \Irr,\, \pi\mapsto \pi^*\]
preserving the cuspidal support. For $\fm$ an aperiodic multisegment, we write $\langle\fm\rangle\coloneq\Z(\fm)^*$.
Using derivatives we then compute in \Cref{S:GJ} the Godement-Jacquet local $L$-factors$\mod\ell$ introduced in \cite{Mzeta}.
\begin{theorem}
    Fix an additive character $\psi$ of $F^*$ and
    let $\Delta=[a,b]_\rho$ be a segment over $R$. If $\rho\cong\chi$ for an unramified character $\chi$ of $\mathrm{F}$ and $q^d\neq 1$, then
        \[L(\langle\Delta\rangle,T)=\frac{1}{ 1-\chi(\varpi_{\mathrm{F}})q^{-db+{\frac{1-d}{2}}}T}, \]
        where $\varpi_\mathrm{F}$ is a uniformizer of $\mathfrak{o}_\mathrm{F}$.
        Otherwise \[L(\langle\Delta\rangle,T)=1.\]
    More generally, if $\fm=\De_1+\ldots +\De_k$ is an aperiodic multisegment
    then \[L(\langle\fm\rangle,T)=\prod_{i=1}^kL(\langle\De_i\rangle,T),\, \epsilon(T,\langle\fm\rangle,\psi)=\prod_{i=1}^k\epsilon(T,\langle\De_i\rangle,\psi)\] and
    \[\gamma(T,\langle\fm\rangle,\psi)=\prod_{i=1}^k\gamma(T,\langle\De_i\rangle,\psi).\]
\end{theorem}
We recall now the definition of the map \[\Cc\colon \Irr_n\iso \{\Cc-\text{parameters of length }n\}, \]\emph{cf.} \cite{Ccor}. Let $W_\mathrm{F}$ be the Weil group of $\mathrm{F}$ and $\nu$ the unique unramified character of $W_\mathrm{F}$ acting on the Frobenius as $\nu(\mathrm{Frob})=q^{-1}$.
The set of $\Cc$-parameters over $\fl$ is then a subset of the equivalence classes of $W_\mathrm{F}$-semisimple Deligne $\fl$-representations, \emph{i.e.} pairs $(\Phi, U)$, where $\Phi$ is a semi-simple $W_\mathrm{F}$-representation over $\fl$ and $U\in \Ho_{W_\mathrm{F}}(\nu \Phi,\Phi).$ The length of such a representation refers to its dimension.
For a precise definition see \Cref{S:LLC}. In \cite{Ccor} to each such $\Cc$-parameter $(\Phi,U)$ an $L$-factor $L((\Phi,U),T)$, an $\epsilon$-factor $\epsilon(T,(\Phi,U),\psi)$ and a $\gamma$-factor $\gamma(T,(\Phi,U),\psi)$ is then associated.
\begin{corollary}
    The map $\Cc$ respects the local factors, \emph{i.e.} for $\pi\in \Irr$
    \[L(\pi,T)=L(\Cc(\pi),T),\, \epsilon(T,\pi,\psi)=\epsilon(T,\Cc(\pi),\psi),\, \gamma(T,\pi,\psi)=\gamma(T,\Cc(\pi),\psi).\]
\end{corollary}
In \Cref{S:Cons} we describe the structure of the representations of the form $\I(\fm)$, which is heavily motivated by the analysis done in \cite{Zel} in the case $R=\ql$. Fix a cuspidal representation $\rho$ and a cuspidal support $\fs$ of the form
\[\fs=d_0[\rho]+\ldots+d_{o(\rho)-1}[\rho v_\rho^{o(\rho)-1}],\] where $d_i$ be the multiplicity of $\rho v_\rho^i,\, \ain{i}{0}{o(\rho)-1}$ in $\fs$. Let $Q={A^+_{o(\rho)-1}}$ be the affine Dynkin quiver with $o(\rho)$ vertices numbered $\{0,\ldots,o(\rho)-1\}$ and an arrow from $i$ to $j$ if $j=i+1\mod o(\rho)$. Hence $\fs$ gives rise to a graded vector space \[V(\fs)\coloneq \bigoplus_{i=1}^{o(\rho)-1} \C^{d_i}.\]
We let $N(V(\fs))$ be the $\C$-variety of nilpotent representations of $Q$ with underlying vector space $V(\fs)$. Then $N(V(\fs))$ admits an action by
\[\gl_{d_0}(\C)\times\ldots\times \gl_{d_{o(\rho)-1}}(\C),\]
whose orbits are in bijection with irreducible representations $\Z(\fm)$ whose cuspidal support is $\fs$. We write $X_\fm$ for the orbit corresponding to $\Z(\fm)$ and write $\fn\lp \fm$ if
$X_\fm\subseteq \overline{X_\fn}$
\begin{theorem}
         Let $\fm,\fn\in \Ms(\rho)$ not necessarily aperiodic. Then $\Z(\fn)$ appears as a subquotient of $\I(\fm)$ if and only if the cuspidal supports of $\fn$ and $\fm$ agree and $\fn\lp \fm$.
\end{theorem}
It is thus expected that the singularities of the space $N(V(\fs))$ allow one to obtain a deeper understanding of the irreducible representations with cuspidal support $\fs$. In this spirit, we extended \cite[Conjecture 4.2]{LMsquare} to representations over $\fl$.
\subsection*{Acknowledgements:}
I would like to thank Alberto M\'inguez for the fruitful discussions and patiently listening to many of my (wrong) ideas. I would also like to thank Nadir Matringe, Vincent Sécherre and Thomas Lanard for their useful remarks.  Finally, I would like to thank the anonymous referee who pointed out several mistakes in the first version of this paper. This work has been supported by the research projects P32333 and PAT4832423 of the Austrian Science Fund (FWF).
\section{Notation and preliminaries}
\counterwithin{theorem}{subsection}
\subsection{ ~}Let $\mathrm{F}$ be a local non-archimedean field, $\mathfrak{o}_{\mathrm{F}}$ its ring of integers with uniformizer $\varpi_\mathrm{F}$ and $\mathrm{k}\coloneq \mathfrak{o}_{\mathrm{F}}/(\varpi_{\mathrm{F}})$ its residue field. Let $p\coloneq \mathrm{char}\, \mathrm{k}$ and $q$ the cardinality of $\mathrm{k}$. Finally, let $\mathrm{D}$ be a division algebra of reduced degree $d$ over $\mathrm{F}$. Fix an algebraically closed field $R\in\{\fl,\ql\}$ with $\ell \neq p$. For a reductive group $G$ over $\mathrm{F}$, we let $\mathfrak{Rep}_G=\mathfrak{Rep}_G(R)$ be the category of smooth, finite length representations of $G(\mathrm{F})$ over $R$ with morphisms the $G(\mathrm{F})$-equivariant $R$-linear morphisms. We will focus on the case $G_n\coloneq \mathrm{GL}_n(\mathrm{D})$, in which case we write $\mathfrak{Rep}_n=\mathfrak{Rep}_{G_n}$
and set
\[\mathfrak{Rep}=\mathfrak{Rep}(R)\coloneq   \bigcup_{n\ge 0} \mathfrak{Rep}_n.\] 
We let $\mathfrak{R}_n(R)=\mathfrak{R}_n$ be the Grothendieck group of representations in $\mathfrak{Rep}_n$ and set \[\mathfrak{R}=  \mathfrak{R}(R)\coloneq \bigoplus_{n\ge0} \mathfrak{R}_n.\] 
Let $\mathfrak{Irr}_n=\Irr_n(R)$ be the set of isomorphism classes of irreducible smooth admissible representations of $G_n$ and set
\[\mathfrak{Irr}\coloneq   \bigcup_{n\ge0}\mathfrak{Irr}_n.\]
From now on a representation $\pi$ of $G_n$ means a smooth, finite length representation of $G_n$ over $R$ and its image in $\mathfrak{R}_n$ is denoted as $[\pi]$ and $\deg(\pi)=n$. We  write $\pi^\lor$ for the contragredient representation, $\tau\hra \pi$ if $\tau$ is a subrepresentation of $\pi$ and $\tau\le \pi$ if $\tau$ is a subquotient of $\pi$, \emph{i.e.} a quotient of a subrepresentation of $\pi$. We denote by $\soc(\pi)$ the socle of $\pi,$ \emph{i.e.} the maximal semi-simple subrepresentation of $\pi$ and by $\cos(\pi)$ the cosocle of $\pi,$ \emph{i.e.} the maximal semi-simple quotient of $\pi$.
For an element $[\mathrm{X}]\in \mathfrak{R}$, we write $[\mathrm{X}]\ge 0$ if the multiplicity of all isomorphism classes of irreducible representations in $[\mathrm{X}]$ is greater or equal than $0$ and we write $[\mathrm{X}]\ge [\mathrm{Y}]$ if $[\mathrm{X}]-[\mathrm{Y}]\ge 0$.
Furthermore, if $S$ is a set we let $\mathbb{N}(S)$ be the commutative monoid consisting of functions $S\rightarrow \mathbb{N}$ with finite support. Finally, we fix once and for all an additive character $\tpsi\colon F\ra \zl$ with open kernel and let $\psi$ be its reduction mod $\ell$.

\subsection{ ~}\label{S:parabind}
Let $n\in\mathbb{Z}_{> 0}$ and denote by $S_n$ the $n$-th symmetric group. A tuple $\alpha=(\alpha_1,\alpha_2,\ldots,\alpha_t),\, \alpha_i\in \mathbb{Z}_{>0}$ will be called a \emph{composition} of $n$ if $\sum_{i=1}^t\alpha_i=n$.
Its reverse composition $\ol{\alpha}$ is denoted by $\ol{\alpha}\coloneq (\alpha_t,\ldots,\alpha_1)$ and we call $\alpha$ ordered or a \emph{partition} if \[\alpha_1\ge\alpha_2\ge\ldots\ge \alpha_t.\] 
If $\alpha'=(\alpha_1',\ldots,\alpha_{t'}')$ is a second composition of $n$ we define the intersection $\alpha\cap\alpha'$ recursively as follows.
\[\alpha\cap\alpha'\coloneq \begin{cases}(\alpha_1,(\alpha_2',\ldots,\alpha_{t'}')\cap(\alpha_2,\ldots,\alpha_t))&\text{ if }\alpha_1=\alpha_1',\\
(\alpha_1,(\alpha_1'-\alpha_1,\ldots,\alpha_{t'}')\cap(\alpha_2,\ldots,\alpha_t))&\text{ if }\alpha_1<\alpha_1',\\
(\alpha_1',(\alpha_2',\ldots,\alpha_{t'}')\cap(\alpha_1-\alpha_1',\ldots,\alpha_t))&\text{ if }\alpha_1>\alpha_1'.
\end{cases}\]
If $\alpha$ and $\alpha'$ are compositions of $n$ we say $\alpha\le\alpha'$ if \[\sum_{i=1}^j\alpha_j\le \sum_{i=1}^j\alpha_j'\] for all $j\in\{1,\ldots,\mathrm{min}(t,t')\}$.
We let \[G_\alpha\coloneq G_{\alpha_1}\times\ldots\times G_{\alpha_t}\] and every irreducible representation $\pi$ of $G_\alpha$ is then of the form \[\pi\cong\pi_1\otimes\ldots\otimes \pi_t,\] for $[\pi_i]\in \Irr_{\alpha_i}$ and $i\in \{1,\ldots,t\}$.
As above let $\mathfrak{Rep}_\alpha$ be the category of smooth, finite length representations of $G_\alpha$. The Grothendieck group of representations in $\mathfrak{Rep}_\alpha$ of finite length is denoted by $\mathfrak{R}_\alpha$, and we write $\mathfrak{Irr}_\alpha$ for the set of isomorphism classes of irreducible representations of $G_\alpha$. 
For $\ain{j}{1}{t}$ let $(\beta_{1,j},\beta_{2,j},\ldots,\beta_{i_j,j})$ be a composition of $\alpha_j$. We then call \[(\beta_{1,1},\ldots,\beta_{i_1,1},\beta_{1,2},\ldots,\beta_{i_2,2},\ldots.,\beta_{1,t},\ldots,\beta_{i_t,t})\] a subcomposition of $\alpha$.

Recall the exact functors of \emph{parabolic induction} and \emph{parabolic restriction}. We fix a choice of a square root $\sqrt{q}$ in $R$. If $P=M\rtimes U$ is a parabolic subgroup defined over $\mathrm{F}$ of $G_\alpha$ with Levi-component $M$ and unipotent-component $U$, we let $\delta_P$ be the module of $P$, \emph{i.e.} the map $\delta_P\colon M\rightarrow \C^*$ \[\delta_P(m)=q^{v_P(m)},\] where $v_P$ is a certain morphism from $M$ to $\mathbb{Z}$ depending on $P$. Our choice of square root of $q$ allows us then to define the character $\delta_P^{\frac{1}{2}}$.
We denote then by \[\mathrm{Ind}_{P}^{G_\alpha }:\Rep_M\rightarrow \Rep_{G_\alpha}\text{ resp. }r_{G_\alpha,P}: \Rep_{G_\alpha}\rightarrow \Rep_M.\]
Recall that the functor $\mathrm{Ind}_{P}^{G_\alpha}$ is the left adjoint of $r_{G_\alpha,P}$, \emph{i.e.}
\[\Ho_{G_\alpha}(-,\mathrm{Ind}_{P}^{G_\alpha}(-))=\Ho_{M}(r_{G_\alpha,P}(-),-) \text{ (Frobenius reciprocity)}.\]
We also fix the minimal parabolic subgroup $P_0$ defined over $\mathrm{F}$ of upper triangular matrices in $G_\alpha$. Note that each parabolic subgroup $P$ defined over $\mathrm{F}$ containing $P_0$ corresponds to a subcomposition $\beta$ of $\alpha$ in such a way that the Levi subgroup of $P$ is $G_\beta$. If $\alpha=(n)$ we write $r_{\beta}\coloneq r_{G_n,P_\beta},\, \ind_\beta\coloneq \ind_{P_\beta}^{G_n}$ and if \[(\beta_{1,1},\ldots,\beta_{i_1,1},\beta_{1,2},\ldots,\beta_{i_2,2},\ldots.,\beta_{1,t},\ldots,\beta_{i_t,t})\] is a subcomposition of $\alpha$ and \[\pi=\pi_{1,1}\otimes\ldots\otimes\pi_{i_1,1}\otimes \pi_{1,2}\otimes\ldots\otimes\pi_{i_2,2}\otimes\ldots\otimes \pi_{1,t}\otimes\ldots\otimes\pi_{i_t,t}\] is an object of $\Rep_\beta$, we write
\[\pi_{1,1}\times\ldots\times \pi_{i_1,1}\otimes \pi_{1,2}\times\ldots\times \pi_{i_2,2}\otimes\ldots\otimes \pi_{1,t}\times \ldots\times \pi_{i_t,t}\coloneq \mathrm{Ind}_{P_\beta}^{G_\alpha}(\pi).\] 
Since both induction and reduction are exact, they induce functors on the corresponding Grothendieck groups.
In particular, parabolic induction induces a product
\[\times:\mathfrak{R}\times \mathfrak{R}\rightarrow \mathfrak{R}, ([\pi_1],[\pi_2])\mapsto [\pi_1\times\pi_2],\] which equips $\mathfrak{R}$ with the structure of an commutative algebra, see \cite[1.16]{alma991023733359705251} for $\ell>2$, \cite[Theorem 1.9]{Zel} for the case $R=\mathbb{C}$ and \cite[Proposition 2.6]{M-S} in general. If $\pi$ is a representation we will write $\pi^k\coloneq \pi\times\ldots\times \pi$ for the product of $k$ copies of $\pi$.
We will also use this opportunity to state two lemmas, which were proven in \cite{LMsquare} over $\ql$, but the presented proofs work as well \emph{muta mutandis} over $\fl$.
\begin{lemma}[{\cite[Lemma 2.1]{LMsquare}}]
    Let $P=M\rtimes U$ be a parabolic subgroup of $G=G_n$ and $Q_1=N_1\rtimes V_1,\, Q_2=N_2\rtimes V_2$ parabolic subgroups of $G$ containing $P$. Let $\pi\in \Rep(M),\, \tau_1\in\Rep(N_1),\, \tau_2\in\Rep(N_2)$ such that there exist injective morphisms
    \[\iota_1\colon\tau_1\ra \ind_{N_1\cap P}^{N_1}(\pi),\,  \iota_2\colon\tau_2\ra \ind_{N_2\cap P}^{N_2}(\pi)\]
    and there exists \[ \iota_3\colon\ind_{Q_1}^G(\tau_1)\ra \ind_{Q_2}^G(\tau_2)\] such that
    \[\ind_{Q_1}^G(\iota_1)=\ind_{Q_2}^G(\iota_2)\circ \iota_3.\]
    Then there exists a representation $\tau$ of $M$ and an injective map $\iota\colon\tau\ra \pi $ such that $\iota_1$ factors through $\ind_{N_1\cap P}^{N_1}(\iota)$ and $\ind_{N_2\cap P}^{N_2}(\iota)$ factors through $\iota_2$.
\end{lemma}
\begin{corollary}[{\cite[Corollary 2.2]{LMsquare}, \cite[Lemma 3.1]{KKKOb}}]\label{C:3ind}
    Let $n_1,n_2,n_3\in\ZZ_{>0}$ and 
    \[\pi_i\in\Rep(G_{n_i}),\, i\in\{1,2,3\}.\]
    Let $\sigma$ be a subrepresentation of $\pi_1\times\pi_2$ and $\tau$ a subrepresentation of $\pi_2\times\pi_3$ such that as subrepresentations of $\pi_1\times\pi_2\times\pi_3$
      \[\sigma\times\pi_3\hra \pi_1\times \tau.\]
    Then there exists a subrepresentation $\omega$ of $\pi_2$ such that     \[\sigma\hra \pi_1\times\omega ,\, \omega\times \pi_3\hra \tau.\]
    In particular, if $\pi_2$ is irreducible and $\sigma\neq 0$, then $\sigma=\pi_2$
\end{corollary}
\subsection{ Geometric Lemma}\label{S:geolem} One of the key tools we will use is the Geometric Lemma, \emph{cf.} \cite[Theorem 2.19]{alma991023733359705251}, \cite[Theorem 5.2]{ASENS_1977_4_10_4_441_0}. Let $A$ be the maximal torus defined over $\mathrm{F}$ of our minimal parabolic subgroup $P_0$ and $W(G,A)$ the Weyl group of $G_n$. The choice of $P_0$ corresponds to a choice of a basis $S$ of the root system $\phi$ of $G_n$ and hence of a set of positive roots $\phi^+$. Let $\alpha$ be a composition of $n$ corresponding to a parabolic subgroup containing $P_0$, which in turn corresponds to a subset $I\subseteq S$. For $w\in W(G,A)$ we let $\alpha^w$ be the composition corresponding to the subset $w(I)\subseteq S$.
For $\alpha_1,\alpha_2$ compositions of $n$, which correspond to subsets $I_1,I_2$ of $S$, we denote
\[W(\alpha_1,\alpha_2)\coloneq \{w\in W(G,A):w(I_1)\subseteq \phi^+ \text{ and } w^{-1}(I_2)\subseteq \phi^+\}\]
 and for $w\in W(\alpha_1,\alpha_2)$ we set $\alpha_2'\coloneq   \alpha_2\cap \alpha_1^w$ and $\alpha_1'=\alpha_1\cap \alpha_2^{w^{-1}}$. For $g\in G_{\alpha_2'}$ and $w\in W(\alpha_1,\alpha_2)$ the element $w^{-1}gw\in G_{\alpha_1'}$ and hence pulling back via the group morphism $g\mapsto w^{-1}gw$ allows us to define a functor
 \[
 w:\Rep_{\alpha_1'}\rightarrow \Rep_{\alpha_2'}.
 \] Set \[F(w)\coloneq   \ind_{P_{\alpha_2'}}^{G_{\alpha_2}}\circ w \circ r_{G_{\alpha_1},P_{\alpha_1'}}\colon \Rep_{\alpha_1}\rightarrow \Rep_{\alpha_2}.\] Let $\le$ the Bruhat order on $W(\alpha_1,\alpha_2)$ and $w_0\ge \ldots\ge w_k$ the corresponding ordering of the elements in $W(\alpha_1,\alpha_2)$. Then \[r_{\alpha_2}\circ \mathrm{Ind}_{\alpha_1}\colon\Rep_{\alpha_1}\rightarrow \Rep_{\alpha_2}\] is glued together from the exact functors $F(w), w\in W(\alpha_1,\alpha_2)$, \emph{i.e.} there exists a filtration
\[0=F_{-1}\subseteq F_{0}\subseteq\ldots \subseteq F_k=r_{\alpha_2}\circ \mathrm{Ind}_{\alpha_1}\]
of $r_{\alpha_2}\circ \mathrm{Ind}_{\alpha_1}$ such that \[F_i\bs F_{i-1}\cong F(w_i),\, \ain{i}{0}{k}.\]
As a consequence
\[[r_{\alpha_2}\circ \mathrm{Ind}_{\alpha_1}]=\sum_{w\in W(\alpha_1,\alpha_2)}[F(w)].\]

This can be described in a more combinatorial way, see for example \cite[Section 1.2]{ML}. 
Let $\alpha=(n),\, \alpha_1=(m_1,\ldots,m_t),\,\alpha_2=(n_1,\ldots,n_{t'})$, $\pi_1\otimes\ldots\otimes \pi_t\in \Rep_{\alpha_1}$ and $\mathrm{Mat}^{\alpha_1,\alpha_2}$ the set of $t\times {t'}$ matrices $B=(b_{i,j})$ with non-negative integer coefficients such that for all $(i,j)\in \{1,\ldots,t\}\times \{1,\ldots,t'\}$ 
\[\sum_{j=1}^{t'}b_{i,j} =m_i,\, \sum_{i=1}^tb_{i,j} =n_j.\]
For $B\in \mathrm{Mat}^{\alpha_1,\alpha_2}$ and $i\in \{1,\ldots,t\}$, $\beta_{B,i}\coloneq (b_{1,i},\ldots,b_{i,{t'}})$ is a composition of $m_i$. Let $l_i$ be the length of $r_{\beta_{B,i}}(\pi_i)$ and write the irreducible decomposition factors of $r_{\beta_{B,i}}(\pi_i)$ as
\[\sigma_i^{(k)}=\sigma_{i,1}^{(k)}\otimes\ldots\otimes \sigma_{i,{t'}}^{(k)},\,\ain{k}{1}{l_i}\]
with $[\sigma_{i,j}^{(k)}]\in \Irr_{b_{i,j}}$. For $j\in\{1,\ldots,{t'}\}$ and a sequence $\underline{k}=(k_1,\ldots,k_t)$ with $\ain{k_i}{1}{l_i}$ set 
\[\Sigma_j^{B,\underline{k}}\coloneq  \sigma_{1,j}^{(k_1)}\otimes\ldots\otimes \sigma_{t,j}^{(k_t)}.\]
Then \[[r_{\alpha_2}(\pi_1\times\ldots\times\pi_t)]=\sum_{B\in \mathrm{Mat}^{\alpha_1,\alpha_2},\,\underline{k}}[\Sigma_1^{B,\underline{k}}\otimes \ldots\otimes \Sigma_{t'}^{B,\underline{k}}].\]
\subsection{ }An irreducible representation $\rho$ of $G_m$ is called \emph{cuspidal} if $r_\alpha(\rho)=0$ for all nontrivial compositions $\alpha$ of $m$ and \emph{supercuspidal} if $\rho$ is not the subquotient of a nontrivially parabolically induced representation. We denote the corresponding subsets of $\Irr$ as $\mathfrak{C}$ resp. $\mathfrak{S}$. From Frobenius reciprocity follows immediately that if $\rho$ is supercuspidal it is cuspidal.

Let $\rho, \rho'$ be cuspidal representations of $G_m$ and $G_{m'}.$ Then by
\cite[§ 4.5]{MST} there exists an unramified character $v_\rho$ of $G_m$ such that $\rho\times\rho'$ is irreducible if and only if $\rho'\cong \rho v_\rho$ or $\rho'\cong \rho v_\rho^{-1}$. We set \[\mathbb{Z}[\rho]\coloneq   \{[\rho v_\rho^k],k\in\mathbb{Z}\}\]
and let $o(\rho)$ be the cardinality of $\mathbb{Z}[\rho]$.
One can associate to $\rho$ a finite extension $\mathrm{k}(\rho)$ of $\mathrm{k}$ of cardinality $q(\rho)$, see \cite[Section 4]{MST}, and define a number $e(\rho)$ as the smallest integer $k$ such that \[1+q(\rho)+\ldots+q(\rho)^{k-1}=0\mod \ell.\] Then $e(\rho)$ satisfies \[e(\rho)= \begin{cases}
    o(\rho)&\text{ if }o(\rho)>1,\\
    \ell&\text{ if } o(\rho)=1,
\end{cases}\] by \cite[Lemma 4.41]{MST}.

There exist surjective maps with finite fibers
\[\mathrm{cusp}:\Irr\rightarrow \mathbb{N}(\mathfrak{C})\text{ and }\mathrm{scusp}:\Irr\rightarrow \mathbb{N}(\mathfrak{S})\]
called \emph{cuspidal support} and \emph{supercuspidal support}.
They are defined by \[\cus^{-1}([\rho_1]+\ldots+[\rho_k])=\{ [\pi]\in \Irr: \pi \hra \rho_{\sigma(1)}\times\ldots\times \rho_{\sigma(k)}\text{ for some }\sigma\in S_k\},\]
\[\scus^{-1}([\rho_1]+\ldots+[\rho_k])=\{[\pi]\in \Irr: [\pi]\le[\rho_1\times\ldots\times \rho_k]\}.\]
We call $k$ the length of the cuspidal support $[\rho_1]+\ldots+[\rho_k]$. A cuspidal representation $\rho$ is called \emph{$\square$-irreducible} if $o(\rho)>1$ and we say an irreducible representation has $\square$-irreducible cuspidal support if its cuspidal support consists of $\square$-representations.
\begin{lemma}\label{L:H2}
Let $\pi$ be an irreducible representation and $\mathfrak{s}=\cus(\pi)$. We write \[\mathfrak{s}=\mathfrak{s}_1+\ldots+\mathfrak{s}_k\] with $\mathfrak{s_i}\in \mathbb{N}(\mathbb{Z}[\rho_i])$ and $\mathbb{Z}[\rho_i] \neq \mathbb{Z}[\rho_j]$ for $i\neq j$. Then there exist irreducible representations $\pi_1,\ldots,\pi_k$ with $\cus(\pi_i)=\mathfrak{s}_i$ and
\[\pi\cong \pi_1\times\ldots\times \pi_k\]
\end{lemma}
\begin{proof}
Since parabolic induction is exact, we can find irreducible representations $\pi_1,\ldots,\pi_k$ such that $\cus(\pi_i)=\mathfrak{s}_i$ and $\pi\hra \pi_1\times\ldots\times \pi_k$. In \cite[Proposition 5.9]{M-S} it was proven that if $k=2$, $\pi_1\times\ldots\times \pi_k$ is irreducible and their method extends easily to the general case $k>2$.
Thus the claim follows.
\end{proof}
\subsection{ Multisegments} We will now fix our notations regarding segments, the central combinatorial objects in the classification of \cite{Zel} and \cite{M-S}.
For $\rho$ a cuspidal representation of $G_m$ and integers $a\le b$ we define a \textit{segment} as the sequence \[[a,b]_\rho\coloneq (\rho v_\rho^a,\rho v_\rho^{a+1},\ldots,\rho v_\rho^b).\] The length of such a segment is \[l([a,b]_\rho)\coloneq b-a+1\] and its degree is \[\deg([a,b]_\rho)\coloneq m(b-a+1).\] 
We will also sometimes consider a formal, empty segment $[]$ of length and degree $0$.
Two segments $[a,b]_\rho$ and $[a',b']_{\rho'}$ are said to be equivalent if they have the same length and $[\rho v_\rho^{a+i}]=[\rho' v_{\rho'}^{a'+i}]$ 
for all $i\in \mathbb{Z}$.
The set of equivalence classes of segments will be denoted by $\mathcal{Seg}$ and for a fixed $\rho$ the set of segments of the form $[a,b]_\rho$ will be denoted as $\mathcal{S}(\rho)$ and we call them $\rho$-segments.

If $\Delta=[a,b]_\rho$ is a segment, we let $a_\rho(\Delta)=a$ and $b_\rho(\Delta)=b$ seen as elements in $\ZZ/ (o(\rho)\ZZ)$. We write 
 \[^-\Delta\coloneq[a+1,b]_\rho,\,\Delta^-\coloneq[a,b-1]_\rho,\,\Delta^+\coloneq[a,b+1]_\rho,\, ^+\Delta\coloneq[a-1,b]_\rho,\,[a]_\rho\coloneq [a,a]_\rho \]
 and $\De^\lor\coloneq [-b,-a]_{\rho^\lor}$. The operation $\De\mapsto {}^-\De$ will sometimes be called \emph{shortening the segment } $\De$ \emph{by }$1$\emph{ on the left} and similarly for the others.
 If $a+1>b$, $^-\De$ is the empty segment and similarly for the other operations.
A segment $\Delta=[a,b]_\rho$ \emph{precedes} $\Delta'=[a',b']_{\rho'}$ if the sequence \[(\rho v_\rho^a,\ldots,\rho v_\rho^b ,\rho' v_{\rho'}^{a'},\ldots,\rho' v_{\rho'}^{b'})\]
contains a subsequence which is, up to isomorphism, a segment of length greater than $l(\De)$ and $l(\De')$.
The segments $\De$ and $\De'$ are called \emph{unlinked} if $\De$ does not precede $\De'$ and $\De'$ does not precede $\De$.

A \emph{multisegment} $\fm=\De_1+\ldots+\De_k$ is a formal finite sum of equivalence classes of segments, \emph{i.e.} an element in $\mathbb{N}[\mathcal{S}]$. We extend the notion of length and degree linearly to multisegments as well as the dual operation $(-)^\lor$. The set of multisegments is denoted by $\mathcal{MS}$ and $\mathcal{MS}(\rho)$ is the set of multisegments $\fm=\De_1+\ldots+\De_k$ such that the equivalence class of $\De_i$ is contained in $\mathcal{Seg}(\rho)$ for all $\ain{i}{1}{n}$. 
A multisegment $\fm$ contains a multisegment $\fm'$ if there exists a multisegment $\fm''$ such that $\fm$ is equivalent to $\fm'+\fm''$. Moreover, $\fm$ is called \emph{aperiodic} if it does not contain a multisegment of the form
\[[a,b]_\rho+[a+1,b+1]_\rho+\ldots+[a+e(\rho)-1,b+e(\rho)-1]_\rho.\]
We let $\Ms_{ap}$ be the set of aperiodic multisegments and $\Ms(\rho)_{ap}\coloneq \Ms_{ap}\cap \mathcal{MS}(\rho)$.
Let $\fm=[a_1,b_1]_{\rho_1}+\ldots+[a_k,b_k]_{\rho_k}$ be a multisegment. We set \[\fm^1\coloneq   [b_1]_{\rho_1}+\ldots+[b_k]_{\rho_k},\] \[\fm^-\coloneq   [a_1,b_1-1]_{\rho_1}+\ldots+[a_k,b_k-1]_{\rho_k}\] and define recursively $\fm^{-(s+1)}\coloneq   (\fm^{-s})^-$ and $\fm^{s+1}\coloneq   (\fm^{-s})^{1}$.
Let $l$ be the largest natural number such that $\fm^l\neq 0$ and define the partition ${\mu_{\fm}\coloneq   (\deg (\fm^1),\ldots,\deg(\fm^l))}$ of $\deg(\fm)$.

 We call $\fm=\De_1+\ldots+\De_k$ {unlinked} if $\De_i$ and $\De_j$ are pairwise unlinked for all $i\neq j,\, \ain{i,j}{1}{k}$.
As for representations, there exists a surjective map
\[\mathrm{cusp}_\Ms:\Ms_{ap}\rightarrow \mathbb{N}(\mathfrak{C}),\]
which maps $[a,b]_\rho\mapsto [\rho v_\rho^a]+\ldots+[\rho v_\rho ^b]$ and is extended linearly to multisegments. We call a multisegment \emph{banal} if its cuspidal support does not contain the cuspidal line $\ZZ[\rho]$.
\subsection{ Classification}\label{S:cusclas}We are now going to recall the constructions of \cite{M-S} and \cite{alma991023733359705251}. Let $\rho$ be a supercuspidal representation of $G_m$ and $n=e(\rho)\ell^r$ for some $r\in\mathbb{Z}_{\ge0}$. Then the representation
\[\rho\times \rho v_\rho\times\ldots\times \rho v_\rho^{n-1}\] contains a unique cuspidal subquotient, which is denoted by $\st(\rho,n)$ and ${o(\st(\rho,n))=1}$. Moreover, every cuspidal non-supercuspidal representation is of the above form and if $\st(\rho,n)\cong  \st(\rho',n')$ then $n=n'$ and $\mathbb{Z}[\rho]=\mathbb{Z}[\rho']$, see \cite[Section 6]{M-S}. In particular, every $\square$-irreducible cuspidal representation is supercuspidal. We let $\scu(G_m)$ be the set of isomorphism classes of $\square$-irreducible cuspidal representations of $G_m$ and $\scu\coloneq\bigcup_{m\ge 0}\scu(G_m).$
We write \[\mathrm{scusp}_{\Ms}\colon \Ms\rightarrow \mathbb{N}(\mathfrak{S})\] for the linear map sending a segment
\[[a,b]_{\st(\rho, n)}\mapsto \sum_{i=0}^{n-1}([\rho v_\rho^{a+i}]+\ldots +[\rho v_\rho^{b+i}]).\]
Fix a cuspidal representation $\rho$ of $G_m$. 
To each segment $[a,b]_\rho$ we are going to associate two irreducible representations $\Z([a,b]_\rho)$ and $\mathrm{L}([a,b]_\rho).$ This requires us to first consider $\He_R(n,q(\rho))$, the $R$-Hecke-algebra generated by $S_1,\ldots,S_{n-1}, X_1,\ldots,X_n$ satisfying the relations
\begin{enumerate}
    \item $(S_i+1)(S_i-q(\rho))=0,\, 1\le i\le n-1$
    \item $S_iS_j=S_jS_i,\, |i-j|\ge 2$
    \item $S_iS_{i+1}S_i=S_{i+1}S_iS_{i+1},\, 1\le i\le n-2$
    \item $X_iX_j=X_jX_i,\, 1\le i,j\le n$
    \item $X_jS_i=S_iX_j,\, i\notin \{j,j-1\}$
    \item $S_iX_iS_i=q(\rho)X_{i+1},\, 1\le i\le n-1$
\end{enumerate}
We let $\Irr(\Omega_{\rho,n})^*$ be the set of all isomorphism classes of 
irreducible representations $\pi$ of $G_{nm}$ whose cuspidal support is up to inertial equivalence contained in $\bbrh$. 
By \cite[§4.4]{MST} there exists a bijection \[\xi_{\rho,n}\colon\Irr(\Omega_{\rho,n})^*\iso \{\text{isomorphism classes of simple }\mathscr{H}(n,q(\rho))\text{-modules}\}.\]
Let $a,b\in \mathbb{Z}$ such that $b-a+1=n$. Then $\mathscr{H}_R(n,q(\rho))$ has two $1$-dimensional modules, $\mathscr{Z}(a,b)$, defined by 
\[S_i\mapsto q(\rho),\, X_j\mapsto q(\rho)^{a+j-1}\] and $\mathcal{L}(a,b)$, defined by \[S_i\mapsto -1,\, X_j\mapsto q(\rho)^{b-j+1}.\]
To a segment $[a,b]_\rho$ we can associate now an irreducible subrepresentation $\Z([a,b]_\rho)$ resp. irreducible quotient $\mathrm{L}([a,b]_\rho)$ of \[\rho v_\rho ^a\times\ldots\times \rho v_\rho ^b\] by demanding \[\xi_{\rho,n}(\Z([a,b]_\rho)=\mathscr{Z}(a,b)\text{ and }\xi_{\rho,n}(\mathrm{L}([a,b]_\rho)=\mathcal{L}(a,b).\]
For the empty segment, we let $Z([])$ be the trivial representation of the trivial group.
Note that equivalent segments give rise to isomorphic representations.
The representations $\Z([a,b]_\rho)$ and $\mathrm{L}([a,b]_\rho)$ behave very well under parabolic restriction.
\begin{lemma}[{\cite[Lemma 7.16]{M-S}}]\label{L:N1}
Let $\ain{k}{a}{b}$. Then
\[r_{((k-a)m,(b-k+1)m)}(\Z([a,b]_\rho))=\Z([a,k-1]_\rho)\otimes \Z([k,b]_\rho),\]
\[r_{((b-k+1)m,(k-a)m)}(\mathrm{L}([a,b]_\rho))=\mathrm{L}([k,b]_\rho)\otimes \mathrm{L}([a,k-1]_\rho).\]
\end{lemma}
The structure of representations induced from representations of the form $\Z([a,b]_\rho)$ or $\mathrm{L}([a,b]_\rho)$ is more complex and will preoccupy us throughout \Cref{S:Cons}. A starting point is the following theorem.
\begin{theorem}[{\cite[Theorem 7.26]{M-S}}]\label{T:N2}
Let $\fm=\De_1+\ldots+\De_k$. Then the following are equivalent.
\begin{enumerate}
    \item $\Z(\De_1)\times\ldots\times \Z(\De_k)$ is irreducible.
    \item $\mathrm{L}(\De_1)\times\ldots\times \mathrm{L}(\De_k)$ is irreducible.
    \item $\fm$ is unlinked.
\end{enumerate}
\end{theorem}
If some of the segments $\De_i$ are linked, the induced representation is no longer irreducible. The following two lemmata already hint how these induced representations might look like.
\begin{lemma}[{\cite[Proposition 7.17]{M-S}}]\label{L:quot}
If $\rho\in \scu$ then $\Z([a,b]_\rho)$ is the socle of
\[\Z([a,b-1])\times \rho v_\rho ^b\text{ and } \rho v_\rho^{a}\times \Z([a+1,b])\]
and the cosocle of 
\[\Z([a+1,b])\times \rho v_\rho ^a\text{ and } \rho v_\rho^{b}\times \Z([a,b-1]).\]
\end{lemma}
Next, we recall the definition of a residually non-degenerate representation. We will use an alternative definition from the one in \cite{M-S}, however it was proven in Theorem 9.10 of said paper that the following definition is equivalent to theirs.
An irreducible representation $\pi$ is called \emph{residually non-degenerate} if there exists a multisegment $\fm=\De_1+\ldots+\De_k,\,\De_i=[a_i,b_i]_{\rho_i}$ for $\ain{i}{1}{k}$ such that\[[\pi]=[\mathrm{L}(\De_1)\times\ldots\times \mathrm{L}(\De_k)]\] and $l(\De_i)<e(\rho_i)$ for $\ain{i}{1}{k}$. More generally, for $\alpha=(\alpha_1,\ldots.,\alpha_t)$ a composition, we call a representation of the form $\pi=\pi_1\otimes\ldots\otimes\pi_t$ of $G_\alpha$ residually non-degenerate if each of the $\pi_i$'s is residually non-degenerate.
\begin{lemma}[{\cite[Corollary 8.5]{M-S}}]\label{C:N1}
If $\pi$ is residually non-degenerate and \[[\pi]\le [\Z(\De_1)\times\ldots\times \Z(\De_k)]\]
then $l(\De_i)=1$ for all $\ain{i}{1}{k}$.
\end{lemma}
Let $\rho$ be a cuspidal representation of $G_m$, $\alpha$ a composition of $nm$ and $\pi$ an irreducible representation of $G_{nm}$. We say $\pi$ is \emph{$\alpha$-degenerate} if  $r_\alpha(\pi)$ contains a residually non-degenerate representation. It follows from the definition given in \cite[§8]{M-S} that if $\alpha'$ is a permutation of $\alpha$ then $\pi$ is $\alpha$-degenerate if and only if it is $\alpha'$-degenerate.
For $\fm=\De_1+\ldots+\De_n$ we define
\[\I(\fm)=[\Z(\De_1)\times\ldots\times \Z(\De_k)],\]
which is well-defined since $\times$ is commutative in the Grothendieck group.
\begin{theorem}[{\cite[Section 8]{M-S}}]\label{T:N4}
Let $\fm$ be a multisegment of the form $\fm=[a_1]_{\rho_1}+\ldots+[a_k]_{\rho_k}$. Then $\I(\fm)$ contains a unique residually non-degenerate representation denoted $\Z(\fm)$ and any cuspidal representation is residually non-degenerate.
\end{theorem}
For $\fm$ a multisegment and $l$ maximal such that $\fm^l\neq 0$ we set \[\st(\fm)\coloneq    \Z(\fm^l)\otimes\ldots\otimes \Z(\fm^1).\] We can now define an irreducible representation $\Z(\fm)$ for a general multisegment $\fm$.
\begin{theorem}[{\cite[Theorem 9.19]{M-S}}]\label{T:N5}
Let $\fm$ be a multisegment. Then $\I(\fm)$ contains a unique irreducible subquotient denoted by  $\Z(\fm)$ which is $\mu_\fm$-degenerate. The multiplicity of $\Z(\fm)$ in $\I(\fm)$ is $1$ and the unique irreducible constituent of $r_{\overline{\mu_\fm}}(\Z(\fm))$ which is residually non-degenerate is $\st(\fm)$. 

If $\alpha=(\alpha_1,\ldots,\alpha_k)$ is an ordered composition of $\deg(\fm)$ and $\pi$ is an irreducible $\alpha$-degenerate subquotient of $\I(\fm)$ then $\alpha\le \mu_\fm$ with equality if and only if $\pi=\Z(\fm)$. 
\end{theorem}
The theorem gives rise to a map \[\Z:\mathcal{MS}\rightarrow \Irr.\]
\begin{prop}[{\cite[Proposition 9.28]{M-S}}]\label{L:sam}
    Let $\fm$ be a multisegment and write $\fm=\fm_1+\ldots +\fm_k$ as a sum of multisegments.
    Then
    \[[\Z(\fm)]\le [\Z(\fm_1)\times\ldots\times \Z(\fm_k)]\] and $\Z(\fm)$ appears with multiplicity $1$.
\end{prop}
\begin{theorem}[{\cite[Theorem 9.36]{M-S}}]\label{T:Zinj}
The map $\Z:\Ms\rightarrow \Irr$ restricted to aperiodic multisegments is injective and $\Z(\fm)^\lor\cong \Z(\fm^\lor)$.
Moreover, if $\fm$ is an multisegment
\[\scus_\Ms(\fm)=\scus(\Z(\fm))\] and if $\fm$ is moreover aperiodic, then
\[\cus_{\Ms}(\fm)=\cus(\Z(\fm)).\]
\begin{lemma}\label{L:S33}
Let $\fm$ be a multisegment.
Then \[[r_{(\deg(\fm^-),\deg(\fm^1))}(\Z(\fm))]\ge [\Z(\fm^-)\otimes \Z(\fm^1)].\]
\end{lemma}
\begin{proof}
By \Cref{T:N5}
\[[r_{\overline{\mu_\fm}}(\Z(\fm))]\ge [\st(\fm)].\]
Let $\pi\in\Irr_{\deg(\fm^-)}$ such that 
\[[r_{\overline{\mu_{\fm^-}}}(\pi)]\ge [\st(\fm^-)],\]
\[[r_{(\deg(\fm^-),\deg(\fm^1)}(\I(\fm))]\ge [r_{(\deg(\fm^-),\deg(\fm^1))}(\Z(\fm))]\ge [\pi\otimes \Z(\fm^1)].\]
The first inequality implies that $\pi$ is $\mu_{\fm^-}$-degenerate. On the other hand, the second inequality implies, using \Cref{C:N1} and the Geometric Lemma, that $[\pi]\le \I(\fm^-)$. 
\Cref{T:N5} forces therefore $\pi\cong \Z(\fm^-)$.
\end{proof}
\end{theorem}
In the same article it was proven that $\Z$ is also surjective, using by passing to the Hecke-algebra and invoking a counting argument. We will reprove in \Cref{S:appd} this surjectivity for all representations with $\square$-irreducible cuspidal support.
\subsection{ Symmetric group}\label{S:repsym}
In this section we quickly review the representation theory of $\mathrm{S}_n$ over $R$, see for example \cite{Jamesbook}.
Let $\lambda=(\lambda_1,\ldots,\lambda_t)$ be a partition of $n$ and denote by $\textbf{1}_{n}$ the trivial representation of $\mathrm{S}_n$. We set \[\mathrm{S}_\lambda\coloneq \mathrm{S}_{\lambda_1}\times\ldots\times \mathrm{S}_{\lambda_t},\] \[\textbf{1}_\lambda\coloneq  \textbf{1}_{\lambda_1}\otimes\ldots\otimes \textbf{1}_{\lambda_t}\]
and  \[M^\lambda\coloneq \mathrm{Ind}_{\mathrm{S}_\lambda}^{\mathrm{S}_n}(\textbf{1}_\lambda)= \mathrm{Hom}_{\mathrm{S}_\lambda}(\mathrm{S}_n,\textbf{1}_\lambda),\]
called the \emph{Young permutation module}.
The representation $M^\lambda$ contains the \emph{Specht-module} $S^\lambda$ with multiplicity $1$, \emph{cf.} \cite[§4]{Jamesbook}. If $R=\ql$, $S^\lambda$ is irreducible for all $\lambda$ and the Specht-modules parameterize all irreducible representations of $\mathrm{S}_n$, however if $R=\fl$, $S^\lambda$ might be reducible. To bypass this issue in the case $R=\fl$, one equips $M^\lambda$ with a natural $\mathrm{S}_n$-invariant, non-degenerate and symmetric bilinear pairing $\langle\cdot,\cdot\rangle$. Under this pairing, let $(S^\lambda)^\bot$ be the orthogonal complement of $S^\lambda$ in $M^\lambda$ and set 
\[D_\lambda\coloneq S^\lambda/(S^\lambda\cap (S^\lambda)^\bot).\]
We recall that $\lambda$ is $\ell$-regular if for all
\[a_i\coloneq\#\{\ain{j}{1}{t}:\lambda_j=i\},\]
$a_i<\ell$.
Now $D_\lambda$ is irreducible if it is non-zero and it is non-zero if and only if $\lambda$ is $\ell$-regular. Moreover, if $\mu$ is a second partition of $n$, $D_\lambda$ appears in $M^{\mu}$ if and only if $\lambda\le \mu$. Finally $M^\mu$ is contained in $M^\lambda$ in the Grothendieck group of $R$-representations of $\mathrm{S}_n$ if and only if $\mu\le \lambda$. More precisely, the multiplicity of $S^\lambda$ in $M^\mu$ is the Kostka number $K_{\lambda,\mu}$.
\subsection{ Integral structures}\label{S:intintro}
Let $\zl$ be the ring of integers of $\ql$ with residue field $\fl$. We choose the square roots in \Cref{S:parabind} in a compatible way, \emph{i.e.} if $R=\ql$, we let $\sqrt{q}\in \zl$, such that its reduction modulo the maximal ideal in $\zl$ gives the chosen square root in $\fl$.

A $\ql$-representation $(\pi,V)$ of $G_n$ is called \emph{integral} if it is admissible and admits an integral structure, \emph{i.e.} a $G_n$-stable $\zl$-submodule $\mo$ of $V$ such that the natural map $\mo\otimes_{\zl} \ql\rightarrow V$ is an isomorphism. 
If $P=M\rtimes U$ is a parabolic subgroup of $G_n$ and $\sigma$ a representation of $M$ with integral structure $\mo$, then $\mathrm{Ind}_P^{G_n}(\mo)$ is an integral structure of $\mathrm{Ind}_P^{G_n}(\sigma)$, see \cite[II.4.14]{alma991023733359705251}.
For a representation $\pi$ with integral structure $\mathfrak{o}$ one can consider the reduction mod $\ell$ of $\pi$, which is defined as
\[\rl(\pi)\coloneq \mo\otimes_{\zl}  \fl.\]
This is not an invariant of $\pi$, it depends on the chosen integral structure. However, its image in the Grothendieck group is. This leads to the following theorem.
Let $\mathfrak{R}_n^{en}(\ql)$ be the subgroup of $\mathfrak{R}(\ql)$ generated by irreducible integral representations.
\begin{theorem}[{\cite[II.4.12]{alma991023733359705251},\cite[Theorem A]{MST}, \cite[Theorem 9.39]{M-S}, \cite[II.4.14, II.5.11]{alma991023733359705251}}]\label{T:rel}
The morphism \emph{reduction mod }$\ell$ of groups
\[\rl:\mathfrak{R}^{en}_n(\ql)\rightarrow \mathfrak{R}_n(\fl)\] is well defined. Moreover, $\rl$ satisfies the following properties.
\begin{enumerate}
    \item If $\trho$ is a cuspidal representation over $\ql$, it admits an integral structure if and only if its central character is integral.
    \item If $\rho$ is a supercuspidal representation over $\fl$ then there exists an integral cuspidal representation $\Tilde{\rho}$ of $G_m$ over $\ql$ such that $\rl(\Tilde{\rho})=\rho$. We refer to such $\trho$ as a lift of $\rho$.
    \item If $\Tilde{\rho},\,\rho$ are cuspidal such that $\rl(\Tilde{\rho})=\rho$ then for integers $a\le b$, $\Z([a,b]_{\Tilde{\rho}})$ admits an integral structure and
    \[\rl(\Z([a,b]_{\Tilde{\rho}})= \Z([a,b]_\rho).\]
    \item $\rl$ commutes with parabolic induction, \emph{i.e.} there is a natural isomorphism\[\mathrm{Ind}_P(\mo)\otimes_{\zl}  \fl\rightarrow \mathrm{Ind}_P(\mo\otimes_{\zl}  \ol{\mathbb{F}}_l).\]
\end{enumerate}
\end{theorem}
A lift of a segment $[a,b]_\rho$ over $\fl$ to $\ql$ is a segment $[a',b']_\trho$ such that $\trho$ is a lift of $\rho$ and $a'\le b'\in \ZZ$ with $a-b=a'-b'$ and $a=a'\mod o(\rho)$. A multisegment $\tfm=\TD_1+\ldots+\TD_k$ over $\ql$ is a lift of a mulitsegment $\fm=\De_1+\ldots+\De_k$ over $\fl$ if $\TD_i$ is a lift of $\De_i$ for $\ain{i}{1}{k}$. In this case we write\[\rl(\tfm)=\fm.\]
\begin{theorem}[{\cite[Theorem 9.39]{M-S}}]\label{T:liftgood}
    Let $\fm$ be a multisegment and $\tfm$ a lift of $\fm$. Then $\Z(\tfm)$ admits an integral structure and $\rl(\Z(\tfm))$ contains $\Z(\fm)$ with multiplicity $1$.
\end{theorem}
\subsection{ Haar-measures}
If one tries to translate the computation of local $L$-factors from $\ql$ to $\fl$, one quickly runs into the problem that, contrary to the situation in $\ql$, for a compactly supported smooth function $f\colon G\ra R$ the following identity does not necessarily have to hold:
\[\int\displaylimits_{G_n}f(g)\ddd g=\int\displaylimits_P\int\displaylimits_{K_n}f(pk)\ddd k\dd p,\]
where $P$ is a parabolic subgroup of $G$, $K_n$ is the open compact subgroup $K_n\coloneq \mathrm{GL}_n(
\mathfrak{o}_\mathrm{D})$ and$\ddd g,\ddd k,\dd p$ are left Haar-measures on the respective groups, see for example in \cite{Mzeta}.
However one can save the situation via the following method, \emph{cf.} \cite[§2.2]{KM}. 
For $H$ a closed subgroup of $G_n$ let $C_c^\infty(H,R)$ be the set of smooth, compactly supported functions on $H$ taking values in $R$. If $H'\subseteq H$ is a closed subgroup, we let $C_c^\infty(H'\bs H)= C_c^\infty(H'\bs H,\delta,R)$ be the space of smooth functions on $H$ which transform by $\delta\coloneq \delta_{H}^{-1}\delta_{H'}$ under left translation by $H'$.
We let $C_{c,en}^\infty(H'\bs H,\delta,\ql)$ be the functions with values in $\zl$ and we denote reduction mod $\ell$ by
\[\rl\colon C_{c,en}^\infty(H'\bs H,\delta,\ql) \ra C_{c}^\infty(H'\bs H,\delta,\fl).\]
We will denote from now on by $\dd h$ a left Haar-measure on $H$ and by $\dr h$ a right Haar-measure on $H$, and we choose measures compatible with $\rl$, i.e. for $f\in C_{c,en}^\infty(H'\bs H,\delta,\ql)$, 
\[\int\displaylimits_{H'\bs H} f(h)\dr h\in \zl\]
and \[\rl\left(\int\displaylimits_{H'\bs H} f(h)\dr h\right)=\int\displaylimits_{H'\bs H} \rl(f(h))\dr h\]
and similarly for $\dd h$.
In \cite[§2.2]{KM} the authors show that if one replaces the left Haar-measure $\dd p$ with a right Haar-measure $\dr p$ on $P$ and $(K_n\cap P )\bs K_n$, above identity holds also in $\fl$:
\begin{equation}\label{E:int2}\int\displaylimits_{G_n}f(g)\ddd g=\int\displaylimits_P\int\displaylimits_{(K_n\cap P )\bs K_n}f(pk)\delta_P^{-1}(p)\ddd k\dr p.\end{equation}
Moreover for $H$ a closed subgroup of $G_n$ and $f\in C^\infty_c(H\bs G_n)$ we can find $F\in C^\infty_c(G_n)$ such that
\[f(g)=\int\displaylimits_HF(gh)\delta^{-1}(h)\dr h\] and write in this situation $f=F^H$. As a consequence we have for $f\in C^\infty_c(P\bs G_n)$ and $P$ a parabolic subgroup of $G_n$
\begin{equation}\label{E:int}\int\displaylimits_{P\bs G_n}F^P(g)\ddd g=\int\displaylimits_{(K_n\cap P )\bs K_n}\int\displaylimits_PF(pk)\delta_P^{-1}(p)\dr p\ddd k=\int\displaylimits_{(K_n\cap P )\bs K_n}F^P(k)\ddd k.\end{equation}
Finally, we recall the following facts about smooth functions on $G_n$. Let $C^\infty(G_n)$ denote the representation of smooth $R$-valued functions on which $G_n\times G_n$ acts by left-right translation. Its dual, $C_c^\infty(G_n)$ is given by the compactly supported smooth functions often referred to as \emph{Schwartz-functions} and such a Schwartz-function $\phi$ acts on $f\in C^\infty(G_n)$ by \[\phi(f)\coloneq \int\displaylimits_{G_n}f(g)\phi(g)\ddd g,\] where $\phi\in C_c^\infty(G_n)$. Finally, if $P$ is a parabolic subgroup of $G_n$ with opposite parabolic subgroup $\op$, let
$C_c^\infty(\op P)$ be the set of Schwartz-functions on $\op P$ on which $\op \times P$ acts by left-right translation. It is the dual of $C^\infty(\op P)$, given by just smooth functions on $\op P$.
\subsection{ Godement-Jacquet local $L$-factors}
In this section we will recall the Godement-Jacquet $L$-functions for general irreducible representations, \emph{cf.} \cite{godement1972zeta} and \cite{Mzeta}. 
Let us quickly review their construction. Let $(\pi,V)$ be a representation over $R$ of $G_n$ and let for $v\in V,v^\lor\in V^\lor$ \[f\coloneq g\mapsto v^\lor( \pi(g)v)\] be a matrix coefficient of $\pi$. In general we say that a smooth function $f$ is a matrix coefficient of $\pi$ if it is the finite sum of functions of the above form. Let $\schr$ be the space of Schwartz-functions on $M_n(\mathrm{D})$, \emph{i.e.} the space locally constant functions with compact support $f\colon M_n(\mathrm{D})\rightarrow R$.
For $H\subseteq G_n$, we write for $N\in \ZZ$, \[H(N)\coloneq \{h\in H:\lvert\det\nolimits'( h)\lvert =q^{-N}\}.\]
For a matrix coefficient $f$ we denote by \[f^\lor\coloneq g\mapsto f(g^{-1})\] and for $\phi\in\schr$ we let $\widehat{\phi}$ be its Fourier transform with respect to our fixed additive character $\psi$, see \cite[Section 1.3]{Mzeta}.
For $f$ a matrix coefficient of $\pi$ and $\phi\in\schr$ one can then construct a formal Laurent series $Z(\phi,T,f)\in R((T))$, which is linear in $f$ and $\phi$.
The coefficient of $T^N$ of $Z(\phi,T,f)$
is given by the integral
\[\int\displaylimits_{G_n(N)}\phi(g)f(g)\ddd g.\]
\begin{theorem}[{\cite[Theorem 2.3]{Mzeta}}]\label{T:Lfon}
     Let $(\pi,V)$ be a representation of $G_n$ which is a subquotient of a representation induced from irreducible representations.
    \begin{enumerate}
        \item There exists a non-zero $P'(\pi,T)\in R[T]$ such that for each matrix coefficient $f$ and $\phi\in\schr$
        \[P(\pi,T)Z(\phi,T,f)\in  R[T,T^{-1}].\]
       \item There exists $\gamma(T,\psi,f)\in R(T)$ such that for each matrix coefficient $f$ and $\phi\in\schr$
        \[Z(\widehat{\phi},T^{-1}q^{-\frac{{dn+1}}{2}},f^\lor)=\gamma(T,\pi, \psi)Z(\phi,Tq^{-\frac{{dn-1}}{2}},f).\]
    \end{enumerate}
\end{theorem}
Note that the statement of this theorem in \cite{Mzeta} is only stated for irreducible representations, however the proof they give, together with \cite[Proposition 2.5]{Mzeta} show that it is true for all the above representations.
\begin{corollary}[{\cite[Corollary 2.4]{Mzeta}}]\label{C:fracid}
    Let $\mathcal{L}(\pi)$ be the $R$-subspace of $R((T))$ generated by
    $Z(\phi,q^{-\frac{dn-1}{2}}T,f)$ where $f$ ranges over all matrix coefficients of $\pi$ and $\phi\in\schr$ over all Schwartz-functions. Then $\mathcal{L}(\pi)$ is a fractional ideal of $R[T,T^{-1}]$ containing the constant functions. It admits a generator 
    \[\frac{1}{P(\pi,T)},\, P(\pi,T)\in R[T]\] which can be normalized by the condition $P(\pi,0)=1$.
\end{corollary}
One then defines
\[L(\pi,T)\coloneq \frac{1}{ P(\pi,T)}\]
and $\epsilon(T,\pi, \psi)$ by 
\[\gamma(T,\pi, \psi)= \epsilon(T,\pi, \psi){\frac{L(\pi^\lor,q^{-1}T^{-1})}{ L(\pi,T)}}.\]
Then \[\epsilon(T,\pi, \psi)\epsilon(q^{-1}T^{-1},\pi^\lor , \psi)=\omega(-1),\]
where $\omega$ is the central character of $\pi$ and $\epsilon(T,\pi, \psi)=C_{\pi,\psi}T^{k(\pi,\psi)}$ for some constant $C_{\pi,\psi}\in R$ and $k(\pi,\psi)\in \ZZ$. 
If $R=\ql$ and $\pi$ is an entire representation then $P(\pi,T)$ has integer coefficients and hence so do the rational functions $\gamma(T,\pi, \psi)$ and $\epsilon(T,\pi, \psi)$. We write for a rational function $F\in \zl[T,T^{-1}]$, $\rl(F)\in \fl[T,T^{-1}]$ for the coefficient-wise reduction $\mod \ell$ of $F$ and $\rl\left(\frac{1}{ F}\right)\coloneq \frac{1}{ \rl(F)}$ if $\rl(F)\neq 0$.
If $R=\C$, setting $T=q^{-s}$ gives the usual Godement-Jacquet $L$-factors of \cite{godement1972zeta}.
\begin{lemma}[{\cite[Corollary 4.2]{Mzeta}}]\label{L:lin}
    Let $\widetilde{\pi}$ be an entire admissible representation over $\ql$ and $\mo$ an integral structure of $\widetilde{\pi}$. Assume that both $\widetilde{\pi}$ and $\mo\otimes_\zl\fl$ satisfy the condition of \Cref{T:Lfon}. Then $\rl(\gamma(T,\widetilde{\pi}, \widetilde{\psi}))$ is well defined, the polynomial $P(\mo\otimes_\zl\fl,T)$ divides $\rl(P(\widetilde{\pi},T))$ in $\fl[T]$ and
         \[\gamma(T,\mo\otimes_\zl\fl, \psi)=\rl(\gamma(T,\widetilde{\pi}, \widetilde{\psi}))= \rl(\epsilon(T,\widetilde{\pi}, \widetilde{\psi})){\frac{\rl(L(\widetilde{\pi}^\lor,q^{-1}T^{-1}))}{ \rl((L(\widetilde{\pi},T))}}.\]
\end{lemma}
Another useful lemma is the following:
\begin{lemma}[{\cite[Proposition 2.5]{Mzeta}}]\label{L:su}
    Let $\pi$ be a representation of $G_n$ satisfying the conditions of \Cref{T:Lfon}. Moreover, let $\tau$ be a subquotient of $\pi$. Then $P(\tau,T)$ divides $P(\pi, T)$ and \[\gamma(T,\pi,\psi)=\gamma(T,\tau,\psi).\] 
\end{lemma}
\section{Intertwining operators and square-irreducible representations}
We will quickly discuss the intertwining operators of \cite{Dat} and square-irreducible representations, \emph{cf.} \cite{LMsquare}.
\subsection{ Intertwining operators}\label{S:intertwiningoperators}
We start with the intertwining operators in arbitrary characteristic introduced in \cite[§7]{Dat} and state some lemmas, which are probably already well-known to the experts.
Fix for the rest of the section a parabolic subgroup $P=P_{(n_1,n_2)}=M\rtimes U$ with Levi-component $M=G_{n_1}\times G_{n_2}$ of $G=G_n$ and $\sigma$ a smooth representation of $M$ of finite length. We write $ R(T)$ and denote by $\sigma_{R(T)}=\sigma\otimes_RR(T)$ the basechange of $\sigma$ to ${R(T)}$. Moreover, we equip ${R(T)}$ with a valuation $v\colon {R(T)}\ra \mathbb{R}$ which restricts to $0$ on $R$ and satisfies $v(T)<0$. Finally, we denote by $\psi_{un}:M\ra {R(T)}$ the generic unramified character sending \[(m_1,m_2)\mapsto T^{-\log_q(\lvert \det\nolimits'(m_1)\lvert )}.\] We write $P'=P_{(n_2,n_1)}=M'\rtimes U'$ with Levi-component $M'=G_{n_2}\times G_{n_1}$ and we write $\overline{\sigma}$ for the representation obtained by twisting $\sigma$ with the inner automorphism given by the element
\begin{equation}\label{E:w}w\coloneq \begin{pmatrix}
    0&1_{n_1}\\
    1_{n_2}&0
\end{pmatrix}.\end{equation}
For example, if $\sigma=\sigma_1\otimes \sigma_2$, we have $\overline{\sigma_1\otimes\sigma_2}=\sigma_2\otimes \sigma_1$.
We define the $R(T)$-representation \[\sigma_{un}\coloneq \sigma_{R(T)}\otimes\psi_{un}.\] 
The representation $\sigma_{un}$ is admissible. Indeed, 
for any open compact subgroup $K\subseteq G$, we have $\lvert\det\nolimits'(K)\lvert=1$, and hence the character $\psi_{un}$ does not play a role when considering $K$-fixed vectors. The claim follows then from \cite[Lemma III.1(ii)]{Ghen}.
By \cite[Lemma 3.7]{Dat}
\[\Ho_{G_n}(\ind_P^G(\sigma_{un}),\ind_{P'}^G(\overline{\sigma}_{un}))\] is non-zero and contains a certain class of non-zero morphisms $J_P(\sigma)$'s in it, which
are up to a choice of a Haar-measure uniquely determined by the following condition. If $f\in \ind_P^G(\sigma_{un})$ is compactly supported in $PwP'$, \emph{cf}. (\ref{E:w}), then
\[J_P(\sigma)(f)(1)=\int\displaylimits_{{U'}}f(w{u'}1)\ddd u'.\]
For $k\in\ZZ$ we denote by $\mathfrak{l}_k\coloneq \sigma\otimes_R (T-1)^kR[T]_{(T-1)}\otimes_{R[T]_{(T-1)}}\psi_{un}$.
\begin{lemma}
Let $\sigma\in \Rep(M)$.
Then there exists a morphism
\[J_P(\sigma)\colon\ind_P^G(\sigma_{un})\ra\ind_{P'}^G(\overline{\sigma}_{un}),\]which sends  $\ind_P^G(\mathfrak{l}_0)$ to $\ind_{P'}^G(\overline{\mathfrak{l}_0})$ and defines a non-zero morphism $R_\sigma$
\[
\begin{tikzcd}
    \ind_P^G(\mathfrak{l}_0)\arrow[r," J_P(\sigma)"]\arrow[d,"T\mapsto 1",twoheadrightarrow]&\ind_{P'}^G(\overline{\mathfrak{l}_0})\arrow[d,"T\mapsto 1",twoheadrightarrow]\\ 
\ind_P^G(\sigma)\arrow[r,"R_\sigma"]&\ind_{P'}^G(\overline{\sigma})
\end{tikzcd}
\]
Such $R_\sigma$ is called an \emph{intertwining operator}.
\end{lemma}
\begin{proof}
Recall that the Geometric Lemma gives inclusions
\[\begin{tikzcd}
    \overline{\sigma_{un}}\arrow[r,hookrightarrow]& r_{P'}(\ind_P^G({\sigma_{un}}))\\
    \overline{\mathfrak{l}_0}\arrow[r,hookrightarrow]\arrow[u,hookrightarrow]& r_{P'}(\ind_P^G({\mathfrak{l}_0}))\arrow[u,hookrightarrow]
\end{tikzcd}\]
which are determined by a suitable choice of Haar-measures, \emph{i.e.} up to a scalar in $R[T]_{(T-1)}$.
We will use the notation of \cite[§2]{Dat} and note that by \cite[Lemma 3.7]{Dat} $\sigma_{un}$ is $(\op,P)$-regular.
Following \cite[Lemma 2.10]{Dat}, we can therefore define a retraction $r$ of the inclusion
\[\begin{tikzcd}
  \overline{\sigma_{un}}\arrow[r,hookrightarrow]\arrow[rr,bend right=30, "1"]&  r_{P'}(\ind_P^G(\sigma_{un}))\arrow[r,"r"]& \overline{\sigma_{un}},
\end{tikzcd}\]
Moreover, this retraction defines by Frobenius reciprocity the intertwiner $J_P(\sigma)$ up to scalar multiplication in $R(T)$.

The lattice $r_{P'}(\ind_P^G(\mathfrak{l}_0))$ is an integral structure of $r_{P'}(\ind_P^G(\sigma_{un}))$ in the sense of \cite[I.9]{alma991023733359705251} by \cite[Proposition 6.7]{Dat}.
 The image of $r_{P'}(\ind_P^G(\mathfrak{l}_0))$ under $r$ is an $R[T]_{(T-1)}$-integral structure $\mathfrak{o}$ of $\overline{\sigma_{un}}$ by \cite[I.9.3]{alma991023733359705251}. 
 We can now fix the maximal $k\in\mathbb{Z}$ such that $\mathfrak{o}\subseteq \mathfrak{l}_k$. The existence of such $k$ follows from \cite[Remark p.80]{alma991023733359705251}, in which it is shown that any two integral structures of an admissible representation are commensurable.
We thus obtain the following diagram.
\[\begin{tikzcd}
  \overline{\mathfrak{l}_0}\arrow[r,hookrightarrow]\arrow[rr,bend right=30, "\iota"]&  r_{P'}(\ind_P^G(\mathfrak{l}_{0}))\arrow[r,"r"]& \overline{\mathfrak{l}_{k}},
\end{tikzcd}\]
where $\iota\colon \mathfrak{l}_0\hra \mathfrak{l}_{k}$ is the natural inclusion and in particular $k\le 0$.
We now multiply $r$ by $(T-1)^{-k}$ and hence the image of $r'\coloneq r(T-1)^{-k}$ is contained in $\overline{\mathfrak{l}_{0}}$ but not in $\overline{\mathfrak{l}_{1}}$.
We thus have a diagram
\[\begin{tikzcd}
  \overline{\mathfrak{l}_0}\arrow[r,hookrightarrow]\arrow[rr,bend right=30, "(T-1)^{-k}\cdot"]&  r_{P'}(\ind_P^G(\mathfrak{l}_{0}))\arrow[r,"r'"]& \overline{\mathfrak{l}_{0}}
\end{tikzcd}\]
and we fix $J_P(\sigma)$ to the intertwiner coming from Frobenius reciprocity applied to $r'$. Moreover, we set $r(\sigma)\coloneq -k\ge 0$, which is well defined. Indeed, the only possible choice we made is choosing the inclusion \[ \overline{\sigma_{un}}\hra r_{P'}(\ind_P^G(\sigma_{un}))\] up to a scalar in $R[T]_{(T-1)}$, which does not affect $k$.
Next observe that the sequence
\[\begin{tikzcd} 0\arrow[r]&\ind_P^G(\mathfrak{l}_1)\arrow[r]&\ind_P^G(\mathfrak{l}_0)\arrow[r,"T\mapsto 1"]&\ind_P^G(\sigma)\arrow[r]&0
\end{tikzcd} \] is exact since the representation in the sequence are admissible.
We thus have the following commutative diagram.
\[
\begin{tikzcd}
\ind_P^G(\sigma_{un})\arrow[r,"J_P(\sigma)"]&\ind_{P'}^G(\overline{\sigma}_{un})\\
    \ind_P^G(\mathfrak{l}_0)\arrow[r," J_P(\sigma)"]\arrow[d,"T\mapsto 1",twoheadrightarrow]\arrow[u,hookrightarrow]&\ind_{P'}^G(\overline{\mathfrak{l}_0})\arrow[d,"T\mapsto 1",twoheadrightarrow]\arrow[u,hookrightarrow]\\ 
\ind_P^G(\sigma)\arrow[r,"R_\sigma"]&\ind_{P'}^G(\overline{\sigma})
\end{tikzcd}
\]
Now if $R_\sigma$ vanishes, $J_P(\sigma)( \ind_P^G(\mathfrak{l}_0))\subseteq \ind_{P'}^G(\overline{\mathfrak{l}_1})$, however applying Frobenius reciprocity would give a contradiction to the construction of $J_P(\sigma)$.
\end{proof}
 Note that $J_P(\sigma)$ in the above lemma is unique up to a unit in $R[T]_{(T-1)}$ and $R_\sigma$ is unique up to a unit in $R$. 
 We call $r(\sigma)$ the order of the pole of the intertwining operator at $\sigma$. Note that $r(\sigma)=0$ if and only if for all $f\in \ind_P^G(\sigma)$ with support contained in $PwP'$,
 \[R_\sigma(f)(g)=\int_{U'}f(wu')\ddd u'.\]
 If $\sigma=\pi_1\otimes\pi_2$, we will also denote by
\[R_{\pi_1,\pi_2}\colon \pi_1\times\pi_2\ra \pi_2\times \pi_1\] the so obtained morphism and $r_{\pi_1,\pi_2}$ the order of vanishing.
The next lemma is the generalization of the list of properties needed to deal with square irreducible representations and collected for $R=\C$ in \cite[Lemma 2.3]{LMsquare}.
\begin{lemma}\label{L:inop}
    Let $0\ra \sigma_1\hra\sigma\sra\sigma_2\ra 0$ be a short exact sequence of representations of $M$  and $\pi_1,\, \pi_2,\, \pi_3$ representations of $G_{k_i},\, i\in \{1,2,3\}$.
    \begin{enumerate}
        \item If $M=G$, $R_\sigma$ is a scalar.
        \item The order of the poles satisfies $r(\sigma)\ge r(\sigma_1)$ and \[\restr{J_P(\sigma)}{\sigma_1}=(T-1)^{r(\sigma)-r(\sigma_1)}J_P(\sigma_1).\]
        \item $R_\sigma$ restricts either to an intertwining operator or to $0$, \emph{i.e.} up to a scalar \[\restr{R_\sigma}{\sigma_1}=\begin{cases}
            R_{\sigma_1}&\text{ if }r(\sigma)=r(\sigma_1),\\
            0&\text{ otherwise.}
        \end{cases}\]
        \item If $\sigma=\pi_1\times\pi_2\otimes\pi_3$, $r_{\pi_1\times\pi_2,\pi_3}\le r_{\pi_1,\pi_2}+r_{\pi_2,\pi_3}$
        and  \[(T-1)^{ r_{\pi_1,\pi_2}+r_{\pi_2,\pi_3}-r_{\pi_1\times\pi_2,\pi_3}}J_P(\sigma)=(J_P(\pi_1\otimes\pi_3)\times 1_{\pi_2})\circ (1_{\pi_1}\times J_P(\pi_2\otimes \pi_3)).\] Therefore, \[(R_{\pi_1,\pi_3}\times 1_{\pi_2})\circ (1_{\pi_1}\times R_{\pi_2,\pi_3})=\begin{cases}
            R_{\pi_1\times\pi_2,\pi_3}&\text{ if }r_{\pi_1\times\pi_2,\pi_3}= r_{\pi_1,\pi_2}+r_{\pi_2,\pi_3},\\
            0&\text{ otherwise,}
        \end{cases}\]
         up to a scalar.
        \item Moreover, $r_{\pi_1\times\pi_2,\pi_3}= r_{\pi_1,\pi_2}+r_{\pi_2,\pi_3}$ if at least one of the $R_{\pi_i,\pi_3}$, $i\in \{1,2\}$ is an isomorphism or if $\pi_3$ is irreducible.
        \item If $R_\sigma$ is an isomorphism, then $ R_{\overline{\sigma}}\circ R_\sigma$ is a scalar.
    \end{enumerate}
\end{lemma}
\begin{proof}
The first point is trivially true.
    Since the following diagram commutes up to a non-zero scalar in ${R(T)}$, (2) and (3) also follow immediately.
    \[\begin{tikzcd}
    \ind_P^G({\sigma_2}_{un})\arrow[r, "J_P(\sigma_2)"]&\ind_{P'}^G(\overline{\sigma_2}_{un})\\
    \ind_P^G(\sigma_{un})\arrow[u,twoheadrightarrow]\arrow[r, "J_P(\sigma)"]&\ind_{P'}^G(\overline{\sigma}_{un})\arrow[u,twoheadrightarrow]\\
\ind_P^G({\sigma_1}_{un})\arrow[u,hookrightarrow]\arrow[r,"J_P(\sigma_1)"]&\ind_{P'}^G(\overline{\sigma_1}_{un})\arrow[u,hookrightarrow]
    \end{tikzcd}\]
The first part of (4) follows from the following, up to multiplication by some non-zero element of ${R(T)}$, commutative diagram.
\[\begin{tikzcd}
\ind_P^G((\pi_1\times\pi_2\otimes\pi_3)_{un})\arrow[rr,"J_P(\pi_1\times\pi_2\otimes\pi_3)"]\arrow[d,"1_{\pi_1}\times J_P(\pi_2\otimes\pi_3)" left]&&\ind_{P'}^G((\pi_3\otimes\pi_1\times\pi_2)_{un})\\
\ind_{P_{k_1,k_3,k_2}}^G((\pi_1\otimes\pi_3\otimes\pi_2)_{un})\arrow[rru,"J_P(\pi_1\otimes\pi_3)\times 1_{\pi_2}",bend right=13]&&
\end{tikzcd}\]
To see that it commutes, apply \cite[Proposition 7.8(ii)]{Dat} to the bottom two arrows and \cite[Proposition 7.8(iii)]{Dat} to the top arrow. Then the commutativity follows from \cite[Proposition 7.8(i)]{Dat}.
For (6), \cite[Proposition 7.8(i)]{Dat} tells us that $J_P(\sigma)\circ J_{P'}(\overline{\sigma})$ is a scalar in $R(T)$.
Thus, after multiplying by the suitable normalization factors and setting $T=1$, the claim follows. Moreover, it also is easy to check that this diagram implies the claims regarding the poles.

It remains to prove (5) and here we argue as in \cite[Theorem 2.3]{LMsquare}. If one of the $R_{\pi_i,\pi_3}$ is an isomorphism, the composition \[(R_{\pi_1,\pi_3}\times 1_{\pi_2})\circ (1_{\pi_1}\times R_{\pi_2,\pi_3})\neq 0\] and hence $r_{\pi_1\times\pi_2,\pi_3}= r_{\pi_1,\pi_2}+r_{\pi_2,\pi_3}$. If $\pi_3$ is irreducible and $r_{\pi_1\times\pi_2,\pi_3}< r_{\pi_1,\pi_2}+r_{\pi_2,\pi_3}$, then
\[(R_{\pi_1,\pi_3}\times 1_{\pi_2})\circ (1_{\pi_1}\times R_{\pi_2,\pi_3})=0\] and hence \[\pi_1\times \mathrm{Im}(R_{\pi_2\times\pi_3})\hra \mathrm{Ker}(R_{\pi_1,\pi_3})\times \pi_2\hra \pi_1\times\pi_3\times \pi_2.\]
Twisting with the inner automorphism given by
    \[w=\begin{pmatrix}
        0&0&1_{k_2}\\
        0&1_{k_3}&0\\
        1_{k_1}&0&0\\
    \end{pmatrix},\]
we obtain that 
\[\mathrm{Im}(R_{\pi_2\times\pi_3})^{w_1}\times \pi_1\hra \pi_2\times\mathrm{Ker}(R_{\pi_1,\pi_3})^{w_2}\hra \pi_2\times\pi_3\times \pi_1,\]
where ${*}^{w'}$ denotes the twist by the inner automorphism$g\mapsto w'gw'^{-1}$, and  \[w_1=\begin{pmatrix}
        0&1_{k_2}\\
        1_{k_3}&0
    \end{pmatrix},\, w_2=\begin{pmatrix}
        0&1_{k_3}\\
        1_{k_1}&0
    \end{pmatrix}.\]
By \Cref{C:3ind}, the irreducibility of $\pi_3$ implies that $\mathrm{Ker}(R_{\pi_1,\pi_3})=\pi_3\times \pi_1$, a contradiction.
\end{proof}
\subsection{ Square irreducible representations}\label{S:sir}
Using the above mod $\ell$-analogous of the results of \cite{LMsquare}
, one can arrive via the same arguments to the following results.
We say $\Pi\in \Rep(G_n)$ is \emph{SI} if $\Pi$ is socle-irreducible, and its socle appears with multiplicity one in its composition series. Similarly, we say $\Pi$ is \emph{CSI}, if is cosocle-irreducible, and its cosocle appears with multiplicity one in its composition series.
A representation $\pi\in \Rep_n$ is called square-irreducible or $\square$-irreducible if $\pi\times\pi\in \Irr_{2n}$.
The following three theorems follow \emph{muta mutandis} as in \cite{LMsquare}, thanks to \Cref{L:inop} and \Cref{C:3ind}.
\begin{lemma}[{\cite[Corollary 2.5]{LMsquare}, \cite[Corollary 3.3]{KKKOb}}]\label{L:Dsi}
    A representation of  $\pi\in \Rep_n$ is $\square$-irreducible if and only if one of the following equivalent conditions holds.
    \begin{enumerate}
        \item $\pi\times\pi$ is irreducible.
        \item $\pi\times \pi$ is SI.
        \item $\dim_R\mathrm{End}_{G_{2n}}(\pi\times\pi)=1$.
        \item $R_{\pi,\pi}$ is a scalar.
    \end{enumerate}
\end{lemma}
Note that a cuspidal representation is square-irreducible if and only if it is $\square$-irreducible cuspidal and we denote the subset of square-irreducible representations of $\Irr_n$ by $\irs(G_n)$. More generally, for $\De=[a,b]_\rho$ a segment, $\Z(\De)$ is $\square$-irreducible if and only if $b-a+1<o(\rho)$.
\begin{lemma}[{\cite[Lemma 2.8]{LMsquare}, \cite[Theorem 3.1]{KKKOa}}]\label{L:si1}
    Let $\pi\in \irs(G_n)$ and $\tau\in \Irr_m$. Then 
    both $\pi\times\tau$ and $\tau\times \pi$ are SI and CSI.
\end{lemma}
\begin{lemma}[{\cite[Theorem 4.1.D]{LMbinary}}]\label{L:si2}
Let $\pi\in \irs(G_n)$.
    The maps
    \[\mathrm{soc}(\pi\times {}\cdot{})=\mathrm{cos}({}\cdot{}\times \pi)\colon \Irr\ra \Irr,\, \mathrm{soc}({}\cdot{}\times \pi)=\mathrm{cos}(\pi\times {}\cdot{})\colon \Irr\ra \Irr\]
    are injective.
\end{lemma}
\begin{lemma}\label{L:sipowers}
    If $\pi\in \irs(G_n)$ then $\pi^k\in \irs(G_{nk})$ for all $k\in\ZZ_{>0}$.
\end{lemma}
\begin{proof}
    By \Cref{L:Dsi} it is enough to show that $R_{\pi^k,\pi^k}$ is a scalar.
    We start by showing that $R_{\pi^k,\pi}$ is a scalar for all $k$, which we do by induction on $k$.
    By assumption, we know that $R_{\pi,\pi}$ is a scalar.    
By \Cref{L:inop}(4) and (5) \[R_{\pi^k,\pi}=(R_{\pi^{k-1},\pi}\times 1_{\pi})\circ(1_{\pi^{k-1}}\times R_{\pi,\pi})\] is a scalar by the induction hypothesis on $k$. By \Cref{L:inop}(6) also $R_{\pi,\pi^k}$ is a scalar.
    Next we prove by induction on $k$ that $R_{\pi^k,\pi^k}$ is a scalar. Again by \Cref{L:inop}(4) and (5), we know that \[R_{\pi^k,\pi^{k-1}}=(R_{\pi^{k-1},\pi^{k-1}}\times 1_{\pi})\circ(1_{\pi^{k-1}}\times R_{\pi,\pi^{k-1}})\] is a scalar. By \Cref{L:inop}(6) also $R_{\pi^{k-1},\pi^k}$ is a scalar. Applying a third time \Cref{L:inop}(4) and (5)
    we see that
    \[R_{\pi^k,\pi^{k}}=(R_{\pi^{k-1},\pi^{k}}\times 1_{\pi})\circ(1_{\pi^{k-1}}\times R_{\pi,\pi^{k}}),\] which therefore is a scalar.
\end{proof}
\begin{lemma}\label{L:nosi}
    Let $\rho$ be a cuspidal representation such that $o(\rho)=1$ and $\fm \in \Ms(\rho)_{ap}$. Then $\Z(\fm)$ is not square-irreducible.
\end{lemma}
\begin{proof}
    Assume otherwise that $\Z(\fm)$ is square-irreducible. Then by \Cref{L:sipowers} for all $k\in \ZZ_{>0}$ we have that $\Z(\fm)^k$ is square-irreducible and in particular irreducible. Hence $\Z(\fm)^{k}\cong \Z(k\fm)$ by \Cref{L:sam}. But if $k>\ell=e(\rho)$, the multisegment $k\fm$ is no longer aperiodic, and hence its cuspidal support is no longer contained in $\mathbb{N}(\ZZ[\rho])$. But on the other hand, the cuspidal support of $\Z(\fm)^{k}$ must be contained in $\Ms(\rho)$, a contradiction.
\end{proof}

\subsection{ P-regular representations}\label{S:preg}
Motivated by the notion of $(P,Q)$
-regularity introduced in \cite[§2]{Dat}, we will introduce a similar condition called $P$-regularity. Note that being $(P,Q)$-regular is a slightly more fine-tuned notion, but for our purposes the following more blunt definition is sufficient.

Let $P$ be a parabolic subgroup of $G=G_n$ with Levi-decomposition $P=M\rtimes U$ and opposite parabolic subgroup $\op$.
A representation $\sigma$ of $M$ of finite length is called $P$-\emph{regular} if
the natural restriction map
        \[\Ho_{G\times G}(C_c^\infty(G),\ind_\op^G(\sigma)\otimes \ind_P^G(\sigma^\lor))\ra \]\[\ra\Ho_{\op\times P}(C_c^\infty(\op P),\ind_\op^G(\sigma)\otimes \ind_P^G(\sigma^\lor))\]
is an isomorphism of $R$-vector spaces.
\begin{lemma}\label{L:preg}
Let $\sigma\in\Rep(M)$.
    \begin{enumerate}
        \item If $\sigma$ is irreducible and $P$-regular, then 
         \[\dim_R\Ho_{G\times G}(C_c^\infty(G),\ind_\op^G(\sigma)\otimes \ind_P^G(\sigma^\lor))=1.\]
        \item If $\sigma$ is irreducible and it appears in $r_{\op}(\ind_{\op}^G(\sigma))$ with multiplicity $1$,
        it is $P$-regular.
    \end{enumerate}
\end{lemma}
\begin{proof}
For the first claim, note that by Frobenius reciprocity, we have 
\[\dim_R\Ho_{\op\times P}(C_c^\infty(\op P),\ind_\op^G(\sigma)\otimes \ind_P^G(\sigma^\lor))=\]\[=\dim_R \Ho_M(r_{\op\times P}(C_c^\infty(\op P)),\sigma\otimes \sigma^\lor)=\]\[=\dim_R \Ho_M(C_c^\infty(M),\sigma\otimes \sigma^\lor)=\dim_R\Ho_M(\sigma,\sigma)=1.\]

For the second claim, note that  
by the first part and the existence of intertwining operators, it follows that it is enough to show that a non-zero map $f\colon C_c^\infty(G)\mapsto \ind_\op^G(\sigma)\otimes \ind_P^G(\sigma^\lor)$ restricts to a non-zero map on $C_c^\infty(\op P)$. We now apply Frobenius reciprocity with respect to $r_P$ to obtain a non-zero map 
\[f'\colon r_{\{1\}\times P}(C_c^\infty(G)))\cong \ind_{P\times M}^{G\times M}(C_c^\infty(M))\ra \ind_\op^G(\sigma)\otimes \sigma^\lor\]
and Frobenius reciprocity with respect to $r_{\op}$ to obtain a non-zero map
\[f''\colon  r_{\op\times \{1\}}(\ind_{P\times M}^{G\times M}(C_c^\infty(M)))\ra \sigma\otimes \sigma^\lor.\]
After applying the Geometric Lemma to the left side, we will show that $f''$ restricts to a non-zero map on \[F(1)(C_c^\infty(M))=C_c^\infty(M)\hra r_{\op\times \{1\}}(\ind_{P\times M}^{G\times M}(C_c^\infty(M)))\ra \sigma\otimes \sigma^\lor.\]
Indeed, let $w_i\in G$ be a permutation matrix and assume that $P w_i \op$ is the smallest element in $P\bs G/ \op$ such that $f'$ does not vanish on $F(w_i)(C_c^\infty(M))$, \emph{i.e.} we obtain a non-zero map
\[f''\colon F(w_i)(C_c^\infty(M))\ra \sigma\otimes\sigma^\lor.\]
Using the specific form of $F(w_i)$, \emph{cf.} \cite[Geometric Lemma]{ASENS_1977_4_10_4_441_0}, and applying first Bernstein reciprocity and then Frobenius reciprocity we obtain a non-zero
morphism
\[C_c^\infty(M)\ra F(w_i^{-1})(\sigma) \otimes \sigma^\lor.\]
where here $F(w_i^{-1})(\sigma)$ is now the subquotient coming from the Geometric Lemma applied to $r_\op(\ind_\op^G(\sigma))$.
This in turn gives a non-zero morphism
\[\sigma\ra F(w_i^{-1})(\sigma).\]
By assumption $\sigma$ appears only with multiplicity $1$ in $r_\op(\ind_{\op}^G(\sigma))$ and $F(1)(\sigma)\cong \sigma$. Thus $w_i^{-1}\in \op$ and hence $P w_i \op= P\op$. We therefore showed that $f''$ does not vanish on $F(1)(C_c^\infty(M))$.
But now $F(1)(C_c^\infty(M))=r_{\op\times P}(C_c^\infty(\op P))$ and thus by Frobenius reciprocity
$f$ does not vanish on on $C_c^\infty(\op P)$.
\end{proof}
    Assume now $P=P_{(n_1,n_2)}$, $\sigma=\sigma_1\otimes\sigma_2$ is $P$-regular, irreducible and either $\sigma_1$ or $\sigma_2$ is $\square$-irreducible. In this case we call $\sigma$ \emph{strongly} $P$-regular.
    By \Cref{L:si1} $\pi\coloneq \mathrm{cos}(\ind_P^G(\sigma))\cong \mathrm{soc}(\ind_\op^G(\sigma))$ is irreducible.
    Moreover, by \Cref{L:preg}(1), the only morphism up to a scalar in \[\Ho_{G}(\ind_P^G(\sigma),\ind_\op^G(\sigma) )\] is given by
    \[\ind_P^G(\sigma)\sra \pi\hra \ind_\op^G(\sigma),\]
    since 
    \[\dim_R\Ho_{G}(\ind_P^G(\sigma),\ind_\op^G(\sigma) )=\]\[=\dim_R\Ho_{G\times G}(C_c^\infty(G),\ind_\op^G(\sigma)\otimes \ind_P^G(\sigma^\lor) )=1.\]
    For the next corollary, we identify for every admissible representation $\sigma$ of $M$
    \[\ind_\op^G(\sigma)\otimes \ind_P^G(\sigma^\lor)\cong \ind_\op^G(\sigma^\lor)^\lor\otimes \ind_P^G(\sigma)^\lor\]
\begin{corollary}\label{C:preg}
    Let $\sigma\in \Irr(M)$ be a strongly $P$-regular representation and let $\pi=\mathrm{soc}(\ind_\op^G(\sigma))$.
    Let \[ T\in \Ho_{\op\times P}(C_c^\infty(\op P),\ind_\op^G(\sigma)\otimes \ind_P^G(\sigma^\lor)),\]
    written as
    \[\phi\mapsto (f_1\otimes f_2\mapsto T(\phi)(f_1\otimes f_2)),\, f_1\otimes f_2\in \ind_\op^G(\sigma^\lor)\otimes \ind_P^G(\sigma).\]
    Then for all $ f_1\otimes f_2\in \ind_\op^G(\sigma^\lor)\otimes \ind_P^G(\sigma)$ there exists a matrix coefficient $f$ of $\pi$ such that for all $\phi\in C_c^\infty(\op P)$ 
    \[T(\phi)(f_1\otimes f_2)=\int\displaylimits_{\op P}\phi(p)f(p)\ddd p.\]
\end{corollary}
\begin{proof}
    Since $\sigma$ is irreducible, \Cref{L:preg}(1) tells us that the natural inclusion 
    \[\Ho_{G\times G}(C_c^\infty(G),\pi\otimes \pi^\lor)\hra \Ho_{G\times G}(C_c^\infty(G),\ind_\op^G(\sigma)\otimes \ind_P^G(\sigma)^\lor)\]
is an isomorphism. Note that every morphism in the former space is given by sending a matrix coefficient $f$ and a Schwartz-function $\phi\in C_c^\infty(G)$ to
    \[\int\displaylimits_Gf(g)\phi(g)\ddd g.\]
Now by $P$-regularity every map in \[\Ho_{\op\times P}(C_c^\infty(\op P),\ind_\op^G(\sigma)\otimes \ind_P^G(\sigma^\lor))\] is just the restriction of a map in \[\Ho_{G\times G}(C_c^\infty(G),\ind_\op^G(\sigma)\otimes \ind_P^G(\sigma)^\lor),\] hence it has to be of the form
\[\phi\in C_c^\infty(\op P)\mapsto \int\displaylimits_{\op P}\phi(p)f(p)\ddd p.\] 
\end{proof}
Next, we see that strongly $P$-regular representation will interact nicely with $L$-factors. Recall that for $\pi$ an irreducible representation, we had
\[P(\pi,T)=\frac{1}{L(\pi,T)}.\]
\begin{prop}\label{P:L1}
Let $P=P_{(n_1,n_2)}=M\rtimes U$ be a parabolic subgroup of $G_n$ and $\sigma=\sigma_1\otimes \sigma_2$ a strongly $P$-regular representation. Set $\pi=\mathrm{cos}(\ind_P^{G}(\sigma))$. Then
    $P(\sigma_2,T)$ divides $P(\pi,T)$.
\end{prop}
Before we come to the proof, let us observe some parallels between this statement and a special case of it, namely \cite[Theorem 2.7]{JL2}. In there Jacquet uses the following fact, which only holds true over $\ql$ in full generality. For a multisegment $\fm$ over $\ql$, let $\langle \fm\rangle$ be the representation associated to $\fm$ via the Langlands classification. If $\rho v_\rho ^b$ a cuspidal representation in the cuspidal support of $\fm$ with $b$ maximal, let $\De_1+\ldots+\De_k$ be the sub-multisegment of $\fm$ consisting of all segments of the form $[a,b]_\rho$.
Then \[\bigtimes_{i=1}^k\langle\De_i\rangle\] is square-irreducible and \[\sigma=\bigtimes_{i=1}^k\langle\De_i\rangle\otimes \langle \fm-\sum_{i=1}^k\De_i\rangle\]
is strongly-$P$ regular with \[\mathrm{cos}(\ind_P^{G_n}(\sigma))=\langle \fm\rangle.\]
Thus \Cref{P:L1} can be seen as a generalization of this fact over $\ql$ and as we will see, also its proof follows essentially the same idea.

Before we start with the proof, we define the following.
Let $H\colon G_n\times G_n\ra R$ be a smooth function such that
\[H(u_1mg_1,u_2mg_2)=H(g_1,g_2)\]
for all $u_1\in \overline{U},\, u_2\in U,\, m\in M,g_1,g_2\in G_n$
and \[m\mapsto H(g_1,mg_2)\] is a matrix coefficient of $\sigma\otimes\delta_p^{\frac{1}{2}}$.
Then for $\phi\in \cp$, 
   define the integral
    \[T_H( \phi)\coloneq \int\displaylimits_{U\times M\times\overline{U}}H(1,m) \phi(\overline{u}^{-1}um)\ddd \overline{u}\ddd u\ddd m.\]
\begin{lemma}\label{L:matrixcoefficient}
    The integrals $T_H(\phi)$ are well defined for $\phi\in \cp$ and define a non-zero $\op\times P$-equivariant morphism
    \[\cp\ra \ind_\op^{G_n}(\sigma) \otimes \ind_P^{G_n}(\sigma^\lor).\]
\end{lemma}
\begin{proof}
By the assumption on $\phi$ the integral converges. Note that above $H$'s are precisely linear combinations of functions of the form
\[H(g_1,g_2)=f^\lor(g_1)(f(g_2))\] for $f^\lor\in \ind_\op^{G_n}(\sigma^\lor),\, f\in \ind_{P}^{G_n}(\sigma)$.
\end{proof}
\begin{proof}[Proof of \Cref{P:L1}]
Having introduced the above machinery, we can now mimic the proof of \cite[Theorem 2.7]{JL2}.
We recall \[\delta_P(m_1,m_2)=\lvert \det\nolimits'(m_1)\lvert^{dn_2}\lvert \det\nolimits'(m_2)\lvert^{-dn_1}.\]
Note that the claim of the proposition is equivalent to \[\mathcal{L}(\sigma_2)\subseteq \mathcal{L}(\pi),\]\emph{cf}. \Cref{C:fracid}.
To show this, we need to find for each ${\phi_2}\in \mathscr{S}_{R}(M_{n_2}(\mathrm{D}))$ and matrix coefficient ${f_2}$ of $\sigma_2$, a ${\phi}\in \mathscr{S}_{R}(M_{n}(\mathrm{D}))$ and a matrix coefficient ${f}$ of $\pi$ such that
    \[Z({\phi},Tq^{-\frac{dn-1}{2}},{f})=Z({\phi_2},Tq^{-\frac{dn_2-1}{2}},{f_2}).\]
Choose now a coefficient $f_1$ of $\sigma_1$ and $\phi_1\in \mathscr{S}_{R}(M_{n_1}(\mathrm{D}))$ supported in a suitable neighbourhood of $1_n$ such that $Z(\phi_1,Tq^{-\frac{dn_1-1}{2}},f_1)=1$ and choose \[\phi_{1,2}\in C_c^\infty(\overline{U}),\,\phi_{2,1}\in C_c^\infty(U)\] which are supported in a suitable neighbourhood of $1_n$ such that 
\[\phi(\overline{u}um)=\phi_{1,2}(\overline{u})\phi_{2,1}(u)\phi_1(m_1)\phi_2(m_2),\, m=(m_1,m_2)\]
    satisfies
    \[\int\displaylimits_{\overline{U}\times U}\phi(\overline{u}^{-1}um)\ddd \overline{u}\ddd u =\phi_1(m_1)\phi_2(m_2).\]
Note that in particular for all $N\in\mathbb{Z}$, \[\phi\chi_{G_n(N)}\in \cp,\] where $\chi_{G_n(N)}$ is the characteristic function of $G_n(N)$ and $\phi$ can be extended by $0$ to a function on $\schr$. Moreover, the coefficient of $T^N$ in \[Z(\phi_2, Tq^{-\frac{dn_2-1}{ 2}}, f_2)=Z(\phi_1, Tq^{-\frac{dn_1-1}{ 2}}, f_1)Z(\phi_2, Tq^{-\frac{dn_2-1}{ 2}}, f_2)\] is
    \[\int\displaylimits_{M(N)}\lvert \det\nolimits'(m_1)\lvert^{\frac{dn_1-1}{2}}\lvert \det\nolimits'(m_2)\lvert^{\frac{dn_2-1}{2}}f_1(m_1)f_2(m_2)\]\[\int\displaylimits_{\overline{U}\times U}\phi(\overline{u}^{-1}um)\ddd \overline{u}\ddd u \ddd (m_1,m_2),\] which can be rewritten as
\begin{equation}\label{E:intu}=\int\displaylimits_{M(N)}\lvert \det\nolimits'(m)\lvert^{\frac{dn-1}{2}}\delta_{P}(m)^{\frac{1}{2}}f_1(m_1)f_2(m_2)\int\displaylimits_{\overline{U}\times U}\phi(\overline{u}^{-1}um)\ddd \overline{u}\ddd u  \ddd m.\end{equation}
Note that here we used that $\phi_1(m_1)$ is supported in a sufficiently small neighborhood of $1$ and hence in the support of $\phi_1$, we have $\lvert \det\nolimits'(m_1)\lvert =1$.
We can now find $H$ as in \Cref{L:matrixcoefficient} such that $H(1,m)=f(m_1)f(m_2)\delta_{P}(m_1,m_2)^{\frac{1}{2}}$ and hence we find by \Cref{C:preg} a matrix coefficient $f$ of $\pi$, which is independent of $N$, such that above integral (\ref{E:intu}) equals 
    \[\int_{G_n}\lvert \det\nolimits'(g)\lvert^{\frac{dn-1}{2}}f(g)(\phi\chi_{G_n(N)})(g)\ddd u g=\int_{G_n(N)}\lvert \det\nolimits'(g)\lvert^{\frac{dn-1}{2}}f(g)\phi(g)\ddd g.\]
    This in turn is nothing else than the coefficient of $T^N$ in
    \[Z({\phi},Tq^{-\frac{dn-1}{2}},{f}).\] Thus $\mathcal{L}(\sigma_2)\subseteq \mathcal{L}(\pi)$.
\end{proof}
\section{Derivatives}\label{S:Der}
In this section we define derivatives with respect to a $\square$-irreducible cuspidal representation, following the ideas of \cite{Mder} and \cite{Jader}, who establish all results in this section in the case $R=\ql$. Derivatives will also give rise to most of the $P$-regular representations we will encounter throughout the rest of the paper.
\subsection{ Basic properties of derivatives}
Let $\rho,\pi\in \Irr$. We call a representation $\pi'\in \Irr$ a \emph{right-$\rho$-derivative} respectively a \emph{left-$\rho$-derivative} of $\pi$ if 
\[\pi\hra \pi'\times \rho\text{ respectively }\pi\hra \rho\times \pi'.\]
As a consequence of \Cref{L:si1} and \Cref{L:si2} we obtain the following lemma and its corollary.
\begin{lemma}\label{L:derivatives}
    Let $\rho\in \irs$ and $\pi\in \Irr$.
    The left- and right-$\rho$-derivative of $\pi$ are up to isomorphism unique. We denote them by $\mathcal{D}_{l,\rho}(\pi)$ and $\mathcal{D}_{r,\rho}(\pi)$ respectively and set them to $0$ if no such representation exists.

    Moreover, if $\mathcal{D}_{r,\rho}(\pi)\neq 0$, then $\pi=\mathrm{soc}(\mathcal{D}_{r,\rho}(\pi)\times \rho)$ and it appears in its decomposition series with multiplicity $1$. Analogously, if $\mathcal{D}_{l,\rho}(\pi)\neq 0$, then $\pi=\mathrm{soc}(\rho\times \mathcal{D}_{l,\rho}(\pi))$ and it appears in its decomposition series with multiplicity $1$.
\end{lemma}
\begin{corollary}\label{C:derinj}
    Let $\pi_1,\, \pi_2\in \Irr$ and $\rho\in\irs$. If $\mathcal{D}_{l,\rho}(\pi_1)\cong \mathcal{D}_{l,\rho}(\pi_2)\neq 0$ or $\mathcal{D}_{r,\rho}(\pi_1)\cong \mathcal{D}_{r,\rho}(\pi_2)\neq 0$, then $\pi_1\cong \pi_2$. 
\end{corollary}
\begin{lemma}\label{L:dervanish}
Let $\rho\in \scu(G_m)$ and $\pi\in \Irr_n$. Then $\Dl(\pi)\neq 0$ if and only if there exists an irreducible representation $\pi'\in \Irr_{n-m}$ such that 
        \[[\pi'\otimes\rho]\le [r_{(n-m,m)}(\pi)].\]
        Similarly, $\Dr(\pi)\neq 0$ if and only if there exists an irreducible representation $\pi'\in \Irr_{n-m}$ such that 
        \[[\rho\otimes\pi']\le [r_{(m,n-m)}(\pi)].\]
\end{lemma}
\begin{proof}
We only prove the claim for right-derivatives, the one for left-derivatives follows analogously.
Note that if $\pi$ admits a right-derivative $\pi'$, we obtain by Frobenius reciprocity $r_{(n-m,m)}(\pi)\sra \pi'\otimes \rho$. On the other hand, let $\pi'$ be such that \[[\pi'\otimes\rho]\le [r_{(n-m,m)}(\pi)].\]
We now argue inductively on $n$ that then $\pi$ admits a right derivative. Choose $\rho'$ a cuspidal representation of $G_{m'}$ and $\tau\in \Irr_{n-m'}$ such that $\pi\hra\tau\times\rho'$. If $\rho'\cong \rho$ we are done, thus we assume from now on otherwise.
Choose $k$ maximal such that there exists $\tau'\in \Irr_{n-km-m'}$
\[\tau\hra \tau'\times\rho^k.\]
We thus have $\pi\hra \tau'\times\sigma$, where $\sigma$ is an irreducible subquotient of $\rho^k\times\rho'$. 
By induction we note that the maximality of $k$ implies that there exists no $\tau''\in \Irr_{n-(k+1)m-m'}$ such that 
\begin{equation}\label{E:whymax}[\tau''\otimes\rho]\le [r_{(n-(k+1)m-m',m)}(\tau')].\end{equation}
We will now show that $\sigma$ admits a right-derivative with respect to $\rho$, which would show the claim. 
Firstly,
\[[\pi'\otimes\rho]\le [r_{(n-m,m)}(\pi)]\le [r_{(n-m,m)}(\tau'\times \sigma)].\] The Geometric Lemma, the fact that $\rho$ is supercuspidal and (\ref{E:whymax}) imply therefore that $k>0$ and there exists 
$\sigma'\in\Irr$ such that \[[\sigma'\otimes\rho]\le [r_{((k-1)m+m',m)}(\sigma)].\]
We first note that if $\rho'\ncong \rho v_\rho^{\pm 1}$, by \Cref{L:H2} $\sigma\cong \rho^k\times \rho'\cong \rho'\times\rho^k$ in which the claim would follow immediately.
Therefore we assume from now on $\rho'\in\{\rho v_\rho,\rho v_\rho^{-1}\}$ and we differentiate two cases, namely $o(\rho)>2$ and $o(\rho)=2$.

\textbf{Case 1: $o(\rho)>2$}

Then $\rho\times \rho'$ has two subquotients $\sigma_1$ and $\sigma_2$ such that $\rho^{k-1}\times\sigma_1$ and $\rho^{k-1}\times\sigma_2$ are irreducible by \Cref{T:N2}. Thus if $k>1$, the claim follows from the commutativity of parabolic induction. If $k=1$, 
\[\sigma'\otimes\rho\le [r_{(m,m)}(\sigma)]\] implies thus by  \Cref{L:N1} that $\sigma\hra \rho'\times\rho$.

\textbf{Case 1: $o(\rho)=2$}

Then $\rho\times \rho'$ has three subquotients $\sigma_1=\Z([0,1]_\rho)$, $\sigma_2=\Z([1,2]_\rho)$ and $\sigma_3=St(\rho,2)$ cuspidal. Note that $\rho^{k-1}\times\sigma_3$ is irreducible by \Cref{L:H2} and hence if $\sigma\hra \rho^{k-1}\times\sigma_3$, we have that $k>1$ and the claim follows from the commutativity of parabolic induction.
If $k=1$, it follows again from \Cref{L:N1} that $\sigma=\sigma_2\hra \rho v_\rho^{-1}\times \rho$.
Thus we assume $k>1$ and also note that if $\sigma\hra \rho^{k-1}\times \sigma_2$, then we are done since, $\sigma_2\hra \rho v_\rho^{-1}\times \rho$.
We are therefore in the situation where $\sigma\hra \rho^{k-1}\times \sigma_1$. Next we show that \begin{equation}\label{E:inque}[\rho\times\sigma_1]=[\Z([0,2]_\rho)]+[\Z([0,1]_\rho+[0,0]_\rho)]\end{equation} and 
\begin{equation}\label{E:specialcase}
    r_{(2m,m)}(\Z([0,2]_\rho))=\Z([0,1]_\rho)\otimes \rho,\, r_{(2m,m)}(\Z([0,1]_\rho+[0,0]_\rho))=\rho\times \rho\otimes \rho v_\rho.
\end{equation}
Indeed, note that by the Geometric Lemma
\[[r_{(2m,m)}(\Z([0,1]_\rho)\times \rho)]=[\Z([0,1]_\rho)\otimes \rho]+[\rho\times \rho\otimes \rho v_\rho]\]
and since $\Z([0,1]_\rho)\times \rho$ does not contain a cuspidal representation by \Cref{C:N1}, it is at most of length $2$. On the other hand, it also contains by \Cref{L:quot} and \Cref{T:N5} the two representations of (\ref{E:inque}) as subquotients. Since by \Cref{L:N1} the parabolic restriction of $\Z([0,2]_\rho)$ is of the above-mentioned form, the claim follows. In particular, if $\sigma\hra \rho^{k-2}\times Z([0,2]_\rho)$ we are done by \Cref{L:quot} and therefore we can assume that 
\[\sigma\hra \rho^{k-2}\times \Z([0,1]_\rho+[0,0]_\rho).\]
From (\ref{E:specialcase}) we obtain that $k>2$. It is now enough to show that $\rho\times \Z([0,1]_\rho+[0,0]_\rho)$ is irreducible, since then \[\rho\times \Z([0,1]_\rho+[0,0]_\rho)\cong  \Z([0,1]_\rho+[0,0]_\rho)\times\rho.\]
To see this note that by \Cref{L:sam} \[[\Z([0,1]_\rho+2[0,0]_\rho]\le [\rho\times \Z([0,1]_\rho+[0,0]_\rho)]\] and by \Cref{T:N5}
\[[r_{(m,3m)}(\Z([0,1]_\rho+2[0,0]_\rho))]\ge [\rho\otimes \Z([1,1]_\rho+2[0,0]_\rho)], \]\[\Z([0,0]_\rho+[1,1]_\rho+[0,0]_\rho)\cong \Z([0,0]_\rho+[1,1]_\rho)\times \rho, \]
since $\Z([0,0]_\rho+[1,1]_\rho)\cong \st(\rho,2)$ is cuspidal. By Frobenius reciprocity we thus have 
\[[r_{(m,2m,m)}(\Z([0,0]_\rho+2[0,0]_\rho))]\ge [\rho\otimes \Z([0,0]_\rho+[1,1]_\rho)\otimes \rho], \]
and hence there exists an irreducible $\pi_3$ such that 
\[[r_{(3m,m)}(\Z([0,1]_\rho+2[0,0]_\rho))]\ge [\pi_3\otimes \rho].\]
On the other hand the Geometric Lemma and (\ref{E:specialcase}) imply
\[[r_{(3m,m)}(\Z([0,1]_\rho+[0,0]_\rho)\times \rho)]\ge [\rho^3\otimes \rho v_\rho].\]
 Since $\rho^3\otimes \rho v_\rho$ is residually non-degenerate, \Cref{T:N5} implies that the irreducible subquotient of 
\[\Z([0,1]_\rho+[0,0]_\rho)\times \rho,\] whose parabolic restriction contains $\rho^3\otimes \rho v_\rho$ must be
$\Z([0,1]_\rho+2[0,0]_\rho).$
But 
\[[r_{(3m,m)}(\Z([0,1]_\rho+[0,0]_\rho)\times \rho)]=[\rho^3\otimes \rho v_\rho]+[\Z([0,1]_\rho+[0,0]_\rho)\otimes \rho]\]
by the Geometric Lemma and (\ref{E:specialcase}) and hence $\pi_3\cong \Z([0,1]_\rho+[0,0]_\rho)$ and therefore
\[[r_{(3m,m)}(\Z([0,1]_\rho+[0,0]_\rho)\times \rho)]=[r_{(3m,m)}(\Z([0,1]_\rho+2[0,0]_\rho))].\]
This implies that any other possible subquotient $\pi_4$ of $\Z([0,1]_\rho+[0,0]_\rho)\times \rho$ must satisfy \[r_{(3m,m)}(\pi_4)= 0.\]
But this yields a contradiction as follows. The cuspidal support of $\pi_4$ is either $3[\rho]+[\rho v_\rho]$ or $Z([0,0]_\rho+[1,1]_\rho)+2[\rho]$. In both cases, it follows that either the right-derivative of $\pi_4$ with respect to $\rho$ or $\rho v_\rho$ is non-zero. But then Frobenius reciprocity gives the desired contradiction.
\end{proof}
\begin{rem}
    Note that for $\rho$ a cuspidal non supercuspidal representation of $G_m$, the \emph{if}-direction of \Cref{L:dervanish} does not hold in general. For example, let $\rho\in\scu$ and \[\pi=\Z([-1,0]_\rho+[1,1]_\rho+\ldots+ [e(\rho)-1,e(\rho)-1]_\rho).\] Then the cuspidal support of $\pi$ is \[[\rho]+\ldots+[\rho v_\rho^{e(\rho)-2}]+2\cdot [\rho v_\rho^{e(\rho)-1}]\] and therefore does not contain the cuspidal representation $\mathrm{St}(\rho,e(\rho))$, but by \Cref{T:N5}, $r_{(m,e(\rho)m)}(\pi)$ contains $\rho v_\rho^{-1}\otimes \mathrm{St}(\rho,e(\rho))$.
\end{rem}
Recall that $\rho^l$ is irreducible if $\rho$ is $\square$-irreducible for all $l\in\mathbb{Z}_{>0}$ by \Cref{T:N2}.
\begin{lemma}\label{L:helpder}
Let $\pi\in\Irr_n$, $\rho\in \scu(G_m)$, $l\in\ZZ_{\ge 0}$ and $\pi'\in \Irr_{n-lm}$ be such that \[\pi\hra \pi'\times \rho^l\] and $\pi'$ admits no right derivative with respect to $\rho$.
    Then $\pi'\otimes\rho^l$ is strongly ${P_{(lm,n-lm)}}$-regular.
    Moreover, if $\pi''$ is an irreducible representation such that $\pi''\otimes\rho^l$ appears in $r_{(n-lm,lm)}(\pi)$ or $r_{(n-lm,lm)}(\pi'\times \rho^l)$, then $\pi''\cong \pi'$, and $\pi'\otimes \rho^l$ appears in $r_{(n-lm,lm)}(\pi'\times \rho^l)$ with multiplicity $1$. Finally, $\pi$ appears with multiplicity $1$ in $\pi'\times\rho^l$.
\end{lemma}
\begin{proof}
To prove the first claim, we observe that for such maximal $l$,
    \[r_{(n-lm,lm)}(\pi'\times \rho^l)\] contains the 
    representation $\pi'\otimes \rho^l$ with multiplicity one by \Cref{L:dervanish} and the Geometric Lemma. Thus it is ${P_{(lm,n-lm)}}$-regular by \Cref{L:preg}(2), since 
    $\overline{P_{(lm,n-lm)}}$ is conjugated to $P_{(n-lm,lm)}$. Moreover, it is also clear that if $\pi''\otimes\rho^l$ appears in $r_{(n-lm,lm)}(\pi'\times \rho^l)$, $\pi''\cong \pi'$.
    Since by Frobenius reciprocity $r_{(n-lm,lm)}(\pi)$ contains $\pi'\otimes\rho^l$, $\pi$ appears with multiplicity $1$ in $\pi'\times\rho^l$ and it is the only irreducible subrepresentation.

    Finally, to see that $\pi'\otimes \rho^l$ is strongly ${P_{(lm,n-lm)}}$-regular, recall that $r_{(n-lm,lm)}(\pi)^\lor\cong r_{\overline{P_{(n-lm,lm)}}}(\pi^\lor)$ by Bernstein reciprocity. Thus
\[[\pi'\otimes \rho ^l]\le[r_{(n-lm,lm)}(\pi)]\]
if and only if
\[[\pi'^\lor\otimes (\rho^\lor) ^l]\le[r_{\overline{P_{(n-lm,lm)}}}(\pi^\lor)].\]
Since $\overline{P_{(n-lm,lm)}}$ is conjugated to $P_{(lm,n-lm)}$ this is equivalent to 
\[[(\rho^\lor) ^l\otimes \pi'^\lor ]\le[r_{(lm,n-lm)}(\pi^\lor)].\]
Now one can prove completely analogously to above that this is equivalent to $\pi^\lor\hra (\rho^\lor)^l\times\pi'^\lor$ and $\pi'^\lor$ does not admit a left derivative with respect to $\rho^\lor$. But this is implies that $\rho^l\times \pi'\sra \pi$.
\end{proof}
\begin{rem}
    Note that \Cref{L:helpder} allows one to give an alternative proof of \Cref{L:derivatives} and \Cref{C:derinj} for $\square$-irreducible cuspidal representations.
\end{rem}
We now define \[\Dl^k(\pi)\coloneq \overbrace{\Dl\circ\cdot \circ \Dl}^k(\pi),\, \Dr^k(\pi)\coloneq \overbrace{\Dr\circ\cdot \circ \Dr}^k(\pi),\, \Dl^0\coloneq \Dr^0\coloneq \mathrm{1},\] $\dlm(\pi)$ the maximal $k$ such that $\Dl^k(\pi)\neq 0$, $\Dlm(\pi)\coloneq \Dr^{\drm(\pi)}(\pi)$, 
$\drm(\pi)$ the maximal $k$ such that $\Dr^k(\pi)\neq 0$ and $\Drm(\pi)\coloneq \Dr^{\drm(\pi)}(\pi)$.
As a corollary of \Cref{L:dervanish} we obtain the following.
\begin{corollary}\label{C:der}
    For $\pi\in \Irr_n$ and $\rho\in \scu(G_m)$, $\dlm(\pi)$ is the maximal $k$ such that there exists $\pi'\in \Irr_{n-mk}$ with 
    \[[\pi'\otimes \rho^k]\le [r_{(n-km,km)}(\pi)].\]
    In this case $\pi'\cong \Dlm(\pi)$ and $\pi'\otimes \rho^k$ appears with multiplicity $1$ in the parabolic restriction. Analogously,
    $\drm(\pi)$ is the maximal $k$ such that there exists $\pi'\in \Irr_{n-mk}$ with 
    \[[\rho^k\otimes \pi']\le [r_{(km,n-km)}(\pi)].\]
    In this case $\pi'\cong \Drm(\pi)$ and $\rho^k\otimes\pi'$ appears with multiplicity $1$ in the parabolic restriction.

   Finally, \[\Dlm(\pi)\otimes \rho^{\dlm(\pi)}\text{ respectively } \rho^{\drm(\pi)}\otimes \Drm(\pi) \] are strongly $P_{(\dlm(\pi)m,n-\dlm(\pi)m)}$- respectively $P_{(n-\drm(pi)m,\drm(\pi)m)}$-regular.
\end{corollary}
Lastly, dualizing behaves well with respect to taking derivatives.
\begin{lemma}\label{L:derdual}
    Let $\pi\in \Irr_n$ and $\rho\in \scu(G_m)$. Then
    \[\Dl(\pi)^\lor\cong \mathcal{D}_{l,\rho^\lor}(\pi^\lor),\, \dlm(\pi)=d_{l,\rho^\lor}(\pi^\lor),\, \Dlm(\pi)^\lor\cong \mathcal{D}_{l,\rho^\lor,max}(\pi^\lor).\]
\end{lemma}
\begin{proof}
 By the contravariance of dualizing we have \[\pi^\lor\cong(\soc(\Dl(\pi)\times\rho))^\lor\cong \cos(\Dl(\pi)^\lor\times\rho^\lor)\stackrel{\Cref{L:si2}}{\cong}\soc(\rho^\lor\times\Dl(\pi)^\lor),\] therefore $\Dl(\pi)^\lor\cong \mathcal{D}_{l,\rho^\lor}(\pi^\lor)$
and from this, the other claims follow straightforwardly.
\end{proof}
\subsection{ Derivatives and multisegments}
In this section we will answer the following question: Given an aperiodic multisegment $\fm$ and $\rho\in \scu(G_m)$, what is $\Dl(\Z(\fm))$ and $\Dr(\Z(\fm))$?
To answer this we will define four maps
\[\Dr,\Dl, \soc({}\cdot{},\rho),\, \soc(\rho,{}\cdot{})\colon \Ms_{ap}\ra \Ms_{ap}.\]
In the case $R=\ql$ this has already been done independently in \cite[Theorem 2.2.1]{Jader} and \cite[Theorem 7.5]{Mder}.
\subsubsection{  Best matching functions}\label{S:bmf} Similarly to \cite[§4.3]{ML}, we make the following definition. Fix $\rho\in \scu$ and let $\fm=\sum_{i=1}^n\De_i$ be a multisegment. A best matching function for $\fm$ with respect to $\rho$ is a tuple $(A,f)$, where $A\subseteq\{1,\ldots,n\}$, possibly empty, and $f\colon A\ra\{1,\ldots,n\}$ an injective function satisfying the following properties.
\begin{enumerate}\label{L:properties}
     \item If $a\in A$, then $\De_a\in \mathcal{S}(\rho)$ and $b_\rho(\De_a)=0$.
    \item If $a\in A$, then $\De_{f(a)}\in \mathcal{S}(\rho),\,b_\rho(\De_{f(a)})=-1$ and $l(\De_a)\le l(\De_{f(a)})$.
    \item If $a\in A,\, b\in A^{\complement} =\{1,\ldots,n\}\setminus A$ such that $\De_b\in \mathcal{S}(\rho),\, b_\rho(\De_b)=0$ 
    and $l(\De_a)\le l(\De_b)\le l(\De_{f(a)})$, then $\De_b=\De_a$.
    \item If $a\in A,\, c\in  f(A)\com =\{1,\ldots,n\}\setminus f(A)$ such that $\De_c\in \mathcal{S}(\rho),\,b_\rho(\De_c)=-1$ 
    and $l(\De_a)\le l(\De_c)\le l(\De_{f(a)})$, then $\De_c=\De_{f(a)}$.
    \item For $b\in A\com ,\,  c\in f(A)\com $ such that $\De_b,\De_c\in \mathcal{S}(\rho), b_\rho(\De_b)=0,\,b_\rho (\De_c)=-1$, then $l(\De_c)<l(\De_b)$.
\end{enumerate}
\begin{ex}
    Let $\rho\in \scu$. 
    If \[\fm=[-1,0]_\rho+[-1,0]_\rho+[-3,0]_\rho+[-2,-1]_\rho+[-3,-1]_\rho+[-3,-1]_\rho,\]
    setting $A=\{1,2\}$, $f(1)=4,\, f(2)=5$ gives a best matching function with respect to $\rho$.
\end{ex}
We will assosciate to a best matching function $(A,f)$ of a multisegment $\fm$ the multisegments $\fm_f\coloneq\sum_{a\in A}\De_a$ and $f(\fm_f)\coloneq \sum_{a\in A}\De_{f(a)}$. 
If $\fm,\fm'$ are two multisegments and $(A,f)$ and $(A',f')$ are two best matching functions of $\fm$ and $\fm'$ with respect to $\rho$, we say $f$ and $f'$ are equivalent if $\fm_f=\fm'_{f'}$ and $ f(\fm_f)=f'(\fm'_{f'})$.
\begin{lemma}\label{L:uniqext}
    Given $\fm$ a multisegment and $\rho\in \scu$ there exists a best matching function $(A,f)$ of $\fm$ with respect to $\rho$, which can be constructed explicitly. Moreover, all best matching functions of $\fm$ with respect to $\rho$ are equivalent.
\end{lemma}
\begin{proof}
Write $\fm=\sum_{i=1}^n\De_i$.
    First assume that either there exists no $\rho$-segment $\De\in \fm$ with $b_\rho(\De)=0$ or for every $\rho$-segment $\De\in \fm$ with $b_\rho(\De)=0$ there exists no $\rho$-segment $\De'\in \fm$ with $b_\rho(\De')=-1$ and $l(\De)\le l(\De')$.
    Then properties (1) and (2) force $A$ to be empty. Moreover, if we take $A$ to be empty, the properties (1), (2), (3), and (4) are trivially true, and property (5) is true by assumption, thus we are done in this case. 

    If above assumption is not satisfied, we can fix a copy $\De$ of the longest $\rho$-segment in $\fm$ with $b_\rho(\De)=0$ such that there exists a $\rho$-segment $\De'\in \fm$ with $b_\rho(\De')=-1$ and $l(\De)\le l(\De')$.
    We also fix $\Gamma$, a copy of the $\rho$-segment of minimal length such that $b_\rho(\Gamma)=-1$ and $l(\De)\le l(\Gamma)$. 
    
    Let us first construct a best matching function of $\fm$ via induction on the number of segments. Without loss of generality, we can assume $\Gamma=\De_{n-1}$ and $\De=\De_n$. Let $f'\colon A'\ra\{1,\ldots,n-2\}$ be a best matching function of $\fm-\De_{n-1}-\De_n$. We set $A=A'\cup\{n\}$ and extend $f'$ to $f\colon A\ra\{1,\ldots,n\}$ by setting $f(n)=n-1$. We now show that $f$ is a best matching function. Note that since $f'$ is a best matching function, properties (1), (2), and (5) are trivially satisfied. For properties (3) and (4), the only non-trivial cases occur if $a=n$. If there would exist $\De_b$ or $\De_c$ as in property (3) and (4) such that \[l(\De_n)<l(\De_b)\le l(\De_{n-1})\text{ respectively }l(\De_n)\le l(\De_c)< l(\De_{n-1}),\] we would obtain a contradiction to the maximality of $\De$ respectively the minimality of $\Gamma.$

    Finally, to show that $\fm_f$ and $f(\fm_f)$ are uniquely determined by $\fm$, we will argue also inductively on the number of segments. We start by showing that if $(A,f)$ is a best matching function then $\De\in \fm_f$ and $\Gamma\in f(\fm_f)$. Note that by property (5) at least one of $n-1,n$ has to be an element of $A$ respectively $f(A)$. If $n\in A$, property (4) implies $f(\De_n)=\Gamma=\De_{n-1}$ by the minimality of $\Gamma$. If on the other hand $n-1\in f(A)$, we let $a'\in A$ be such that $f(a')=n-1$ and assume $\De\notin \fm_f$. By the maximality of $\De$, we have that $l(\De_{a'})\le l(\De)\le l(\Gamma)$. Now property (3) implies $\De_{a'}=\De$, a contradiction.
    We thus can assume without loss of generality that $n\in A$ and $n-1\in f(A)$. Next, we show that every best matching function $(A,f)$ of $\fm$ is equivalent to a best matching function $(A,g)$ of $\fm$ such that $g(n)=n-1$.
   We let $a'\in A$ be such that $n-1=f(a')$. Define \[g\colon A\ra \{1,\ldots,n\}, \,g(a)=\begin{cases}
       n-1&\text{ if }a=n,\\
       f(n)&\text{ if }a=a',\\
        f(a)&\text{ otherwise}.\\
   \end{cases}\]It is clear that $f$ and $g$ are equivalent, it thus remains to show that $g$ is a best matching function.
    Property (1) is trivially satisfied. Property (2) is non-trivial only for $a'$. But since by the maximality of $\De=\De_n$, $l(\De_{a'})\le l(\De_n)$, it follows from property (2) of $f$. 
Property (3) and (4) are trivially true for all $a$ except $a\in\{a',n\}$. Let $\De_b,\, \De_c$ be like in properties (3) and (4). For $a=n$, if \[l(\De_n)\le l(\De_b)\le l(\De_{g(n)})=l(\Gamma),\] the maximality of $\De=\De_n$ implies $\De_n=\De_b$. If 
\[l(\De_n)\le l(\De_c)\le l(\De_{g(n)})=l(\Gamma),\] the minimality of $\Gamma$ implies $\Gamma=\De_c$.
For $a=a'$, if \[l(\De_{a'})\le l(\De_b)\le l(\De_{g(a')})=l(\De_{f(n)}),\] the maximality of $\De=\De_n$ implies \[l(\De_b)\le l(\De_n)\le l(\Gamma)=l(\De_{f(a')})\] and hence by property (3) of $f$, $\De_b=\De_{a'}$. Finally, if \[l(\De_{a'})\le l(\De_c)\le l(\De_{g(a')})=l(\De_{f(n)}),\] we first show $l(\De_c)\ge l(\De_n)$. Indeed, if $l(\De_c)<l(\De_n)$, we would obtain that $l(\De_{a'})\le l(\De_c)<l(\Gamma)=l(\De_{f(a')})$, contradicting property (3) of $f$. Thus $l(\De_n)\le l(\De_c)\le l(\De_{f(n)})$ and hence by property (3) of $f$, $\De_c=\De_{f(n)}$.

Thus we can replace $f$ by an equivalent best matching function sending $n$ to $n-1$ and hence we can restrict $f$ to a best matching function \[f'\colon A\setminus\{n\}\ra\{1,\ldots,n-2\}\] of $\fm-\De_{n-1}-\De_n$. By the induction hypothesis, we know that all best matching functions of $\fm-\De_{n-1}-\De_n$ are equivalent and hence all best matching functions of $\fm$ are equivalent.   
\end{proof}
Let $\fm$ be a multisegment and $\rho\in\scu$ with a corresponding best matching function $(A,f)$. The above lemma allows us to make the following definition. A $\rho$-segment $\De\in \fm-\fm_f$ with $b_\rho(\De)=0$ is called \emph{free} and a $\rho$-segment $\De\in\fm-f(\fm_f)$ with $b_\rho(\De)=-1$ is called \emph{extendable}.
We then set 
\begin{enumerate}
    \item $\dlm(\fm)$ to the number of free $\rho$-segments in $\fm$,
    \item   \[\Dl(\fm)\coloneq \begin{cases}
        0&\text{ if there exists no free }\rho\text{-segment in }\fm,\\
        \fm-\De+\De^-&\text{ otherwise},\\ 
    \end{cases}\]
    where $\De$ is the shortest free $\rho$-segment of $\fm$,
    \item $\Dlm(\fm)$ to the multisegment obtained from $\fm$ by replacing all free $\rho$-segments $\De$ by $\De^-$,
    \item\[\soc(\fm,\rho)\coloneq \begin{cases}
    \fm+[0,0]_\rho&\text{if there exists no extendable } \rho\text{-segment in }\fm,\\
    \fm+\De^+-\De&\text{otherwise,}     
\end{cases}\]
where $\De$ is the longest extendable $\rho$-segment,
    \item $\drm(\fm)\coloneq d_{r,\rho^\lor}(\fm^\lor)$, $\Dr(\fm)\coloneq \mathcal{D}_{r,\rho^\lor}(\fm^\lor)^\lor$, $\Drm(\fm)\coloneq \mathcal{D}_{r,\rho^\lor,max}(\fm^\lor)^\lor$, $\soc(\rho,\fm)\coloneq \soc(\fm^\lor,\rho^\lor)^\lor$.
\end{enumerate}
\begin{lemma}\label{L:help2}
        Let $\fm$ be a multisegment and $\rho\in \scu$. Then any best matching functions with respect to $\rho$ of $\fm$ and $\soc(\fm,\rho)$ are equivalent. If $\Dl(\fm)\neq 0$, its best matching functions with respect to $\rho$ are also equivalent to those of $\fm$.
    \end{lemma}
\begin{proof}
We only show the claim regarding $\soc(\fm,\rho)$, the claim regarding $\Dl(\fm)$ follows analogously.
Let $\fm=\sum_{i=1}^{n-1}\De_i+\De_n^-$ and $\soc(\fm,\rho)=\sum_{i=1}^{n}\De_i$.
We fix a best matching function $(A,f)$ of $\fm$.

Since by construction $\fm_f$ and $f(\fm_f)$ are contained in $\soc(\fm,\rho)$, it is enough by \Cref{L:uniqext} to show that $(A,f)$ satisfies the 5 properties of being a best matching function $\soc(\fm,\rho)$. The properties (1), (2), and (4) of \ref{S:bmf} are trivially satisfied. 

 The properties (3) and (5) of \ref{S:bmf} are satisfied for all $b$ except for $b=n$. 
We start with property (3). Assume there exists $a\in A$ such that \begin{equation}\label{E:c3}
    l(\De_{f(a)})\ge l(\De_n)> l(\De_a).
\end{equation} Firstly, this implies that $\De_n$ cannot be of length $1$, hence we can assume $\De_n^-$ to be non-zero. Thus we have $l(\De_n^-)\ge l(\De_a)$.
In this case, we can apply property (3) of $f$ to $a$ and $n$ since $f$ is a best matching function of $\fm$ and $\De_n^-$ is an extendable segment. Thus we obtain that $\De_{f(a)}=\De_n^-$, contradicting \Cref{E:c3}.

For property (5), if there exists an extendable segment $\De_c$ such that $l(\De_c)\ge l(\De_n)$, we would obtain that $l(\De_c)>l(\De_n^-)$, which would contradict the maximality of $\De_n^-$.
\end{proof}
\begin{lemma}\label{L:invers}
    Let $\fm$ be a multisegment and $\rho\in\scu$. Then
    \[\Dl(\soc(\fm,\rho))=\fm.\] If moreover $\Dl(\fm)\neq 0$, \[\soc(\Dl(\fm),\rho)=\fm\] and similarly for $\Dr$ and $\soc(\rho,\cdot)$.
\end{lemma}
\begin{proof}
    We only prove,
    \[\Dl(\soc(\fm,\rho))=\fm,\] the other claims follow analogously. Then the claim is equivalent to the following claim. If $\De$ is the shortest free $\rho$-segment in $\soc(\fm,\rho)$, the segment $\De^-$ is the longest extendable $\rho$-segment in $\fm$. 
    
    Let $\De^-$ be the longest extendable $\rho$-segment in $\fm$. By \Cref{L:help2} $\De$ is a free $\rho$-segment of $\soc(\fm,\rho)$. If $\De'$ would be a shorter free $\rho$-segment of $\soc(\fm,\rho)$, it would be a free $\rho$-segment of $\fm$ by \Cref{L:help2} and $l(\De')\le l(\De^-)$. This would contradict property (5) of a best matching function.
\end{proof}
Finally, we show that the four maps indeed take aperiodic multisegments to aperiodic multisegments.
\begin{lemma}\label{L:deraperiodic}
    If $\fm$ is aperiodic, \[\Dl(\fm),\Dr(\fm), \soc(\fm,\rho),\soc(\rho,\fm)\] are aperiodic.
\end{lemma}
\begin{proof}
Write $\fm=\sum_{i=1}^n\De_i$ and let $(A,f)$ be a best matching function of $\fm$ with respect to $\rho$.
We only show the claim for $\Dl(\fm)$, the others follow completely analogously.
 We assume by contradiction otherwise, \emph{i.e.} $\Dl(\fm)$ contains a multisegment equivalent to 
\[[a,0]_\rho+\ldots+[a+e(\rho)-1,e(\rho)-1 ]_\rho,\] for a suitable $a\in \ZZ_{<0}$.
If $[a+e(\rho)-1,e(\rho)-1 ]_\rho$ appears in $\fm$, $\fm$ cannot be aperiodic, since by the construction of $\Dl(\fm)$, $\fm$ also has to contain \[[a,0]_\rho+\ldots+[a+e(\rho)-2,e(\rho)-2 ]_\rho.\] Thus we assume now that $[a+e(\rho)-1,e(\rho)-1 ]_\rho$ does not appear in $\fm$ and will produce a contradiction.

Let $\De$ be a the shortest free $\rho$-segment in $\fm$, \emph{i.e.} the segment such that $\Dl(\fm)=\fm+\De^--\De$. Since we assumed that $[a+e(\rho)-1,e(\rho)-1 ]_\rho$ does not appear in $\fm$, we have $\De^-=[a+e(\rho)-1,e(\rho)-1 ]_\rho$ and hence $\De=[a-1,0]_\rho$. Since $\fm$ contains then both $[a,0]_\rho$ and $ [a-1,0 ]_\rho$ and \[l([a,0]_\rho)<l([a-1,0 ]_\rho),\] the minimality of $\De$ implies that $[a,0]_\rho\in \fm_f$. Let $a'\in A$ be such that $\De_{a'}=[a,0]_\rho$. By property (2) and (3) of \ref{S:bmf} we obtain that \[2-a=l([a-1,0)]_\rho)>l(\De_{f({a'})})\ge l([a,0]_\rho)=1-a\] and hence $l(\De_{f({a'})})=-a+1$. This implies that \[\De_{f({a'})}=[a-1,-1]_\rho=[a+e(\rho)-1,e(\rho)-1 ]_\rho,\] a contradiction.
\end{proof}
\begin{lemma}\label{L:help4}
    Let $\fm$ and $\ain{k}{0}{\dlm(\fm)}$. Then $\Dl^k(\fm)$ is the multisegment obtained from $\fm$ by shortening the shortest $k$ free $\rho$-segments in $\fm$ by one on the right.
\end{lemma}
\begin{proof}
     Note that by \Cref{L:help2} we have that a best matching function of $\fm$ is a best matching function of $\Dl^k(\fm)$ and from this the claim follows straightforwardly. 
\end{proof}
\begin{corollary}\label{L:redgood}
    Let $\fm$ be an aperiodic multisegment and $\rho\in\scu$. Then \[\mathcal{D}_{r,\rho,max}(\fm)^-=\mathcal{D}_{r,\rho v_\rho^{-1},max}(\fm^-).\] Set $l=d_{r,\rho v_\rho^{-1}}(\fm^-)$ and $k$ to the largest positive integer such that \[\fm^1-(l+k)[0,0]_\rho+l[-1,-1]_\rho\] contains at least as many copies of $[-1,-1]_\rho$ as of $[0,0]_\rho$.
    Then \[d_{r,\rho}(\fm)=l+k,\,\Dlm(\fm)^1=\fm^1-(l+k)[0,0]_\rho+l[-1,-1]_\rho.\]
\end{corollary}
\begin{proof}
Write $\fm=\sum_{i=1}^n\De_i$ such that there exists $\ain{m}{1}{n}$ with $l(\De_i)=1$ if and only if $i>m$.
Let $(A,f)$ be a best matching function of $\fm$ with respect to $\rho$. 
Set $A^-\coloneq A\cap\{1,\ldots,m\}$. Then the restriction of $f$
\[f\colon A^-\ra\{1,\ldots,m\}\] is a best matching function of $\fm^-$ with respect to $\rho v_\rho^{-1}$. Note that $(A^-,f)$ is well defined since $l(\De_a)\le l(\De_{f(a)})$ and hence if $l(\De_a)>1$, then also $l(\De_{f(a)})>1$. Thus $a\in A,a\le m$ implies $f(a)\le m$.
Moreover, by \Cref{L:uniqext} it is enough to show that the 5 properties of \ref{S:bmf} are satisfied, which is easy to check. 

    By \Cref{L:help4}, we can compute the above quantities in two steps. First dealing with the free segments of length $1$ and then with the ones of length greater $1$. 
    Namely, let $k$ be the number of free segments with respect to $\rho$ of length $1$ in $\fm$. Then
    \[\Dl^k(\fm)^-=\fm^-\] and $\Dl^k(\fm)$ has no free segments with respect to $\rho$ of length $1$. Denote the number of free segments in $\Dl^k(\fm)$ with respect to $\rho$ by $l$. Now 
    $l=d_{r,\rho v_\rho^{-1}}(\fm^-)$ since $(f^-,A)$ is a best matching function of $\fm^-$. Moreover, we obtain by the same reason that \[\Dl^{d_{r,\rho}(\fm)}(\fm)^-=\Dl^l(\fm)^-=\mathcal{D}_{r,\rho v_\rho^{-1}}^l(\fm^-)\] and \[\Dlm(\fm)^1=\fm^1-(k+l)[0,0]_\rho+l[-1,-1]_\rho.\]
    Finally, if $\Dlm(\fm)^1$ would contain more copies of $[0,0]_\rho$ than of $[-1,-1]_\rho$, $\Dlm(\fm)$ would have to contain a free segment with respect to $\rho$, since every best matching function of $\Dlm(\fm)$ pairs up the segments ending in $0$ and ending in $-1$. Thus if there are more copies of $[0,0]_\rho$ than of $[-1,-1]_\rho$ in $\Dlm(\fm)^1$, there needs to exist a free segment in $\Dlm(\fm)$. This however contradicts \Cref{L:help4}.
.\end{proof}
\subsubsection{  }
We will now make the relation between derivatives of multisegments and derivatives of representations explicit.
\begin{theorem}\label{T:derseg}
    For $\fm$ an aperiodic multisegment and $\rho\in \scu$, we have
    \[\Dl(\Z(\fm))=\Z(\Dl(\fm)),\, \Dr(\Z(\fm))=\Z(\Dr(\fm)),\]
    \[\dlm(\Z(\fm))=\dlm(\fm),\, \drm(\Z(\fm))=\drm(\fm)\]
    and
    \[\Dlm(\Z(\fm))=\Z(\Dlm(\fm)),\, \Drm(\Z(\fm))=\Z(\Drm(\fm)).\]
\end{theorem}
We will take great care in the proof not to use the surjectivity of $\Z$ for representations of $\square$-irreducible cuspidal support over $\fl$, since we want to eventually use the above theorem to give a new proof of it.
We also remark that the theorem is already known over $\ql$, \emph{cf.} \cite[Theorem 2.2.1]{Jader} and \cite[Theorem 7.5]{Mder}.
Let $\fm=\sum_{i=1}^n\De_i$ be an aperiodic multisegment over $\fl$ and $\rho$ a supercuspidal representation. We call a lift $\tfm=\sum_{i=1}^n\widetilde{\De_i}$ of $\fm$, $\widetilde{\De_i}$ a lift of $\De_i$ for all $\ain{i}{1}{n}$, \emph{$\rho$-right-derivative compatible} if there exists a lift $\trho$ of $\rho$ such that for all $\rho$-segments in $\fm$ ending in $0$ respectively $-1$, their lift in $\tfm$ ends in $0$ respectively $-1$ and is a segment over $\trho$. Similarly,  a lift $\tfm$ of $\fm$ is said to be \emph{$\rho$-left-derivative compatible} if there exists a lift $\trho$ of $\rho$ such that for all $\rho$-segments in $\fm$ starting in $0$ respectively $1$, their lift in $\tfm$ starts in $0$ respectively $1$ and is a segment over $\trho$. If $\rho\in \scu$, there exists for all $\fm\in \Ms(\rho)$ such lifts since $\rho$ is supercuspidal. The following is then easy to see.
\begin{lemma}\label{L:redandder}
    Let $\tfm$ be a $\rho$-right-derivative compatible lift of an aperiodic multisegment $\fm$ and $\trho$ the corresponding lift of $\rho$. Then $\rl(\mathcal{D}_{r,\trho}(\tfm))=\Dl(\fm)$. More precisely, if $(A,f)$ is a best matching function with respect to $\trho$ of $\tfm$, $(A,f)$ is a best matching function of $\fm$ with respect to $\rho$.

    A similar statement holds for $\rho$-left-compatible lifts.
\end{lemma}
\begin{lemma}\label{L:maxrightcomp}
Let $\rho\in\scu(G_m)$, $\tfm$ be a $\rho$-right-derivative compatible lift of an aperiodic multisegment $\fm$ and $\trho$ the corresponding lift of $\rho$. Let moreover $\alpha=(\deg(\Dlm(\fm)),md_{r,\rho})$ and $\pi\in \Irr$. Then the multiplicity of \[\pi\otimes \rho^{d_{r,\rho}}\] in \[\rl(r_\alpha(\Z(\Dlm(\tfm))\times \trho^{d_{r,\rho}}))\] is the multiplicity of $\pi$ in 
        $\rl(\Z(\Dlm(\tfm)))$.
        An analogous statement holds for $\rho$-left-derivative compatible lifts.
\end{lemma}
\begin{proof}
    We will only deal with the case of the $\rho$-right-compatible lift $\tfm$. Recall that parabolic restriction commutes with reduction mod $\ell$ on the level of Grothendieck groups. Now since $\tfm$ is a $\rho$-right-compatible lift $\tfm$ and \Cref{T:derseg} is known over $\ql$, there exists no lift $\trho'$ of $\rho$ such that
    \[\mathcal{D}_{r,\trho'}(\Z(\Dlm(\tfm)))\neq 0.\]
    Thus \Cref{L:dervanish} and the Geometric Lemma show that the $(G_{md_{r,\rho}},\rho^{d_{r,\rho}})$-invariant part of \[[\rl(r_\alpha(\Z(\Dlm(\tfm))\times \trho^{d_{r,\rho}}))]\] is \[[\rl(\Z(\Dlm(\tfm)))\otimes \rho^{d_{r,\rho}}].\] From this the claim follows immediately.
\end{proof}
Before we start with the proof of \Cref{T:derseg}, we need the following lemmas.
    \begin{lemma}\label{L:rede}
        Let $\Z(\fn)$ be a residually-degenerate representation and $\rho\in \scu$ of $G_m$. Let $d_i$ be the multiplicity of $[i,i]_\rho$ in $\fn,\, i\in\{-1,0\}$. Then $\Dl(\Z(\fn))\neq 0$ if and only if $d_{-1}<d_0$. Similarly,
        $\mathcal{D}_{l,\rho v_\rho^{-1}}(\Z(\fn))\neq 0$ if and only if $d_{-1}>d_0$.
    \end{lemma}
    \begin{proof}
    Via \Cref{L:H2} we can reduce the lemma to the case where $\fn$ has cuspidal support in $\mathbb{N}(\ZZ[\rho])$. Thus $\fn$ is an aperiodic, banal multisegment consisting of segments of length $1$. Let $\widetilde{\fn}$ be a $\rho$-right-derivative compatible lift of $\fn$. Since $\fn$ is banal, it follows that $\widetilde{\fn}$ can be chosen such that $\rl(\Z(\widetilde{\fn}))=\Z(\fn)$ by \cite[§6]{banal}. 
Let $d=\min(d_{-1},d_0)$ and write \[\fn=\overbrace{[-1,-1]_\rho+\ldots+[-1,-1]_\rho}^{d_{-1}}+\overbrace{[0,0]_\rho+\ldots+[0,0]_\rho}^{d_{0}}+\ldots.\] 
Setting $A=\{d_{-1}+1,\ldots, d_{-1}+d\}$ and $f(d_{-1}+i)=i$ for all $\ain{i}{1}{d}$ gives a best matching function $(A,f)$ of $\fn$ with respect to $\rho$. We let $\trho$ be the lift of $\rho$ in the definition of $\widetilde{\fn}$.
    
Thus $d_{-1}<d_0$ if and only if $\Dl(\fn)\neq 0$ which is equivalent to $\mathcal{D}_{r,\trho}(\widetilde{\fn})\neq 0$ by \Cref{L:redandder}. Since the lemma holds over $\ql$, we have by \Cref{L:dervanish} that this is equivalent to the existence of $\widetilde{\tau}\in \Irr$ such that $\widetilde{\tau}\otimes \trho$ appears in $r_{(\deg\fn-m,m)}(\Z(\widetilde{\fn}))$. Because parabolic restriction commutes with reduction mod $\ell$, we obtain from \Cref{L:dervanish} that this is equivalent to $\Z(\fn)$ admitting a right $\rho$-derivative.
The second claim follows from an analogous argument.
    \end{proof}
    \begin{lemma}\label{L:smallestcase}
        Let $\pi$ be a non-residually degenerate representation of $G_n$ such that there exists $\rho\in \scu(G_m)$ with
        \[[r_{(m,n-m)}(\pi)]\ge [\rho v_\rho^{-1}\otimes \Z(\fn)]\]
        and $\fn$ is residually degenerate.
        Then \[[r_{(n-m,m)}(\pi)]\ge [\Z(\fn+[-1,-1]_\rho-[0,0]_\rho)\otimes \rho].\]
    \end{lemma}
    \begin{proof}
        Using \Cref{L:H2} it is easy to see that it is enough to show the claim for $\pi$ having cuspidal support contained in $\mathbb{N}(\ZZ[\rho])$. 
        If $[\pi]\le \I(\fm)$ with $\fm$ a multisegment containing at least one segment not of length $1$, \Cref{T:N5} shows that above situation can happen if and only if $\fm=\fn+[-1,0]_\rho-[0,0]_\rho$ and $\pi\cong \Z(\fm)$. In this case it is also easy to see that the only $(\deg(\fn),m)$-degenerate representation appearing in the parabolic restriction $r_{(\deg(\fn),m)}(\I(\fm))$ is $\Z(\fn+[-1,-1]_\rho-[0,0]_\rho)\otimes \rho$. Since the parabolic restriction of $\pi$ must contain a $(\deg(\fn),m)$-degenerate representation, we are done in this case.

        We thus assume that $\pi$ does not appear as a subquotient of such a $\I(\fm)$. We write $\fn'\coloneq \fn+[-1,-1]_\rho$ and let $d_i$ be the multiplicity of $[i,i]_\rho$ in $\fn'$, $\ain{i}{-1}{o(\rho)-2}$. Choose a lift $\trho$ of $\rho$ to $\ql$ and set \[\widetilde{\fn}\coloneq d_{-1}[-1,-1]_\trho+d_0[0,0]_\trho+\ldots+d_{o(\rho)-2}[o(\rho)-2,o(\rho)-2]_\trho.\] Then $\widetilde{\fn}$ is a $\rho v_\rho^{-1}$-left-derivative compatible lift of $\fn'$ and $[\pi]\le \I(\fn')=\rl(\I(\widetilde{\fn})).$ Let $\pi'\in \Irr$ be such that $[\widetilde{\pi}]\le \I(\widetilde{\fn})$ and $[\pi]\le \rl(\widetilde{\pi}).$ Since over $\ql$ we already have the subjectivity of $\Z$, we have $\widetilde{\pi}\cong \Z(\widetilde{\fk})$ for some multisegment $\widetilde{\fk}$.
        Since \[[\pi]\le\rl(\widetilde{\pi})\le \rl(\I(\widetilde{\fk}))=\I(\rl(\widetilde{\fk})),\]
        the observations of the previous paragraph imply that $\widetilde{\fk}=\widetilde{\fn}$.
        By \Cref{L:rede} and \Cref{L:dervanish} we have that $d_{-1}>d_0$ and
        \[[\rho v_\rho^{-1}\otimes (\rho v_\rho^{-1})^{d_{-1}-d_0-1}\otimes \Z(\fn'-(d_{-1}-d_0)[-1,-1]_\rho)]\le \]\[\le [r_{((m,(d_{-1}-d_0-1)m,\deg(\fm)-(d_{-1}-d_0)m)}(\rho v_\rho^{-1}\otimes \Z(\fn))]\le\]\[\le  [r_{((m,(d_{-1}-d_0-1)m,\deg(\fm)-(d_{-1}-d_0)m)}(\pi)].\]
        Thus
        \[[(\rho v_\rho^{-1})^{d_{-1}-d_0}\otimes \Z([\fn'-(d_{-1}-d_{0})[-1,-1]_\rho)]\le  [r_{(d_{-1}-d_{0})m,\deg(\fm)-(d_{-1}-d_{0})m)}(\pi)].\]
        Since we already know \Cref{T:derseg} to be true over $\ql$, we know by \Cref{C:der} that \[(\trho v_\trho^{-1})^{d_{-1}-d_{0}}\otimes \Z([\widetilde{\fn}-(d_{-1}-d_{0})[-1,-1]_\trho)\] appears with multiplicity $1$ in $r_{(d_{-1}-d_{0})m,\deg(\fm)-(d_{-1}-d_{0})m)}(\Z(\widetilde{\fn}))$. Furthermore, since parabolic restriction commutes with reduction mod $\ell$, \Cref{C:der} also shows that \begin{equation}\label{E:eul}(\rho v_\rho^{-1})^{d_{-1}-d_{0}}\otimes \Z(\fn'-(d_{-1}-d_{0})[-1,-1]_\rho)\end{equation} appears in \begin{equation}\label{E:longis}r_{(d_{-1}-d_{0})m,\deg(\fm)-(d_{-1}-d_{0})m)}(\rl(\Z(\widetilde{\fn}))).\end{equation}By \Cref{L:maxrightcomp} and \Cref{T:liftgood} (\ref{E:eul}) appears with multiplicity $1$ in (\ref{E:longis}).
        By \Cref{L:rede} (\ref{E:eul}) appears also in the parabolic restriction of $\Z(\fn')$, which appears in $\rl(\Z(\widetilde{\fn}))$ by \Cref{T:liftgood}.
        Thus $\pi\cong \Z(\fn')$, a contradiction to the fact that $\pi$ is not residually degenerate.
    \end{proof}
\begin{lemma}\label{L:redderm}
    Let $\fm$ be an aperiodic multisegment, $\rho\in\scu(G_m)$ and define the composition $\alpha\coloneq\alpha_{\fm,\rho}= (\mu_{\overline{\Dlm(\fm)}},m\dlm(\fm))$. Then
    \[[r_\alpha(\Z(\fm))]\ge [\st(\Dlm(\fm))\otimes \rho^{\dlm(\fm)}].\]
\end{lemma}
\begin{proof}
    We argue by induction on $t$, the maximal length of a segment in $\fm$.
    Note that we have by \Cref{L:S33}
    \[[r_{(\deg (\fm^-),\deg(\fm^1))}(\Z(\fm))]\ge [\Z(\fm^-)\otimes \Z(\fm^1)]\]
    and hence by the induction hypothesis, we know that 
    \[[r_{(\alpha_{\fm^-,\rho v_\rho^{-1}},\deg(\fm^1))}(\Z(\fm))]\ge [\st(\mathcal{D}_{r,\rho v_\rho^{-1},max}(\fm^-))\otimes (\rho v_\rho^{-1})^{d_{r,\rho v_\rho^{-1}}}\otimes \Z(\fm^1)].\]
    Let $\pi\in \Irr_{N},\, N\coloneq m\cdot d_{r,\rho v_\rho^{-1}}(\fm^-)+\deg(\fm_1),\, l\coloneq d_{r,\rho v_\rho^{-1}}(\fm^-)$ be such that 
    \begin{equation}\label{E:red1}[r_{(\deg \fm-N,N)}(\Z(\fm))]\ge [\st(\mathcal{D}_{r,\rho v_\rho^{-1},max}(\fm^-))\otimes \pi],\end{equation}
    \begin{equation}\label{E:smallestcase}[r_{(N-\deg(\fm^1),\deg (\fm_1))}(\pi)]\ge [(\rho v_\rho^{-1})^{l}\otimes \Z(\fm^1)].\end{equation}
    From \Cref{T:N5} it follows that $ml\le \deg \fm_1$ and the maximal partition $\alpha'$ such that $\pi$ is $\alpha'$-degenerate is \[(\deg (\fm_1),ml).\]
    Since we know from (\ref{E:smallestcase}) that
    \[[r_{(m,\ldots,m,\deg (\fm_1))}(\pi)]\ge [\rho v_\rho^{-1}\otimes\ldots \rho v_\rho^{-1}\otimes \Z(\fm^1)]\]
    we can just apply \Cref{L:smallestcase} $l$-times and obtain     \[[r_{(\deg (\fm_1),m,\ldots,m)}(\pi)]\ge [\Z(\fm^1+l[-1,-1]_\rho-l[0,0]_\rho)\otimes \rho\otimes\ldots\otimes \rho].\]
    Now since $o(\rho)>1$ and $\rho$ is supercuspidal, this implies \begin{equation}\label{E:red2}[r_{(\deg (\fm_1),N-\deg(\fm^1))}(\pi)]\ge [ \Z(\fm^1+l[-1,-1]_\rho-l[0,0]_\rho)\otimes \rho^l].\end{equation}
 Write $\fn=\fm^1+l[-1,-1]_\rho-l[0,0]_\rho$.
Combining (\ref{E:red1}) and (\ref{E:red2}) shows that
    \begin{equation}\label{E:red4}[r_{(\deg \fm-N,\deg(\fm_1),N-\deg(\fm^1))}(\Z(\fm))]\ge [\st(\mathcal{D}_{r,\rho v_\rho^{-1},max}(\fm^-))\otimes \Z(\fn)\otimes \rho^l].\end{equation}
    Write now $\fn=\fn'+k[0,0]_\rho$ such that $\fn'$ contained at least as many copies of $[-1,-1]_\rho$ as of $[0,0]_\rho$.
    Applying \Cref{L:rede} $k$-times in combination with \Cref{L:dervanish}, we see that\[[r_{(\deg(\fn'),m,\ldots,m)}(\Z(\fn))]\ge [\Z(\fn')\otimes \rho\otimes\ldots\otimes \rho]\]
    or equivalently
\begin{equation}\label{E:red3}
    [r_{(\deg(\fn'),km)}(\Z(\fn))]\ge [\Z(\fn')\otimes \rho^k].
\end{equation} 
Combining (\ref{E:red3}) and (\ref{E:red4}), we obtain that 
\[[r_{(\deg \fm-N,\deg(\fn'),N-\deg(\fn'))}(\Z(\fm))]\ge [\st(\mathcal{D}_{r,\rho v_\rho^{-1},max}(\fm^-))\otimes \Z(\fn')\otimes \rho^{l+k}].\]
 Now $l+k$ is the number of free $\rho$-segments in $\fm$ since $l$ is by \Cref{L:redgood} the number of free $\rho$-segments of length greater than $1$ and $k$ is the number of free $\rho$-segments of length precisely $1$. Moreover, \Cref{L:redgood} also shows that \[\Dlm(\fm)^-=\mathcal{D}_{r,\rho v_\rho^{-1},max}(\fm^-),\,\Dlm(\fm)^1=\fm^1-(k+l)[0,0]_\rho+l[1,1]_\rho.\]
Thus \[\st(\mathcal{D}_{r,\rho v_\rho^{-1},max}(\fm^-))\otimes \Z(\fn')=\st(\Dlm(\fm)),\, \rho^{l+k}=\rho^{\dlm(\pi)},\]
which finishes the induction step.    
\end{proof}
\begin{proof}[Proof of \Cref{T:derseg}]
We will only proof $\Dlm(\Z(\fm))=\Z(\Dlm(\fm))$, the other claims follow then quickly using \Cref{C:der} and \Cref{L:derdual}.

We will show that $\Z(\fm)=\mathrm{soc}(\Z(\Dlm(\fm))\times \rho^{\dlm(\fm)})$. To do so we first can reduce the claim via \Cref{L:H2} to the assumption that $\fm$ has cuspidal support contained in $\mathbb{N}(\ZZ[\rho])$. We choose a $\rho$-right-derivative compatible lift $\tfm$ of $\fm$. Then $\rl(\Dlm(\tfm))=\Dlm(\fm)$ and from the $\ql$-case we already know that \[\Z(\Dlm(\tfm))=\Dlm(\Z(\tfm)).\]
We set $\pi\coloneq \mathrm{soc}(\Z(\Dlm(\fm))\times \rho^{\dlm(\fm)})$ and $\alpha\coloneq (\deg(\Dlm(\fm)),m\cdot\dlm(\fm))$. 
By \Cref{T:liftgood} \[[\pi]\le [\mathrm{soc}(\Z(\Dlm(\fm))\times \rho^{\dlm(\fm)})]\le [\rl(\Z(\mathcal{D}_{r,\trho,max} (\tfm))\times \trho^{\dlm(\fm)})]\]
and 
\[[\Z(\fm)]\le [\rl(\Z(\tfm))]\le [\rl(\Z(\mathcal{D}_{r,\trho,max} (\tfm))\times \trho^{\dlm(\fm)})].\]
Moreover, Frobenius reciprocity implies $r_\alpha(\pi)$ contains \begin{equation}\label{E:containswith1}
    \Z(\mathcal{D}_{r,\rho,max}(\fm))\otimes \rho^{\dlm(\fm)}.
\end{equation}
and \Cref{L:redderm} implies there exists $\pi'\in \Irr$ which is $\mu_{\overline{\Dlm(\fm)}}$-degenerate such that \[[r_\alpha(\Z(\fm))]\ge [\pi'\otimes \rho^{\dlm(\fm)}].\]
\Cref{L:maxrightcomp} shows that \[[\pi']\le \rl(\Z(\mathcal{D}_{r,\trho,max}(\tfm)))\le\rl(\I(\mathcal{D}_{r,\trho,max}(\tfm)))=\I(\Dlm(\fm)).\] Hence \Cref{T:N5} shows that $\pi\cong'\Z(\Dlm(\fm))$.
 But \Cref{L:maxrightcomp} also shows that (\ref{E:containswith1}) appears with multiplicity $1$ in \begin{equation}r_\alpha(\rl(\Z(\mathcal{D}_{r,\trho,max} (\tfm))\times \trho^{\dlm(\fm)})).\end{equation}
 Therefore $\pi\cong \Z(\fm)$.
\end{proof}
\section{Applications of \texorpdfstring{$\rho$}{p}-derivatives}\label{S:appd}
In this section we give several applications of the above developed machinery.
\subsection{ } The first consequence is the surjection of the map $\Z$ for representations with $\square$-irreducible cuspidal support over $\fl$.
\begin{theorem}
    Let $\pi\in \Irr_n(\fl)$ with $\square$-irreducible cuspidal support. Then there exists an aperiodic multisegment $\fm$ such that $\pi\cong \Z(\fm)$.
\end{theorem}\begin{proof}
We argue by induction on $n$.
    Let $\pi$ be an irreducible representation of $G_n$ with $\square$-irreducible cuspidal support and let $\pi'\in \Irr$ and $\rho$ be a $\square$-irreducible cuspidal representation such $\pi\hra\pi'\times\rho$. By the induction hypothesis, we know that there exists an aperiodic multisegment $\fm'$ such that $\pi'\cong \Z(\fm')$. By \Cref{T:derseg} we have then that \[\pi\cong \mathrm{soc}(\pi'\times \rho)=\Z(\mathrm{soc}(\fm',\rho)).\]
\end{proof}
\subsection{ }
 We will now give an explicit description of the Aubert-Zelevinsky dual $\pi^*$ using the theory of derivatives.
Recall the Grothendieck group $\RR=\bigcup_{n\in \mathbb{N}}\RR_n$. 
In \cite{Zel} the Aubert-Zelevinsky involution \[\bd\colon\RR\rightarrow \RR,\]
which sends $[\pi]\in\Irr_n$ to
\[\bd([\pi])\coloneq \sum_\alpha (-1)^{r(\alpha)}[\mathrm{Ind}_{\alpha}\circ r_\alpha(\pi)]\]
was introduced,
where the sum is over all compositions of $n$ and $r(\alpha)$ is the number of elements of $\alpha$.
For $\Pi$ a representation of finite length, we write \[\bd(\Pi)=\sum_{\pi\in \Irr}k_\pi[\pi],\, k_\pi\in \mathbb{Z},\] where all but finitely many $k_\pi$ are non-zero. We say $\pi\in \Irr$ appears in $\bd(\Pi)$ if $k_\pi\neq 0$.
The map $\bd$ satisfies the following properties.
\begin{theorem}\label{T:invol}
    Let $\bd$ be the Aubert-Zelevinsky involution.
    \begin{enumerate}
        \item If $\pi_1,\pi_2$ are irreducible representations then
    \[\bd([\pi_1\times\pi_2])=\bd([\pi_1])\times \bd([\pi_2]).\]
    \item If $\beta$ is a composition of $n$ and $\pi$ an irreducible representation then \[\bd([r_\beta(\pi)])=\mathrm{Ad}(w_\beta)\circ r_\beta(\bd([\pi)])),\]
    where $w_\beta$ is the longest element of the Weyl group $W(\beta,(n))$, \emph{cf.} \Cref{S:geolem}.
    \item If $\pi$ is irreducible, there exists a unique irreducible representation $\pi^*$ with the same cuspidal support as $\pi$ appearing in  
    $\bd([\pi])$ with sign $(-1)^{r(\pi)}$, where $r(\pi)$ is the length of the cuspidal support of $\pi$.
    \end{enumerate}
\end{theorem}
\begin{proof}
    For (1) and (2) see \cite[Theorem 1.7]{Aub} and \cite[Proposition A.2]{banal}.
    For (3) see \cite{MoeglinetJ1986} for the case $R=\ql$ and \cite[Theorem 2.5]{MSI} for the general case $R\in\{\fl,\ql\}$.
\end{proof}
\begin{theorem}\label{T:mainin}
    Let $\pi\in \Irr_n$ and $\rho\in \scu$. Then \[\Dl(\pi)^*\cong \Dr(\pi^*).\]
\end{theorem}
\begin{proof}
If $\pi=\rho^k$, the claim follows from \Cref{T:invol}(1). Thus we can assume that $\Dlm(\pi)\neq 0$.
     By Frobenius reciprocity \[[r_{(n-\dlm(\pi)m,\dlm(\pi)m)}(\pi)]\ge [\Dlm(\pi)\otimes \rho^{\dlm(\pi)}]\] and hence by \Cref{T:invol} (2)
     \[r_{(\dlm(\pi)m,n-\dlm(\pi)m)}(\bd([\pi]))]\] contains the representation $\rho^{\dlm(\pi)}\otimes \Dlm(\pi)^*$ with a non-zero coefficient $(-1)^{\dlm(\pi)}(-1)^{r(\Dlm(\pi))}$.
    Recall that $(-1)^{r(\pi)}[\pi^*]$ is the unique constituent of
$\bd([\pi])$ with cuspidal support the same as $\pi$, therefore it follows that
\[(-1)^{r(\pi)}[r_{(\dlm(\pi)m,n-\dlm(\pi)m)}(\pi^*)]\ge\]\[\ge   (-1)^{\dlm(\pi)}(-1)^{r(\Dlm(\pi))}[\rho^{\dlm(\pi)}\otimes \Dlm(\pi)^*]=\]\[=(-1)^{r(\pi)}[\rho^{\dlm(\pi)}\otimes \Drm(\pi^*)],\] where the last equality follows from the induction hypothesis.
The claim then follows from \Cref{C:der}.
\end{proof}
If $\fm$ is an aperiodic multisegment $\fm$, we set $\fm^*$ to the aperiodic multisegment such that $\Z(\fm)^*=\Z(\fm^*)$.

For $\fm$ an aperiodic multisegment with $\square$-irreducible cuspidal support we can thus offer the following algorithm to compute $\fm^*$, where $\fm^*$ is the aperiodic multisegment such that $\Z(\fm)^*=\Z(\fm^*)$. It works recursively on $\deg(\fm)$.
Let $\rho\in \scu$ such that $\Dl(\fm)\neq 0$. We then compute $\fm^*$ by first computing $\Dl(\fm)$. From this we can already compute $\Dl(\fm)^*$ recursively. But by \Cref{T:invol} and \Cref{T:derseg} $\Dl(\fm)^*=\Dr(\fm^*)$. By \Cref{L:derivatives} and \Cref{T:derseg} \[\soc(\rho,\Dl(\fm)^*)=\fm^*.\] Note that we thus rediscover the algorithm of computing the Aubert-Zelevinsky dual of \cite{LTV}.
\subsection{ Godement-Jacquet local factors}\label{S:GJ}
In this section we will compute the local $L$-factors associated to an irreducible smooth representation over $\fl$ and show that the map
\[\Cc\colon \Irr_n\iso \{\Cc-\text{parameters of length }n\} \]
defined in \cite{Ccor} respects these local $L$-factors.
We start by proving the following lemma, which in the case $R=\ql$ was proved in \cite[Proposition 2.3]{JL2}.
\begin{lemma}\label{L:Lind}
    Let $\sigma_1,\ldots,\sigma_k$ be irreducible representations of $G_{n_1},\ldots,G_{n_k}$ over $R$. Then
    \[L(\sigma_1\times\ldots\times \sigma_k,T)=\prod_{i=1}^k L(\sigma_i,T)\]
    and 
    \[\gamma(T,\sigma_1\times\ldots\times \sigma_k,\psi)=\prod_{i=1}^k \gamma(T,\sigma_i,\psi).\]
\end{lemma}
\begin{proof}
We can mimic the proof of \cite[Proposition 2.3]{JL2} more or less \emph{muta mutandis}.
We start by showing the inclusion \[\mathcal{L}(\sigma_1\times\ldots\times \sigma_k)\subseteq\prod_{i=1}^k\mathcal{L}(\sigma_i).\]
Set $\sigma\coloneq \sigma_1\times\ldots\times \sigma_k,\, \alpha=(n_1,\ldots,n_k)$ and $P\coloneq P_\alpha$ with Levi-component $M$ and unipotent component $U$.
Recall that \[\delta_P(m)=\prod_{i=1}^k\lvert \det\nolimits'(m_i)\lvert ^{d\delta_i},\, \delta_i\coloneq -n_1-\ldots-n_{i-1}+n_{i+1}+\ldots +n_k.\]
Moreover, recall that as in \cite[Proposition 2.3]{JL2} the matrix coefficients of $\sigma$ are exactly given by integrals of the form
\begin{equation}\label{E:matrixcoefficient}f(g)= \int\displaylimits_{P\bs G_n}H(g'g,g')\ddd g',\end{equation}\[H(g_1,g_2)=h^\lor(g_1)(h(g_2)), \, h\in\sigma_1\times\ldots\times \sigma_k,\, h^\lor\in \sigma_1^\lor\times\ldots \times \sigma_k^\lor.\]
Note that $H$ is smooth, it satisfies for all $p\in P,\, g_1,g_2\in G_n$ \[H(pg_1,pg_2)=\delta_P(p)H(g_1,g_2)\] and $m\mapsto H(mg_1,g_2)$ is a matrix coefficient of $\sigma_1\otimes\ldots\otimes \sigma_k\otimes \delta_P^\frac{1}{ 2}$.
Furthermore, given such $H$, the function \[f(g)= \int\displaylimits_{P\bs G_n}H(g'g,g')\ddd g'=\int\displaylimits_{(K_n\cap P )\bs K_n}H(kg,k)\ddd k\] is a matrix coefficient of $\sigma$.
Fix now $\phi\in \sch$ and as in \cite[Proposition 2.3]{JL2} we see that the coefficient of $T^N$ in $Z(\phi,Tq^{-\frac{dn-1}{2}},f)$ is 
\begin{equation}\label{E:niceform}\int\displaylimits_{(K_n\cap P )\bs K_n}\int\displaylimits_{(K_n\cap P )\bs K_n}\int\displaylimits_{P(N)}\phi(k^{-1}pk')\delta_P(p)^{-1}\lvert \det\nolimits'(p)\lvert^{\frac{dn-1}{2}}H(pk',k)\dr p\ddd k\ddd k'\end{equation}
We write  for\[m=(m_1,\ldots m_k),\, u=\begin{pmatrix}
    1_{n_1}&\ldots &u_{i,j}\\
    0&\ddots&\vdots\\
    0&0&1_{n_k}
\end{pmatrix},\,  p(m,u)\coloneq\begin{pmatrix}
    m_1&\ldots &u_{i,j}\\
    0&\ddots&\vdots\\
    0&0&m_k
\end{pmatrix}\]
and express $\dr p$ as \[\dr p=\ddd u\ddd m\prod_{i=1}^{k}\lvert \det\nolimits'(m_i)\lvert^{\frac{d(n-\delta_i-n_i)}{2}}.\]
Thus (\ref{E:niceform}) equals to
\[\int\displaylimits_{(K_n\cap P )\bs K_n}\int\displaylimits_{(K_n\cap P )\bs K_n}\int\displaylimits_{M(N)}\int\displaylimits_U\phi(k^{-1}p(m,u)k')\ddd u\,\]\[\prod_{i=1}^{k}\lvert \det\nolimits'(m_i)\lvert^{\frac{d(n_i-n-\delta_i)}{2}}\lvert \det\nolimits'(m)\lvert^{\frac{dn-1}{2}}H(mk',k)\ddd m\ddd k\ddd k'.\]
For $m\in M$ and $k,k'\in (K_n\cap {P})\bs K_n$ we set \begin{equation}\label{E:intdual}\phi(m;k,k')\coloneq\int\displaylimits_U\phi(k^{-1}p(m,u)k)\ddd u,\, h(m;k,k')\coloneq H(mk',k)\delta_P^{-\frac{1}{ 2}}(m).\end{equation}
Thus the coefficient of $T^N$ is
\begin{equation}\label{E:indfin}\int\displaylimits_{(K_n\cap P )\bs K_n}\int\displaylimits_{(K_n\cap P )\bs K_n}\int\displaylimits_{M(N)}\phi(m;k,k')h(m;k,k')\prod_{i=1}^k\lvert \det\nolimits'(m_i)\lvert^{{\frac{dn_i-1}{ 2}}}\ddd m\ddd k\ddd k'.\end{equation}
As in \cite[Proposition 2.3]{JL2} one sees then that this is the coefficient of $T^N$ of a finite sum of elements of the form
\[\prod_{i=1}^kZ(\phi_i,Tq^{-{\frac{dn_i-1}{ 2}}},f_i)\] with $f_i$ a matrix coefficient of $\sigma_i$ and $\phi_i\in \mathscr{S}_{\fl}(M_{n_i}(\mathrm{D}))$ and hence 
  \[\mathcal{L}(\sigma_1\times\ldots\times \sigma_k)\subseteq \prod_{i=1}^k\mathcal{L}(\sigma_i).\]
Next we show that \[ \prod_{i=1}^k\mathcal{L}(\sigma_i)\subseteq \mathcal{L}(\sigma_1\times\ldots\times \sigma_k).\]
Given $\phi_i\in \mathscr{S}_{\fl}(M_{n_i}(\mathrm{D}))$, we can find $\phi\in \sch$ such that
\[\int\displaylimits_U\phi(p(m,u))\ddd u=\prod_{i=1}^k\phi_i(m_i),\, m=(m_1,\ldots, m_k)\in M.\]
Moreover, we choose two smooth functions $\xi, \xi'\in C_c^\infty((K_n\cap P)\bs K_n)$ such that for all $p\in P$ and $N\in \ZZ$
\[\int\displaylimits_{(K_n\cap P)\bs K_n}\int\displaylimits_{(K_n\cap P)\bs K_n}\int\displaylimits_P(\phi\chi_{G_n(N)})(k^{-1}pk')\xi(k)\xi'(k')\delta_P^{-1}(p)\dr p\ddd k\ddd k'=\]\[=\int\displaylimits_{P(N)}\phi(p)\delta_P^{-1}(p)\dr p,\]
where $\chi_{G_n(N)}$ denotes the characteristic function of $G_n(N)$. For example, one could take the characteristic functions of sufficiently small open compact subgroups in $G_n$ whose pro-order is invertible in $R$ and rescale them appropriately. Projecting them into $C_c^\infty((K_n\cap P)\bs K_n)$ gives then the desired functions.
Then the coefficient of $T^N$ in
\[\prod_{i=1}^k Z(\phi_i, Tq^{-\frac{dn_i-1}{2}},f_i)\]
is
\[\int\displaylimits_{M(N)}\prod_{i=1}^k f_i(m_i)\phi_i(m_i)\lvert \det\nolimits'(m_i)\lvert^{{\frac{dn_i-1}{ 2}}}\ddd m=\]
\begin{equation}\label{E:indu}=\int\displaylimits_{(K_n\cap P)\bs K_n}\int\displaylimits_{(K_n\cap P)\bs K_n}\int\displaylimits_{P(N)}\lvert \det\nolimits'(p)\lvert^{\frac{dn-1}{2}}\delta_P^{\frac{1}{2}}(m)\phi(k^{-1}pk')\delta_P^{-1}(p)\end{equation}\[\prod_{i=1}^k f_i(m_i)\xi(k)\xi'(k')\ddd k\ddd k'\dr p.\]
Setting \[H_1(p,k,k')\coloneq \prod_{i=1}^k f_i(m_i)\xi(k)\xi'(k')\] and choosing $H$ as in (\ref{E:matrixcoefficient}) with $H(pk,k')=H_1(p,k,k')$ shows that (\ref{E:indu}) is the coefficient of $T^N$ in $Z(\phi,Tq^{-\frac{dn-1}{2}},f)$ as in (\ref{E:niceform}).

Finally, to prove the equality of the $\gamma$-factors, we go back to 
(\ref{E:intdual}) and define analogously for $\sigma_1^\lor\times\ldots\times\sigma_k^\lor$, $h^\lor(m;k,k')$ and $\phi^\lor(m;k,k')$. Observe that if $h^\lor$ is defined via $f^\lor$, the dual matrix coefficient of $f$, then 
\[h^\lor(m;k,k')=h(m^{-1};k,k')\] and if $\phi^\lor(m;k,k')$ is defined using the Fourier-transform of $\phi$, then
\[{\phi^\lor}(m;k,k')=(\widehat{\phi})(m;k,k').\]
The claim follows then by writing $Z(\phi,Tq^{-\frac{dn-1}{2}},f)$
and $Z(\widehat{\phi},T^{-1}q^{-\frac{dn+1}{2}},f^\lor)$ in the form of (\ref{E:indfin}).
\end{proof}
We write for $\fm$ an aperiodic multisegment $\langle\fm\rangle\coloneq \Z(\fm)^*$.
\begin{theorem}\label{T:JL}
    Let $\Delta=[a,b]_\rho$ be a segment over $R$. If $\rho\cong\chi$ for an unramified character $\chi$ of $\mathrm{F}$ and $q^d\neq 1$, then
        \[L(\langle\Delta\rangle,T)=\frac{1}{ 1-\chi(\varpi_{\mathrm{F}})q^{-db+{\frac{1-d}{2}}}T}, \]
        where $\varpi_{\mathrm{F}}$ is a uniformizer of $\mathfrak{o}_\mathrm{F}$.
        Otherwise \[L(\langle\Delta\rangle,T)=1.\]
    More generally, if $\fm=\De_1+\ldots +\De_k$ is an aperiodic multisegment
    then \[\mathcal{L}(\langle\fm\rangle)=\mathcal{L}(\langle\De_1\rangle\times\ldots\times \langle\De_k\rangle)\]
    and \[L(\langle\fm\rangle,T)=L(\langle\De_1\rangle\times\ldots\times \langle\De_k\rangle,T)=\prod_{i=1}^kL(\langle\De_i\rangle,T).\
        \] 
\end{theorem}
If $R=\ql$ this is {\cite[Theorem 2.7]{JL2}} and the case $\fm$ a banal multisegment was covered in \cite[Theorem 3.1, Theorem 5.7]{Mzeta}.
From now on we thus assume that $R=\fl$.
\begin{proof}[Proof of \Cref{T:JL}]
The idea of the proof is to apply \Cref{P:L1} together with \Cref{C:der}. Indeed, note that this shows that 
\begin{equation}\label{E:div}P(\Dlm(\pi),T)\lvert P(\pi,T)\end{equation} for all $\square$-irreducible cuspidal unramified representations $\rho$. Furthermore, note that by \Cref{T:mainin} we have a precise description of $\Dlm(\langle \fm\rangle)$ in terms of the multisegment $\fm$.

We start by proving the claim that \[L(\langle[0,b]_\rho\rangle,T)=\frac{1}{ 1-\chi(\varpi_{\mathrm{F}})q^{-db+{\frac{1-d}{2}}}T}\] if $\rho$ is unramified and $\square$-irreducible and $1$ otherwise. We proceed by induction on $l([0,b]_\rho)$. Note that the base-case is \cite[Theorem 3.1]{Mzeta} and if $\rho$ is either ramified or not $\square$-irreducible, \Cref{L:su} implies \[P(\langle[0,b]_\rho\rangle,T)\lvert P(\rho^{b+1},T)\stackrel{\Cref{L:Lind}}{=}1,\] which proves the claim.
Thus we assume $\rho$ unramified and $\square$-irreducible. Then the claim follows from the induction-hypothesis since \[\Dlm(\langle[0,b]_\rho\rangle)=\langle[1,b]_\rho\rangle\] by \Cref{T:mainin} and \Cref{L:quot} and hence by (\ref{E:div})
\[1-\rho(\varpi_{\mathrm{F}})q^{-db+{\frac{1-d}{2}}}T\vert P(\langle[0,b]_\rho\rangle,T).\]
On the other hand, since $\rho$ is supercuspidal, we can find a lift $[0,b]_\trho$ with $\trho$ unramified and hence $\rl(\langle[0,b]_\trho\rangle)$ contains $[0,b]_\rho$. By \Cref{L:su} and \Cref{L:lin} also \[ P(\langle[0,b]_\rho\rangle,T)\lvert 1-\rho(\varpi_{\mathrm{F}})q^{-db+{\frac{1-d}{2}}}T,\] proving the claim for a segment. 

Next we compute the $L$-function of $\langle\fm\rangle$.
Let us argue by induction on $\deg(\fm)$ and assume that we have proven the claim already for all multisegments of degree lesser than $\deg(\fm)$, the base case being \cite[Theorem 3.1]{Mzeta}. We can write $\langle\fm\rangle=\pi_1\times \pi_2$, where $\pi_1$ has $\square$-irreducible unramified cuspidal support and $\pi_2$ has $\square$-reducible or ramified cuspidal support by \Cref{L:H2}. Since a representation induced from $\square$-reducible cuspidal representations or ramified cuspidal support has trivial $L$-factor by \Cref{L:Lind}, $L(\pi_2,T)=1$ by \Cref{L:su}.
Again by \Cref{L:Lind}
\[L(\langle \fm\rangle,T)=L(\pi_1,T)L(\pi_2,T)=L(\pi_1,T)\] and
hence we can assume that $\langle\fm\rangle$ has $\square$-irreducible cuspidal unramified support.

From the $\ql$-case we obtain that for any lift $\tfm$ of $\fm$ \[\prod_{i=1}^kP(\langle\De_i\rangle,T)=\rl(P(\langle \tfm\rangle),T),\] since we assumed $\langle\fm\rangle$ to have $\square$-irreducible unramified cuspidal support.
    From (\ref{E:div}) and \Cref{T:derseg} we know that \[P(\Dlm(\langle\fm \rangle),T)\lvert P(\langle\fm\rangle,T)\]
    and from \Cref{L:su} and \Cref{L:Lind} it follows that \[P(\langle\fm\rangle,T)\lvert \prod_{i=1}^kP(\langle\De_i\rangle,T).\] 
    Fix now a $\square$-irreducible cuspidal unramified character $\rho$ such that $\Dl(\pi)\neq 0$. Let moreover $(A,f)$ be a best matching function of $\fm^\lor$ with respect to $\rho^\lor$. 
    By \Cref{T:derseg}, \Cref{T:invol} and the induction-hypothesis we then have that
    \[\prod_{i=1}^kP(\langle\De_i\rangle,T) P(\Dlm(\langle \fm\rangle),T)^{-1}=P(\rho,T)^{-k},\]
    where $k$ is the multiplicity of $[0,0]_\rho$ in $\fm-(\fm_f^\lor)^\lor$.
    Similarly, for a fixed $a\in\ZZ$ such that $\mathcal{D}_{\rho^\lor v_\rho ^{-a},r}(\langle \fm^\lor\rangle )\neq 0$, we let 
    $(A',f')$ be a best matching function of $\fm$ with respect to $\rho v_\rho^a$. 
    Set $k'$ to the multiplicity of $[-a,-a]_{\rho^\lor}$ in $\fm^\lor-(\fm_{f'})^\lor$. We thus obtain that \[P(\langle\fm\rangle,T)P(\rho,T)^l=\prod_{i=1}^kP(\langle\De_i\rangle,T)\] for some $l\le k$ and
    \[P(\langle\fm^\lor),T)P(\rho^\lor v_\rho^{-a},T)^{l'}=\prod_{i=1}^kP(\langle\De_i^\lor\rangle,T)\]
    for $l'\le k'$. We will now show that $l=l'=0$.
    By \Cref{L:lin}(2)
    \[\frac{\prod_{i=1}^kP(\langle\De_i\rangle,T)}{\prod_{i=1}^kP(\langle\De_i^\lor\rangle,q^{-1}T^{-1})}\rl(\epsilon(T,\langle\tfm), \widetilde{\psi}))=\]\[={\frac{\rl(L(\langle\tfm^\lor),q^{-1}T^{-1}))}{ \rl((L(\langle\tfm),T))}}\rl(\epsilon(T,\langle\tfm), \widetilde{\psi}))={\frac{L(\langle\fm^\lor),q^{-1}T^{-1})}{ L(\langle\fm),T)}}\epsilon(T,\langle\fm\rangle, {\psi}).\]    
        Rewriting, we obtain that
        \[\frac{P(\rho,T)^l}{P(\rho^\lor v_\rho^{-a},q^{-1}T^{-1})^{l'}}=\frac{\epsilon(T,\langle\fm\rangle, {\psi})}{\rl(\epsilon(T,\langle\tfm), \widetilde{\psi}))}.\]
Recall now that the right side is a unit in $\fl[T,T^{-1}]$ and hence so has to be left side. But
\[\frac{P(\rho,T)^l}{P(\rho^\lor v_\rho^{-a},q^{-1}T^{-1})^{l'}}=\frac{(1-\rho(\varpi_{\mathrm{F}})q^{{\frac{1-d}{2}}}T)^l}{(1-\rho^\lor(\varpi_{\mathrm{F}}^{-1}) q^{ad+{\frac{1-d}{2}}-1}T^{-1})^{l'}}\]and in order for this to be a unit
either $l=l'=0$, in which case we are done, or $l=l'>0$
and
$-a+1=0\mod o(\rho)$. 
We let now $d_0$ be the multiplicity of $[0,0]_\rho$ and $d_1$ be the multiplicity $[1,1]_\rho$ in $\fm$.
By property (5) of \Cref{S:bmf} and \Cref{T:invol} we obtain that if $[0,0]_\rho$ appears with non-zero multiplicity in $\fm-(\fm_f^\lor)^\lor$, then there have to exist more copies of $[0,0]_\rho$ in $\fm$ than copies of $[1,1]_\rho$. Thus if $k\ge l>0$, we have $d_1<d_0$.
 But on the other hand, the same argument shows that if $[-1,-1]_{\rho^\lor}=[-a,-a]_{\rho^\lor}$ appears with non-zero multiplicity in $\fm^\lor-(\fm_{f'})^\lor$, then there have to exist more copies of  $[-1,-1]_{\rho^\lor}$ in $\fm^\lor$ than copies of  $[0,0]_{\rho^\lor}$. Thus if $k'\ge l'>0$, we have $d_0<d_1$. But since we already showed that $l=l'$, the only possibility is $l=l'=0$.
\end{proof}
We obtain as a corollary the following.
\begin{corollary}\label{C:Lf}
    Let $\fm=\De_1+\ldots+\De_k$ be an aperiodic multisegment.
    Then
    \[\epsilon(T,\langle\fm\rangle,\psi)=\prod_{i=1}^k\epsilon(T,\langle\De_i\rangle,\psi)\] and
    \[\gamma(T,\langle\fm\rangle,\psi)=\prod_{i=1}^k\gamma(T,\langle\De_i\rangle,\psi).\]
\end{corollary}
\begin{proof}
Notice that if $\pi$ is a subquotient of $\pi'$,
$\gamma(T,\pi,\psi)=\gamma(T,\pi',\psi)$ and hence the claim for the $\gamma$-factor follows from \Cref{L:Lind}. By the functional equation, the claim also follows for the $\epsilon$-factors.
\end{proof}
\subsubsection{ LLC and local factors}\label{S:LLC}
In this subsection assume $\mathrm{D}=\mathrm{F}$.
We recall the map
\[\Cc\colon \Irr_n\iso \{\Cc-\text{parameters of length }n\} \]
defined in \cite{Ccor}.
Let us introduce the main actors of this story. We denote by $W_\mathrm{F}$ the Weil group of $\mathrm{F}$, $I_\mathrm{F}$ the inertia sugroup and by $\nu$ the unique unramified character of $\wf$ acting on the Frobenius by $\nu(\mathrm{Frob})=q^{-1}$. A Deligne $R$-representation of $\wf$ is a pair $(\Phi,U)$, where $\Phi$ is a finite-dimensional, smooth representation of $\wf$ over $R$ and $U$ an element in $\Ho_{\wf}(\nu\Phi,\Phi)$.
We call $(\Phi,U)$ {semi-simple} if $\Phi$ is semi-simple and if $R=\fl$, let $o(\Phi)$ be the smallest natural number $k$ such that \[\Phi\cong \nu^k\Phi.\] Let \[\ZZ[\Phi]\coloneq \begin{cases}
    \{\Phi,\ldots,\nu^{o(\Phi)-1}\Phi\}&\, R=\fl,\\
    \{\nu^k\Phi:k\in \ZZ\}&\, R=\ql.
\end{cases}\]
We say a semi-simple $R$-Deligne representation is supported on $\ZZ[\Phi]$ if the underlying $\wf$-representation is the direct sum of elements in $\ZZ[\Phi]$.
We say two indecomposable representations $(\Phi,U)$ and $(\Phi',U')$ are equivalent if $(\Phi',U')\cong (\Phi,\lambda U)$ for some non-zero scalar $\lambda\in R$ and two general semi-simple $R$-Deligne representations are equivalent if their respective indecomposable parts are equivalent.
We denote by $\mathrm{Rep}_{ss}(\wf,R)$ the set of equivalence classes of semi-simple $R$-Deligne representations and by $\mathrm{Rep}_{ss,n}(\wf,R)$ the subset of representations of length $n$, where the length is the dimension of the underlying $\wf$-representation.
We call an $R$-Deligne representation $(\Phi,U)$ {nilpotent} if $U$ is nilpotent and recall that over $R=\ql$, the equivalence classes of semi-simple representations are in bijections with equivalence classes of nilpotent semi-simple representations. We denote the latter set by $\mathrm{Nil}_{ss}(\wf,R)$ and by $\mathrm{Nil}_{ss,n}(\wf,R)$ its subset consisting of representations of length $n$. 

For $r\in \ZZ_{\ge 1}$, we let $[0,r-1]$ be the $R$-Deligne representation whose underlying $\wf$-representation is \[\bigoplus_{i=0}^{r-1}\nu^k\]
and $U$ is defined by \[U(x_0,\ldots, x_{r-1})=(0,x_0,\ldots, x_{r-2}).\]
For $\Phi$ a general irreducible representation of $\wf$ we can then define
\[[0,r-1]\otimes \Phi\]
and these parametrize all nilpotent semi-simple $R$-Deligne representations up to equivalence. 
The local Langlands correspondence gives then a well-known canonical bijection
\[\mathrm{LLC}\colon \mathrm{Nil}_{ss,n}(\wf,\ql)\iso \Irr_n(\ql).\]
Over $\fl$ there exists a second class of indecomposable semi-simple representations of the following form.
Let $\Phi$ be an irreducible representation of $\wf$, denote \[\Psi(\Phi)\coloneq \bigoplus_{k=0}^{o(\Phi)-1}\nu^k\Phi\] and pick an isomorphism $I\colon \nu^{o(\Phi)}\Phi\iso \Phi$. We then define $C(I,\Phi)= C(\Phi)$ as the $\fl$-Deligne representation with underlying $\wf$-representation $\Psi(\Phi)$ and \[U(x_0,\ldots x_{o(\Phi)-1})=(I(x_{o(\Phi)-1}),x_0,\ldots, x_{o(\Phi)-2}).\] Then $C(\Phi)$ depends up to equivalence only on $\ZZ[\Phi]$ and the underlying morphism $U$ of $C(\phi)$ is a bijection.
The indecomposable semi-simple $\fl$-Deligne representations are then classified up to equivalence by \[[0,r-1]\otimes \Phi\text{ and }[0,r-1]\otimes C(\Phi),\] see \cite[Main Theorem 1]{Ccor}. 

We call an semi-simple $\ql$-representation $\Phi$ of $\wf$ \emph{integral} if it admits a $\wf$-stable $\zl$-lattice $L$, which generates $\Phi$. One can then 
define its reduction mod $\ell$ $\rl(\Phi)$ by taking the semi-simplification of $L\otimes_\zl\fl$, which is independent of the chosen lattice. We let $\mathrm{Nil}_{ss}^e(\wf,\ql)$ be the semi-simple nilpotent representations whose underlying $\wf$-representation is integral.
One obtains hence a map \[\rl\colon \mathrm{Nil}_{ss}^e(\wf,\ql)\ra \mathrm{Nil}_{ss}(\wf,\fl)\]
sending \[[0,r-1]\otimes \Phi\mapsto  [0,r-1]\otimes\rl(\Phi).\]

We also recall the following surjection introduced in \cite[I.8.6]{LLC}. We denote by $\Irr_n(\ql)_e$ the isomorphism classes of integral irreducible representations. For $\pi$ in this set, choose $\alpha$ the maximal partition such that $\pi$ is $\alpha$-degenerate. Then map $\pi$ to the unique irreducible constituent in $[\rl(\pi)]$ which is $\alpha$-degenerate.
The so-constructed map is surjective and denoted by
\[\jl\colon \Irr_n(\ql)_e\ra \Irr_n(\fl).\]
\begin{theorem}[{\cite[Theorem 1.6, Theorem 1.8.5]{LLC}}]
    There exists a bijection \[\mathrm{V}\colon \Irr_n(\fl) \iso \mathrm{Nil}_{ss,n}(\wf,\fl) ,\]
which is uniquely characterized by the equality
\[\mathrm{V}(\jl(\mathrm{LLC}(\Phi,U)^*)^*)=\rl(\Phi,U).\]
\end{theorem}
 We recall next the map
\[\vc\colon \mathrm{Nil}_{ss}(\wf,\fl)\ra \mathrm{Rep}_{ss}(\wf,\fl),\]
 \emph{cf.} \cite[§6.3]{Ccor}.
A nilpotent $\fl$-Deligne representation supported in $\ZZ[\Phi]$ can always be written in the form
\[\Phi_{ap}\oplus \Phi_{cyc}\]
with \[\Phi_{ap}=\bigoplus_{i\ge 1}\bigoplus_{k=1}^{o(\Phi)-1}c_{i,k}[0,i-1]\otimes \nu^k\Phi,\] $c_{i,k}\in \ZZ_{\ge 0}$, $c_{i,k}=0$ for $i>>0$ and for each $i$ there exists at least one $k$ such that $c_{i,k}=0$.
Furthermore, $\Phi_{cyc}$ is of the form
\[\Phi_{cyc}=\bigoplus_{i\ge 1}a_{i}[0,i-1]\otimes \bigoplus_{k=1}^{o(\Phi)-1}\nu^k\Phi\] with $a_i\in \ZZ_{\ge 0}$, $a_i=0$ for $i>>0$.
Then \[\vc(\Phi_{ap}\oplus \Phi_{cyc})\coloneq \Phi_{ap}\oplus \bigoplus_{i\ge 1}a_{i}[0,i-1]\otimes C(\Phi).\] Finally, for a nilpotent $\fl$-Deligne representation $(\Phi_1,U_1)\oplus\cdots\oplus (\Phi_k,U_k)$ with each of the summands supported in another $\ZZ[\Phi']$,
\[\vc((\Phi_1,U_1)\oplus\cdots\oplus (\Phi_k,U_k))\coloneq \vc(\Phi_1,U_1)\oplus\cdots\oplus \vc(\Phi_k,U_k).\]
One can then finally define the map
\[\Cc\coloneq \vc\circ\mathrm{V}\colon \Irr_n(\fl)\ra \mathrm{Rep}_{ss}(\wf,\fl).\]
We note the following properties of $\Cc$, see \cite[Examples 6.5]{Ccor}. If $\rho$ is supercuspidal, $\Cc(\rho)$ is banal if and only $\rho\in \scu$ and in this case $\Cc(\rho)$ is unramified if and only if $\rho$ is.
On the other hand, if $\rho$ is $\square$-reducible, $\Cc(\rho)=(\Phi,I_\Phi)$ for some bijective map $I_\Phi$.
More generally, if $\pi= \langle \fm\rangle\in \Irr_n$ for an aperiodic multisegment $\fm=\De_1+\ldots+\De_k$ with $\De_i=[a_i,b_i]_{\rho_i}$ we have \[\Cc(\pi)=\bigoplus_{i=1}^k[0,b_i-a_i]\otimes \nu^{a_i}\Cc(\rho_i).\]

For $(\Phi,U)\in \mathrm{Rep}_{ss}(\wf,R)$ one can associate a local $L$-factor 
\[L((\Phi,U),T)\coloneq \det\nolimits(\restr{1-T\Phi(\mathrm{Frob})}{\mathrm{Ker}(U)^{I_\mathrm{F}}})^{-1}\in R(T),\] an $\epsilon$-factor \[\epsilon(T,(\Phi,U),\psi)\in R[T,T^{-1}]^*\] and $\gamma$-factor \[\gamma(T,(\Phi,U),\psi)\in R(T)\] to $(\Phi,U)$, \emph{cf.} \cite[§5]{Ccor}. Moreover, the $\gamma$-factor does only depend on $\Phi$, all three factors are multiplicative with respect to $\oplus$ and for $(\widetilde{\Phi},U)\in  \mathrm{Nil}_{ss}^e(\wf,\ql)$,
\[\rl(\gamma(T,(\widetilde{\Phi},U),\widetilde{\Psi}))=\gamma(T,\rl(\widetilde{\Phi},U),\psi),\]
see \cite[Proposition 5.6, Proposition 5.11]{Ccor}.
\begin{theorem}\label{T:ccor}
    The map $\Cc$ respects the local factors, \emph{i.e.} for $\pi\in \Irr_n$
    \[L(\pi,T)=L(\Cc(\pi),T),\, \epsilon(T,\pi,\psi)=\epsilon(T,\Cc(\pi),\psi),\, \gamma(T,\pi,\psi)=\gamma(T,\Cc(\pi),\psi).\]
\end{theorem}
\begin{proof}
    We start with the equality of the $\gamma$-factors. Since the $\gamma$-factor of $\pi$ does only depend on the supercuspidal support of $\pi$ by \Cref{L:su} and the $\gamma$ factor of $\Cc(\pi)$ depends only on the underlying $\wf$-representation, the multiplicativity of the $\gamma$-factor of $\mathrm{Rep}_{ss}(\wf,\fl)$ and \Cref{L:Lind} show that it is enough to show the equality for a supercuspidal representations $\rho$. But by picking a lift $\trho$ of $\rho$, we see that
    \[\gamma(T,\rho,\psi)=\rl(\gamma(T,\trho,\widetilde{\Psi}))=\rl(\gamma(T,\mathrm{LLC}^{-1}(\trho),\widetilde{\Psi}))=\gamma(T,\Cc(\rho),\psi).\]
    For the $L$-factor, note that by \Cref{T:JL} both sides are multiplicative and hence it is enough to show the claim for a segment $\pi=\langle [a,b]_\rho\rangle$. It is now easy to check from the definition and the discussion above the theorem, that
    \[L([0,b-a]\otimes \Cc(\rho v_\rho^a),T)=L(\Cc(\rho v_\rho^b),T).\]
    Now if $\rho$ is either non-banal, \emph{i.e.} $\square$-reducible, or not an unramified character, it follows from the above discussion that either the underlying morphism in $\Cc(\rho v_\rho^b)$ is bijective or there exists no $I_{\mathrm{F}}$-fixed vector, hence the $L$-factor is trivial. On the other hand, if $\rho$ is an unramified character, \[\Cc(\rho v_\rho^b)(\mathrm{Frob})=\rho (\varpi_{\mathrm{F}})q^{-b}.\] We thus arrive at
    \[L([0,b-a]\otimes \Cc(\rho v_\rho^a),T)=\begin{cases}
        1&\text{ if }o(\rho)=1\text{ or }\rho \text{ is ramified},\\
        \frac{1}{ 1-\rho(\varpi_{\mathrm{F}})q^{-b}T}&\text{ if }\rho\in\scu\text{ and }\rho \text{ is unramified}.
    \end{cases}\]
    By \Cref{T:JL} this concise with $L(\langle [a,b]_\rho\rangle,T)$. Finally, the equality of $\epsilon$-factors follows from the functional equation \cite[Definition 5.5]{Ccor}.
\end{proof}
\section{Quiver varieties}\label{S:Cons}
 In this section we will describe certain quiver varieties and explain how they relate to the representation theory of $G_n$ over $\fl$. We start by setting the following geometric stage.
 \subsection{ Quiver varieties}\label{S:qv}
We fix a cuspidal representation $\rho$.
Let $Q={A^+_{o(\rho)-1}}$ be the affine Dynkin quiver with $o(\rho)$ vertices numbered $\{0,\ldots,o(\rho)-1\}$ and an arrow from $i$ to $j$ if $j=i+1\mod o(\rho)$, \emph{i.e.} of $Q$ is of the form
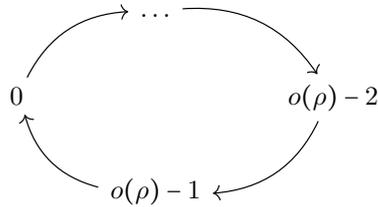
\[\begin{tikzcd}
	& \ldots\arrow[dr, bend left]& \\
	0\arrow[ru, bend left] && o(\rho)-2\arrow[dl, bend left]\\
	& o(\rho)-1\arrow[ul, bend left] &
\end{tikzcd}\]
If $o(\rho)=\infty$, we let $A_\infty=A_\infty^+$
\[\begin{tikzcd}    
\ldots\arrow[r]&i-1\arrow[r]&i\arrow[r]&i+1\arrow[r]&\ldots
\end{tikzcd}\]
A finite dimensional representation of $Q$ is a graded vector space \[V=\bigoplus_{i=0}^{o(\rho)-1}V_i\] over some algebraically closed and complete field $K$ (e.g. $\mathbb{C}$) and a linear map $T_{i\rightarrow j}\in \mathrm{Hom}(V_i,V_{j})$ for each edge $i\rightarrow j$. Fixing $V$ we denote by $E(V)$ the affine space of representations of $Q$ with underlying vector space $V$ and equip it with the natural topology coming from $K$.
The group \[\mathrm{GL}(V)=\mathrm{GL}(V_0)\times\ldots\times \mathrm{GL}(V_{o(\rho)-1})\]
acts on $E(V)$ by
\[(g_0,\ldots,g_{o(\rho)-1})\cdot (x_0,\ldots,x_{o(\rho)-1})\coloneq (g_1x_0g_0^{-1},g_2x_1g_1^{-1},\ldots,g_0x_{o(\rho)-1}g_{o(\rho)-1}^{-1}).\] We denote by $N(V)\subseteq E(V)$ the space of nilpotent representations, \emph{i.e.} the collection of all $T\in E(V)$ such that for large enough $N\in\mathbb{N}$, $T^N=0$, where $T$ is seen as an element of $\mathrm{Hom}(V,V)$.
 Set \[\mathfrak{s}_V\coloneq (\dim_K V_0)\cdot [\rho]+\ldots+(\dim_K V_{o(\rho)-1})\cdot [\rho v^{o(\rho)-1}].\]
 The $\mathrm{GL}(V)$-orbits of $N(V)$ are parametrized by multisegments with cuspidal support $\mathfrak{s}_V$ as follows, see \cite[§15]{Lusztig1991QuiversPS}. 
 
 First assign to a segment $\De=[a,b]_\rho$ with cuspidal support \[\cus_\Ms(\De)={d_0\cdot [\rho]+\ldots+d_{o(\rho)-1} \cdot[\rho v^{o(\rho)-1}]}\] the vector space \[V(\De)\coloneq \bigoplus_{i=0}^{o(\rho)-1}K^{d_i}\] together with a basis $e_a,e_{a+1},\ldots,e_b$, $e_i\in V_i$, where the index of $e_i$ is seen modulo $o(\rho)$. The linear map $\lambda(\De)$ associated to $\De$ sends $e_i$ to $e_{i+1}$ for $a\le i<b$ and $e_b$ to $0$.
 
 To a multisegment $\fm=\De_1+\ldots+\De_n$ we then associate the linear maps $\lambda(\fm)\coloneq \lambda(\De_1)\oplus\ldots\oplus\lambda(\De_n)$ of $ V(\fm)\coloneq  V(\De_1)\oplus\ldots\oplus V(\De_n)$.
 It is easy to see that if $\mathrm{cusp}(\fm)=\mathfrak{s}_V$ then $V(\fm)\cong V$ as graded vector spaces and that $\lambda(\fm)$ and $\lambda(\fm')$ lie in the same orbit of $G_V$ if and only if $\fm$ and $\fm'$ are equivalent.
Denote the orbit of $\lambda(\fm)$ in $V$ as $X_\fm$.
We define an order on the set of multisegments by setting $\fn \lp \fm$ if \[X_\fm\subseteq \overline{X}_\fn ,\] where $\overline{X}_\fn$ is the closure of $X_\fn$ in $N(V)$ with its topology being the one coming from $E(V)$.

\subsection{ }\label{S:motiv} 
 To motivate the next definitions we recall first the case $R=\ql$, which was treated in \cite{Zel}.
Assume for a moment that $R=\ql$. Let $\rho$ be a cuspidal representation of $G_m$ and $\Delta=[a,b]_\rho$ and $\Delta'=[a',b']_\rho$ two segments in $\mathcal{S}(\rho)$ such that $\De$ precedes $\De'$, \emph{i.e.} \[a+1\le a'\le b+1\le b'.\]Define the union and intersection of $\Delta$ and $\Delta'$ as \[\Delta\cup\Delta'\coloneq [a,b']_\rho,\, \Delta\cap\Delta'\coloneq [a',b]_\rho\] and observe that they are unlinked. In \cite[§4]{Zel} the following decomposition in the Grothendieck group was proven.
\begin{equation}\label{E:caseql}
    [\Z(\Delta)\times \Z(\Delta')]=[\Z(\Delta\cup\Delta)\times \Z(\Delta\cap\Delta')]+[\Z(\Delta+\Delta')].
\end{equation}
Let $\fm\in \Ms(\rho)$ be a multisegment. Then in \cite{ZelKL} and \cite[§7]{Zel} the authors proved that in this case $[\Z(\fn)]\le \I(\fm)$ if and only if $\fn\lp \fm$, see \Cref{S:qv}.
In this section we are going to prove the analogous statement over $\fl$.

\subsection{ Elementary operations} 
We recall the following combinatorial description of $X_\fm\subseteq \overline{X_\fn}$ in terms of the underlying multisegments. Fix a cuspidal $\rho$, $Q$ and $V\cong V(\fm)\cong V(\fn)$ as in \Cref{S:qv}.
We define the notation of an elementary operation on a multisegment as follows.
Pick two $\rho$-segments $\De$ and $\De'$ in $\fm$, which are by definition only defined up to equivalence. Next pick two representatives for $\De$ and $\De'$
\[[a,b]_\rho=(\rho v_\rho^a,\ldots, \rho v_\rho^b),\,[a',b']_\rho=(\rho v_\rho^{a'},\ldots, \rho v_\rho^{b'}) \]
such that \begin{equation}\label{E:ineq}
    a+1\le a'\le b+1\le b'
\end{equation} and define the elementary operation with respect to $(\De,\De')$
\[\fm\mapsto \fm-\De-\De'+[a,b']_\rho+[a',b]_\rho.\] If there do not exist representatives satisfying (\ref{E:ineq}), we cannot perform an elementary operation with respect to $(\De,\De')$.
If we can perform an elementary operation we call the segments linked.
\begin{lemma}[{\cite{orbitclos}, \cite[Theorem 3.12]{orbitclos2}}]\label{L:ele}
    Let $\fm$ and $\fn$ be two multisegments with the same support contained in $\Ms(\rho)$. Then $\fn\lp\fm$ if and only if $\fn$ can be obtained from $\fm$ via finitely many elementary operations.
\end{lemma}
The next lemmas will allow us to have a better grip on elementary operations and their proofs are rather straightforward.
\begin{lemma}\label{L:comblem}
Let $\De_1,\, \De_2,\, \Gamma_1,\,\Gamma_2$ be segments in $\mathcal{S}(\rho)$, where $\De_1$ or $\De_2$ are also allowed be the empty segment and $\De_1+\De_2\lp \Gamma_1+\Gamma_2$. Then either \begin{equation}\label{E:c1}\De_1+\De_2^+\lp\Gamma_1+\Gamma_2^+\text{ and }(\De_1+\De_2^+)^1=(\Gamma_1+\Gamma_2^+)^1\text{ or }\end{equation}\begin{equation}\label{E:c2}\De_1^++\De_2\lp\Gamma_1+\Gamma_2^+\text{ and }(\De_1^++\De_2)^1=(\Gamma_1+\Gamma_2^+)^1.\end{equation}
If in the first case $\De_2$ is the empty segment, we mean by $\De_2^+$ the segment $[b+1,b+1]_\rho$, where $b$ is such that $\Gamma_2=[a,b]_\rho$ and analogously in the second case for $\De_1$.
Moreover, \[\De_1^++\De_2^+\lp\Gamma_1^++\Gamma_2^+\text{ and }(\De_1^++\De_2^+)^1=(\Gamma_1^++\Gamma_2^+)^1.\]
\end{lemma}
\begin{proof}
We will prove the first claim by induction on the number of elementary operations necessary to obtain $\De_1+\De_2$ from $\Gamma_1+\Gamma_2$. We can moreover assume without loss of generality that $l(\De_1)>l(\De_2)$.
Assume first that they differ by one and write the representatives of $\Gamma_1$ and $\Gamma_2$ as $[a,b]_\rho$ and $[a',b']_\rho$ such that the elementary operation is performed by \[\Gamma_1+\Gamma_2\mapsto [a,b']_\rho+[a',b]_\rho\] and the representatives satisfy (\ref{E:ineq}).
Now, since $l(\De_1)>l(\De_2)$,
\[\De_1=[a,b']_\rho,\, \De_2=[a',b]_\rho.\]
We have $a+1\le a'\le b+1\le b'+1$ and therefore
\[\Gamma_1+\Gamma_2^+\mapsto [a,b'+1]_\rho+[a',b]_\rho\] is an elementary operation, implying (\ref{E:c2}).
Next we assume that again they differ by one elementary operation but this time the representative of $\Gamma_1$ is $[a',b']_\rho$ and the representative of $\Gamma_2$ is $[a,b]_\rho$, again satisfying (\ref{E:ineq}). Again by $l(\De_1)>l(\De_2)$, \[\De_1=[a,b']_\rho,\, \De_2=[a',b]_\rho.\] Note that $a+1\le a'\le b+2$. If $b+2\le b'$, we are in the situation of (\ref{E:c1}), since \[\Gamma_1+\Gamma_2^+\mapsto [a,b']_\rho+[a',b+1]_\rho\] is an elementary operation. On the other hand, if $b+1=b'$, then
\[\Gamma_1+\Gamma_2^+=\De_1+\De_2^+\] and we have (\ref{E:c1}).

If they differ by more than one elementary operation, choose $\fn=\De_1'+\De_2'$ such that $\fn\lp \Gamma_1+\Gamma_2$ differ by one elementary operation and $\De_1+\De_2\lp\fn$. By the base case $\De_1'+\De_2'^+\lp \Gamma_1+\Gamma_2^+$ or $\De_1'^++\De_2'\lp \Gamma_1+\Gamma_2^+$ and hence we obtain by the induction hypothesis $\De_1+\De_2^+\lp\Gamma_1+\Gamma_2^+$ or $\De_1^++\De_2\lp\Gamma_1+\Gamma_2^+$. The claim regarding $(-)^1$ follows easily by the same induction argument.

For the second claim, we will also proceed on the number of elementary operations necessary to obtain $\De_1+\De_2$ from $\Gamma_1+\Gamma_2$.
of $\Gamma_1$ and $\Gamma_2$ as $[a,b]_\rho$ and $[a',b']_\rho$ such that the elementary operation is performed by \[\Gamma_1+\Gamma_2\mapsto [a,b']_\rho+[a',b]_\rho\] and the representatives satisfy (\ref{E:ineq}). Then $a+1\le a'\le b+2\le b'+1$ and hence \[\Gamma_1^++\Gamma_2^+\mapsto [a,b'+1]_\rho+[a',b+1]_\rho\] is an elementary operation. The induction step follows now as in the second claim.
\end{proof}
Next, we need the following lemma, which can already be found in \cite{Zel} for complex representations. The argument presented there works analogously, however we will write it out for completeness.
\begin{lemma}\label{L:meh2}
Let $\fn_1,\,\fn_2,\, \fm_1,\,\fm_2$ be multisegments with $\fn_1^1=\fm_1^1$ and $\fn_1^-+\fn_2\lp \fm_1^-+\fm_2$. Then there exists $\fn'=\fn_1'+\fn_2'$ with $\fn'\lp \fm_1+\fm_2$, $\fn_1'^1=\fn_1^1$ and $\fn_1'^-+\fn_2'=\fn_1^-+\fn_2$.
\end{lemma}
\begin{proof}
Assume first that $\fn_1^-+\fn_2$ and $\fm_1^-+\fm_2$ differ by one elementary operation.
Let $\De_1+\De_2$ resp. $\Gamma_1+\Gamma_2$ be the multisegments contained in $\fm_1^-+\fm_2$ resp. $\fn_1^-+\fn_2$ involved in the elementary operation.
Applying \Cref{L:comblem} gives multisegments  $\De_1^*+\De_2^*$ resp. $\Gamma_1^*+\Gamma_2^*$ such that $\Gamma_1^*+\Gamma_2^*$ is contained in $\fm_1+\fm_2$, $\De_1^*+\De_2^*\lp \Gamma_1^*+\Gamma_2^*,$ \[\De_1^*\in\{\De_1,\De_1^+\},\, \De_2^*\in\{\De_2,\De_2^+\},\,\Gamma_1^*\in\{\Gamma_1,\Gamma_1^+\},\, \Gamma_2^*\in\{\Gamma_2,\Gamma_2^+\}\] and $(\epsilon_{\Gamma_1}\Gamma_1^*+\epsilon_{\Gamma_2}\Gamma_2^*)^1=(\epsilon_{\De_1}\De_1^*+\epsilon_{\De_2}\De_2^*)^1$ with 
\[\epsilon_\Lambda\coloneq \begin{cases}
    1&\text{ if }\Lambda^*=\Lambda^+,\\
    0&\text{ otherwise.}
\end{cases}\]
Set \[\fn'\coloneq \De_1^*+\De_2^*+\fm_1+\fm_2-\Gamma_1^*-\Gamma_2^*.\] It is clear by construction that $\fn'\lp\fm_1+\fm_2$ via one elementary operation. 
Write $\De_1^*+\De_2^*-\Gamma_1^*-\Gamma_2^*$ as $\mathfrak{k}_1+\mathfrak{k}_2$, where we allow that the coefficients of the segments in $\mathfrak{k}_i$'s are not necessarily positive integers
and \[\mathfrak{k}_1=\epsilon_{\De_1}\De^*+\epsilon_{\De_2}\De_2^*-\epsilon_{\Gamma_1}\Gamma_1^*-\epsilon_{\Gamma_2}\Gamma_2^*.\] 
 Note that $\mathfrak{k}_1^-+\mathfrak{k}_2=\Gamma_1+\Gamma_2-\De_1-\De_2$ and $\mathfrak{k}_1^1=0$, where we extended the operations $(-)^-$ and $(-)^1$ in the straightforward way to multisegments with integer coefficients.

We now define $\fn_1'\coloneq \fm_1+\mathfrak{k}_1$ and $\fn_2'\coloneq \fm_2+\mathfrak{k}_2$. By construction, the coefficients of all segments in $\fn_1'$ and $\fn_2'$ are positive integers because $\epsilon_{\Gamma_1}=1$ if and only if $\Gamma_1$ was a segment of $\fm_1$ and similarly for $\Gamma_2$. Thus $\fn_1'$ and $\fn_2'$ are well-defined multisegments.
We therefore obtain
\[\fn_1'^-+\fn_2'=\fm_1^-+\fm_2+\De_1+\De_2-\Gamma_1-\Gamma_2=\fn_1^-+\fn_2\]
since $\mathfrak{k}_1^-+\mathfrak{k}_2=\De_1+\De_2-\Gamma_1-\Gamma_2$ 
and $\fn_1'^1=\fm_1^1=\fn_1^1,$ since $\mathfrak{k}_1^1=0$. Therefore we constructed $\fn'=\fn_1'+\fn_2'$ as desired.

If $\fn_1^-+\fn_2$ differs by more than one elementary operation from $\fm_1^-+\fm_2$ the claim follows by induction on the number of elementary operations. Indeed, pick $\fk$ such that $\fm_1^-+\fm_2$ is obtained from $\fk$ via one elementary operation and $\fk\lp \fm_1^-+\fm_2$. By the induction hypothesis we know that there exists $\fk'=\fk_1'+\fk'_2$, $\fk_1'^-+\fk_2'=\fk',\, \fk_1'^-=\fm_1^1$ with $\fk'\lp \fm_1+\fm_2$. 
Moreover, by applying the induction hypothesis a second time,
we obtain $\fn'=\fn_1'+\fn_2'$, $\fn_1'^-+\fn_2'=\fn_1^-+\fn_2$, $\fn'\lp \fk'\lp \fm_1+\fm_2$ and $\fn_1'^1=\fk_1'^1=\fm_1^1=\fn_1^1.$
\end{proof} 
\subsection{ }
We are now ready to prove the result announced at the beginning of the section.
\begin{theorem}\label{T:cons}
     Let $\fm$ and $\fn$ be two multisegments with the same cuspidal support contained in $\Ms(\rho)$. Then $\Z(\fn)$ appears in $\Z(\fm)$ if and only if $\fn\lp \fm$.
\end{theorem}
\begin{proof}

    \textbf{First direction: $\fn\lp \fm\Rightarrow  [\Z(\fn)]\le \I(\fm)$}

    By \Cref{L:ele} and \Cref{T:N5} it is enough to show that if $\fn$ is obtained from $\fm$ via one elementary operation, then $\I(\fn)\le \I(\fm)$. By the exactness of parabolic induction, it is thus enough that if $\De=[a,b]_\rho$ and $\De'=[a',b']_\rho$ satisfying (\ref{E:ineq}), then
    \[\I([a,b']_\rho+[a',b]_\rho)\le \I(\De+\De').\]
    If $\rho$ is supercuspidal,
we can use \Cref{T:rel} to find $\Tilde{\rho}$ such that $\rl(\Tilde{\rho})=\rho$.
    By (\ref{E:caseql}) we have that 
    \[\I([a,b']_\trho+[a',b]_\trho)\le \I([a,b]_\trho+[a',b']_\trho).\] Since parabolic restriction commutes with $\rl$, we have thus that 
    \[\I([a,b']_\rho+[a',b]_\rho)=\rl(\I([a,b']_\trho+[a',b]_\trho))\le \rl(\I([a,b]_\trho+[a',b']_\trho))=\I(\De+\De').\]
    If $\rho$ is not supercuspidal, $o(\rho)=1$ and we argue as follows. Write $\fn=[a,b']_\rho+[a',b]_\rho$ and $\fm=\De+\De'$.
    Observe first that if $\Z(\fn)$ does not have its cuspidal support contained in $\mathbb{N}(\ZZ[\rho])$ then it follows by \cite[Proposition 9.32]{M-S} that $\fn$ consists of two segments of equal length and $\ell=2$.
 For such an $\fn$, $[\Z(\fn)]\le \I(\fn)\le \I(\fm)$ implies therefore $\fn=\De+\De'$ by \Cref{T:N5}.
It thus follows that it is enough to show that the multiplicity of any irreducible representation with cuspidal support in $\mathbb{N}(\ZZ[\rho])$ in $\I(\fn)$ is lesser or equal than its multiplicity in $\I(\fm)$. Let $n$ be the length of $\fm$.
We can therefore utilize the map $\xi_{\rho,n}$ between irreducible representations with inertially equivalent to representations having cuspidal support $n\cdot [\rho]$ and irreducible modules of $\He(n,1)$. Note that this map can be upgraded in the following way, see \cite[Proposition 4.19]{MST}. If $\alpha=(\alpha_1,\ldots,\alpha_k)$ is a composition of $n$ and $\pi=\pi_1\otimes\ldots\otimes\pi_k$ is an irreducible representation of $M_\alpha$ with cuspidal support in $n\cdot [\rho]$, then there exists a bijection between between the irreducible constituents of \[\Ho_{\mathscr{H}_\fl(\alpha_1,1)\otimes\ldots \otimes \mathscr{H}_\fl(\alpha_k,1)}(\mathscr{H}(n,1),\xi_{\rho,\alpha_1}(\pi_1)\otimes\ldots\otimes \xi_{\rho,\alpha_k}(\pi_k))\]
and the irreducible constituents of $\pi_1\times\ldots\times \pi_k$ with cuspidal support $n\cdot[\rho]$. Moreover, the map respects multiplicities. Now for any irreducible representation with cuspidal support $n\cdot[\rho]$, the $\fl[X_1^{\pm 1},\ldots ,X_n^{\pm 1}]$-part of the Hecke-algebra acts trivially, see \cite[Theorem 3.7]{Mathas}.
 Thus we can regard $\xi_{\rho,\alpha_i}(\pi_i)$ and $\xi_{\rho,n}(\pi)$ as $\fl[\mathrm{S}_n]$-modules and as well as the induced Hom-spaces.
We are thus in the realm of the modular representation theory of $\mathrm{S}_n$, which was already discussed in \Cref{S:repsym}. The desired claim is then the last remark of this section.

\textbf{Second direction: $\fn\lp \fm\Leftarrow  [\Z(\fn)]\le \I(\fm)$}

We will proceed inductively and assume the statement to be true for all $G_k$ with $k<\deg(\fm)$.
By \Cref{L:S33}
\[[r_{(\deg(\fm^-),\deg(\fm^1))}(\I(\fm))]\ge[r_{(\deg(\fm^-),\deg(\fm^1))}(\Z(\fn))]\ge [\Z(\fn^-)\otimes \Z(\fn^1)].\]
 We apply the Geometric Lemma applied to $r_{(\deg(\fm^-),\deg(\fm^1))}(\I(\fm))$ and observe that \Cref{L:N1} implies that there exist $\fm_1,\fm_2$ with $\fm=\fm_1+\fm_2$
 such that \[[\Z(\fn^-)\otimes \Z(\fn^1)]\le \I(\fm_1^-+\fm_2)\otimes \I(\fm_1^1).\]
 Thus $\fm_1^1=\fn^1$ by \Cref{T:N4} and by the induction hypothesis $\fn^-\lp \fm_1^-+\fm_2$.
We apply \Cref{L:meh2} to the case $\fn_1=\fn, \fn_2=0$ and obtain therefore the existence of $\fn'=\fn'_1+\fn'_2$ such that $\fn'\lp\fm$, ${\fn'_1}^1=\fn^1$ and $\fn^-={\fn'_1}^-+\fn'_2$. For any multisegment $\mathfrak{k}$ we will write $\mathfrak{k}[1]$ for the multisegment obtain from $\mathfrak{k}$ by replacing every segment $\De$ in it by $^-(\De^+)$ and $\mathfrak{k}^+$ 
for the multisegment obtain from $\mathfrak{k}$ by replacing every segment $\De$ in it by $\De^+$. Note that $(\mathfrak{k}^+)^1=\mathfrak{k}^1[1]$.

Because $\fn^-={\fn'_1}^-+\fn'_2$, it follows that $\fn$ is of the form
\[\fn=(\fn_1'^-)^++(\fn_2')^++\mathfrak{k}',\]
where $\mathfrak{k}'\in\Ms(\rho)$ is some multisegment consisting of segments of length $1$.
Thus $\fn^1$ is of the form \[\fn^1=\fn_1'^2[1]+\fn_2'^1[1]+\mathfrak{k}'.\]
But on the other hand $\fn^1=\fn_1'^1$ and hence
\[\fn_1'^1=\fn_1'^2[1]+\fn_2'^1[1]+\mathfrak{k}'.\] Therefore \[\fn_1'=(\fn_1'^-)^++\fn_2'^1[1]+\mathfrak{k}'.\]
Since $(\fn_2')^+\lp \fn_2'+\fn_2'^1[1]$, it follows that
\[\fn=(\fn_1'^-)^++(\fn_2')^++\mathfrak{k}'\lp (\fn_1'^-)^++\fn_2'+\fn_2'^1[1]+\mathfrak{k}'=\fn'.\]
Thus $\fn\lp\fn'\lp\fm$.
\end{proof}
\begin{corollary}\label{C:openorbit}
A multisegment $\fm$ in $\Ms(\rho)$ is unlinked if and only if $X_\fm$ is an open orbit of $N(V(\fm))$.
\end{corollary}
\begin{proof}
 Since there are only finitely many multisegments $\fn\in \Ms(\rho)$ with $\cus_{\Ms}(\fn)=\cus_{\Ms}(\fm)$, we have 
\[\ol{N(V) \backslash X_\fm}=\ol{\bigcup_{\fn\neq \fm} X_\fn}=\bigcup_{\fn\neq \fm}\ol{X_\fn},\]\[\bigcup_{\fn\neq \fm}X_\fn=N(V)\backslash X_\fm.\]
Now $\bigcup_{\fn\neq \fm}\ol{X_\fn}=\bigcup_{\fn\neq \fm}X_\fn$ if and only if $X_\fm$ is not contained in $\overline{X_\fn}$ for any $\fn\neq \fm$.
By \Cref{T:cons} and \Cref{L:ele} this is equivalent to $\fm$ being unlinked.
\end{proof} Recall the following theorem of \cite{10.1093/imrn/rns278}.
\begin{theorem}[{\cite[Theorem 6.1]{10.1093/imrn/rns278}}]
Let $\rho$ be an unramified character of $G_1$ and $\fm$ an aperiodic multisegment in $\Ms(\rho)$. Then $\Z(\fm)$ is unramified if and only if $\fm$ is unlinked. Moreover, if $\pi\in\Irr_{}(\Omega_{\rho',n})$ is an unramified representation for some cuspidal $\rho'$, then $\rho'$ is an unramified character and there exists and an unlinked multisegment ${\fm\in \Ms(\rho')}$ such that $\Z(\fm)\cong \pi$.
\end{theorem}
Observe that an unlinked multisegment in $\Ms(\rho)$ is automatically aperiodic. Therefore, if $\fm\in\Ms(\rho)$ is unlinked, the cuspidal support of $\Z(\fm)$ is contained in $\bbrh$ by \cite[Proposition 9.34]{M-S}.
\begin{corollary}
Let $\pi\cong \Z(\fm)\in \Irr_{}(\Omega_{\rho,n})$ be a representation. Then $\pi$ is unramified if and only if the orbit $X_\fm$ is open and $\rho$ is an unramified character.
\end{corollary}
We will also use the opportunity to give a formula for the number of unramified representations with given cuspidal support. One has yet to find a use for it though.
\begin{prop}
Let $\rho$ an unramified character, \[\mathfrak{s}=d_0 [\rho]+\ldots+d_{o(\rho)-1} [\rho v^{o(\rho)-1}],\,d_i\in\mathbb{Z}_{\ge 0}\] a fixed cuspidal support, $D\coloneq \min_{i} d_i$ and $m=\#\{i:d_i=D\}$.
The number of unramified irreducible representations $\pi$ with cuspidal support $\mathfrak{s}$, or equivalently, the number of irreducible components of $N(V(\fs))$, is $m$ if $D>0$ and $1$ if $D=0$.
\end{prop}
\begin{proof}
If $D=0$ the space $N(V)=E(V)$ has precisely one irreducible component and thus only one open orbit. Therefore there exists a unique unramified irreducible representation with cuspidal support $\mathfrak{s}$.
We will now give a precise description of the corresponding multisegment $\fm_\fs$.
Let $d_k=\max_i d_i$ and $j\le k$ be such that \[d_{j-1}<d_j=d_{j+1}=\ldots=d_k>d_{k+1}.\]
Set \[\fs'\coloneq d_0 [\rho]+\ldots+(d_j-1)[\rho v^j]+\ldots+(d_k-1)[\rho v^k]+\ldots+d_{o(\rho)-1} [\rho v^{o(\rho)-1}].\]
Then we can compute $\fm_\fs$ iteratively by noting that $\fm_\fs=\fm_{\fs'}+[j,k]$.

If $D>0$ let $\fm_\fs$ be the set of irreducible unramified representations with cuspidal support $\fs$. Let \[\fs^{+1}=(d_0+1)\cdot [\rho]+\ldots+(d_{o(\rho)-1}+1) \cdot[\rho v^{o(\rho)-1}]\]
Then we construct an injective map $i\colon \fm_\fs\rightarrow \fm_{\fs^{+1}}$ by sending an unlinked segment $\fm=\De_1+\ldots+\De_k$ with $l(\De_1)\ge l(\De_i)$ for $\ain{i}{1}{k}$ and $\De_1=[a,b]$ to $[a,b+o(\rho)]+\De_2+\ldots+\De_k$, which is unlinked since $\fm$ is.

There exists also a map $p\colon \fm_{\fs^{+1}}\rightarrow \fm_{\fs}$ constructed as follows. Note that for $\De_1+\ldots+\De_k=\fm\in \fm_{\fs^{+1}}$ there exists a unique $\De_i$ with $l(\De_i)\ge o(\rho)$. Indeed, $\rho v^{o(\rho)-1}$ and $\rho$ must be contained in the same segment, which is then of length greater than $o(\rho)$, since otherwise the two segments would be linked. Since any two segments of length greater than $o(\rho)-1$ are linked, uniqueness follows.
Assume without loss of generality that $l(\De_1)\ge l(\De_i)$ for $\ain{i}{1}{k}$ and write $\De_1=[a,b]$. The map then sends $\fm$ to $[a,b-o(\rho)]+\De_2+\ldots+\De_k$, which is again unlinked. Because $D>0$ it follows that $p$ is injective.

Since both $i$ and $p$ are injective if $D>0$ and $p\circ i=\mathrm{id}$, $\#\fm_\fs=\#\fm_{\fs^{+1}}$ if $D>0$. Thus it suffices to prove the claim for $D=1$. Write \[\fs=\fs'+ [\rho]+\ldots+[\rho v^{o(\rho)-1}]\]
We know from the case $D=0$ that we can write \begin{equation}\label{E:ofthatform}
    \fm_{\fs'}=\fm_1+\ldots+\fm_k
\end{equation} where $\fm_i=\De_i^1+\ldots+\De_i^{l_i}$, $\De_i^j=[a_i^j,b_i^j]$ with $a_i^j\le a_i^{j'}\le b_i^{j'}\le b_i^j$ for $j\le j'$, $\De_i^j=\De_i^{j'}$ if and only if $j=j'$ and $b_i^1+1<a_{i'}^1$ if $i<i'$.

We now construct $m$ many unlinked segments with support $\fs$ as follows. For $\ain{c}{0}{o(\rho)-1}$ such that either $c\le a_i^1-2$ or $c\ge b_i^1+2$ for all $i$ it is easy to check that $\fm_{\fs}+[c,c+o(\rho)-1]$ is unlinked. If $b_i+1=c$ or $c=a_i^1-1$ then the multisegment \[\fm_1+..+\fm_{i-1}+[a_i^1,b_i^1+o(\rho)]+[a_i^2,b_i^2]+\ldots+[a_i^{l_i},b_i^{l_i}]+\fm_{i+1}+\ldots+\fm_k\]
is unlinked. Thus $\#\fm_\fs\ge m$.
On the other hand if $\fm\in \fm_\fs$ then $p(\fm)=\fm_{\fs'}$ which shows that $\fm$ has to be of the form (\ref{E:ofthatform}) and hence showing $\#\fm_\fs\le m$.
\end{proof}
\subsection{ $\square$-irreducible representations revisited}
We will use the opportunity to extend \cite[Conjecture 4.2]{LMsquare} to representations with $\square$-irreducible cuspidal support and give a conjectural classification of square-irreducible representations. Using \Cref{L:H2} we can quickly reduce the classification to a classification for representations $\Z(\fm)$ for $\fm \in \Ms(\rho)_{ap}$ for some cuspidal representation $\rho$. By \Cref{L:nosi} $o(\rho)>1$ if $\Z(\fm)$ has to have any chance of being square-irreducible.
We fix for the rest of the section an aperiodic multisegment $\fm\in \Ms(\rho)_{ap}$ for some $\rho\in \scu$. We set $Q^+\coloneq A_{o(\rho)-1}^+$. Moreover, we recall the quiver representation $(V(\fm),\lambda(\fm))$ over $\C$  of $Q^+$ as in \Cref{S:qv}.
Next, we denote by $Q^-$ the quiver obtained from $Q$ by inverting the arrows. For a fixed graded vector space $V=\oplus_{i=1}^{o(\rho)-1}V_i$ over $\C$ as in \Cref{S:qv} we write $E^+(V)$ respectively $E^-(V)$ for the representations of $Q^+$ respectively $Q^-$ with underlying vector space $V$ and by $N^+(V)$ and $N^-(V)$ the subvariety of nilpotent representations. The third quiver of this story is denoted by
$\overline{Q}$ and of the form
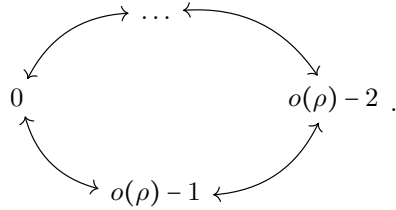
\[\begin{tikzcd}
	& \ldots\arrow[dr, bend left,leftrightarrow]& \\
	0\arrow[ru, bend left,leftrightarrow] && o(\rho)-2\arrow[dl, bend left,leftrightarrow]\\
	& o(\rho)-1\arrow[ul, bend left,leftrightarrow] &
\end{tikzcd}.\]
We let $\overline{E}(V)$ be the affine variety of representations of $Q$ with underlying vector space $V$. Again $\overline{E}(V)$ admits an action of \[\mathrm{GL}(V)=\mathrm{GL}(V_0)\times\ldots \times \gl(V_{o(\rho)-1}).\]
Moreover, there exist two $G_V$-equivariant maps
\[\begin{tikzcd}
E^+(V)\arrow[r,twoheadleftarrow, "p_+"]&\overline{E}(V)\arrow[r,twoheadrightarrow, "p_-"]& E^-(V).
\end{tikzcd}\]
The preimage of $N^+(V)\times N^-(V)$ is denoted by $\overline{N}(V)$. Finally, we recall the moment map
\[[\cdot,\cdot]\colon \overline{N}(V)\ra \mathfrak{gl}(V),\]\[(X_{i\mapsto i+1}, Y_{i+1\mapsto i})_{\ain{i}{0}{o(\rho)-1}}\mapsto (X_{i\mapsto i+1}Y_{i+1\mapsto i}-Y_{i\mapsto i-1}X_{i-1\mapsto i})_{\ain{i}{0}{o(\rho)-1}}.\]
Let $\Lambda(V)$ be the kernel of $[\cdot,\cdot]$.

Recall now above fixed multisegment $\fm$ and $V\coloneq V(\fm)$, $\lambda(\fm)$ and $X_{\fm}$ from \Cref{S:qv}. To avoid confusion we denote the map $\lambda(\fm)^+$ in $N^+(V(\fm))$ and $\lambda(\fm)^-$ in $N^-(V(\fm))$ and similarly for $X_{\fm}^+$ and $X_{\fm}^-$.
We now let $\Ms(V)_{ap}$ be the aperiodic segments $\fn\in \Ms(\rho)_{ap}$ with the same cuspidal support as $\fm$.
For $\fn\in \Ms(V)_{ap}$ denote the closure of $p_{+}^{-1}(X_\fn^+)\cap \Lambda(V)$ in $\Lambda(V)$ by $\mathcal{C}_\fn^+$ and the closure of $p_{-}^{-1}(X_\fn^-)\cap \Lambda(V)$ in $\Lambda(V)$ by $\mathcal{C}_\fn^-$.
\begin{lemma}[{\cite[Proposition 15.5]{Lusztig1991QuiversPS}}]
    The maps $\fn\mapsto \mathcal{C}_\fn^+$ respectively $\fn\mapsto \mathcal{C}_\fn^-$ induce a bijection between $\Ms(V)_{ap}$ and the irreducible constituents of $\Lambda(V)$.
\end{lemma}
Moreover, the interaction of $\mathcal{C}^+$ and $\mathcal{C}^-$ is governed by the Aubert-Zelevinsky duality.
\begin{lemma}[{\cite[Propostion 5.2]{LTV}}]
For $\fn\in \Ms(V)_{ap}$
    \[\mathcal{C}_\fn^+=\mathcal{C}_{\fn^*}^-.\]
\end{lemma}
We thus can generalize \cite[Conjecture 4.2]{LMsquare} as follows.
\begin{conjecture}
    Let $\fm\in \Ms(\rho)_{ap}$ for some $\rho\in \scu$.
    Then $\Z(\fm)$ is square-irreducible if and only if $\mathcal{C}_\fm^+$ contains an open $\gl(V)$-orbit.
\end{conjecture}
At this point, some remarks are in order.
\begin{rem}
(1). If $\fm$ is banal, \emph{i.e.} the cuspidal support of $\fm$ does not contain $[\rho]+\ldots+[\rho v_\rho^{o(\rho)-1}]$, then above conjecture reduces to \cite[Conjecture 4.2]{LMsquare}.

(2). The conjecture is true for segments. First of all note that $\Z(\De)$, $\De=[a,b]_\rho$, is $\square$-irreducible if and only if $b-a+1<o(\rho)$. Thus if $b-a+1<o(\rho)$, we know from the banal case that $\mathcal{C}_\De^+$ admits an open orbit. On the other hand, if $b-a+1\ge o(\rho)$, we can reduce the claim by \cite[Corollay 8.7]{aizenbud2022binary} to the case $b-a+1= o(\rho)$. Indeed, if we assume $b-a+1> o(\rho)$, let in the Corollary $C_1=\mathcal{C}_{\De^-}^+$ and $C_2=\mathcal{C}_{[b,b]_\rho}^+$. In the language of the above Corollary, it is then not hard to see that 
$C_1*C_2=\mathcal{C}_\De^+$. Thus we can assume without loss of generality that $a=0,\, b=o(\rho)-1$. Then the claim that $\mathcal{C}_\De^+$ contains not an open orbit is equivalent to the claim that \[T_{\lambda(\De)}^*X_\De=\{y\in N^-(V(\De)):[\lambda(\De),y]=0\}\] does not contain an open orbit of the centralizer of $\lambda(\De)$. But in this case the centralizer is equal to $\{(a,\ldots,a)\in (C^*)^{o(\rho)}\}$, which acts trivially on $T_{\lambda(\De)}^*X_\De$. However, $T_{\lambda(\De)}^*X_\De$ is one-dimensional and hence  $\mathcal{C}_\De^+$ does not contain an open orbit.
\end{rem}
\begin{ex}
    Let us give the following family of examples of non-banal $\square$-irreducible representations. Let $\rho\in\scu(G_m)$ and consider $\fm=[0,o(\rho)-1]_\rho+[0,0]_\rho$. We claim that $\Z(\fm)$ is $\square$-irreducible and $\mathfrak{C}_\fm^+$ admits an open orbit.

    Let us first check that $\mathfrak{C}_{\fm}^+$ admits an open orbit. Note that \[\fm^*=[0,0]_\rho+\ldots+[o(\rho)-2,o(\rho)-2]_\rho+[o(\rho)-1,o(\rho)]_\rho\]
    We set \[x\coloneq ((1,0),1,\cdots,1,0)\in X_\fm^+,\, x^*\coloneq ((0,1)^T,0,\ldots,0)\in X_{\fm^*}^-.\]
    Here the $i$-th coordinate represents the map from $V_{i-1}(\fm)\ra V_i(\fm)$ respectively $V_i(\fm)\ra V_{i-1}(\fm)$.
    It suffices now to show that the stabilizer of $x$ admits an open orbit in \[T_x^*X_\fm=\{y\in N^-(V(\fm)):[x,y]=0\}.\]
    To see this, we observe that the stabilizer of $x$ equals to
    \[G_x\coloneq \{(\begin{pmatrix}
        a&0\\b&c
    \end{pmatrix},a\ldots,a)\in \mathrm{GL}_2(\C)\times (\C^*)^{o(\rho)}\}.\]
    Moreover, the set $\{(x,y)^T,0,\ldots,0\}:y\neq 0\}$ is an open subset of $T_x^*X_\fm$, which is equal to $G_x\cdot x^*$, finishing the argument.

    We continue by showing that $\Z(\fm)$ is $\square$-irreducible. To do so, we prove the following three lemmata.
\begin{lemma}\label{L:sicus}
    Let $\rho\in\scu(G_m)$ and $\fm,\fn\in \Ms(\rho)$ such that $\fn+\fm$ is aperiodic. Then $\Z(\fm)\times \Z(\fn)$ is irreducible if and only if \[[r_{(n-m,m)}(\Z(\fm+\fn))]=[r_{(n-m,m)}(\Z(\fm)\times \Z(\fn))]\]
\end{lemma}
    \begin{proof}
        One direction is trivially true by \Cref{L:sam}. For the other one, assume that the above equality holds and $\Z(\fn)\times\Z(\fn)$ is not irreducible.
        Then there exits either an irreducible quotient or subrepresentation $\pi$ of $\Z(\fn)\times \Z(\fm)$ different from $\Z(\fn+\fm)$. Since $\pi$ has the same cuspidal support as $\Z(\fn+\fm)$, the claim follows from the exactness of parabolic restriction.
    \end{proof}
    \begin{lemma}\label{L:helpful}
        Let $\ain{a}{0}{o(\rho)-2}$. Then
        \[\Z([0,a]_\rho)\times \Z([0,o(\rho)-1]_\rho+[0,0]_\rho)\] is irreducible.
    \end{lemma}
    \begin{proof}
        We argue by induction on $a$.
        \Cref{T:derseg} and \Cref{L:dervanish} allow us to compute
         \[[r_{(o(\rho)m+(a+1)m,m)}(\Z([0,a]_\rho+[0,o(\rho)-1]_\rho+[0,0]_\rho))]=\]\[= [\Z([0,a-1]_\rho+[0,o(\rho)-1]_\rho+[0,0]_\rho)\otimes \rho v_\rho^a]+\]\[+[\Z([0,a]_\rho+[0,o(\rho)-2]_\rho+[0,0]_\rho)\otimes \rho v_\rho^{-1}].\] On the other hand, the Geometric Lemma gives that
        \[[r_{(o(\rho)m+(a+1)m,m)}(\Z([0,a]_\rho)\times \Z([0,o(\rho)-1]_\rho+[0,0]_\rho))]=\]\[= [\Z([0,a-1]_\rho)\times\Z([0,o(\rho)-1]_\rho+[0,0]_\rho)\otimes \rho v_\rho^a]+\]\[+[\Z([0,a]_\rho) \times \Z([0,o(\rho)-2]_\rho+[0,0]_\rho)\otimes \rho v_\rho^{-1}].\] If $a=0$, \Cref{T:N2} implies \[\rho\times \Z([0,o(\rho)-2]_\rho+[0,0]_\rho)\cong \rho^2\times \Z([0,o(\rho)-2]_\rho)\cong \Z([0,0]_\rho+[0,o(\rho)-2]_\rho+[0,0]_\rho)\]
and hence the parabolic restrictions of the two representations agree. 
If $a>0$, the induction hypothesis and \Cref{T:N2} imply that the parabolic restrictions of the two representations agree, hence the claim follows by \Cref{L:sicus}.
    \end{proof}
    \begin{lemma}\label{L:helpful2.5}
        The representation
        \[\rho\times \rho\times \Z([0,o(\rho)-1]_\rho+[0,0]_\rho)\] is irreducible.
    \end{lemma}
    \begin{proof}
        We proceed analogously to the proof of the above lemma. Namely, the Geometric Lemma and \Cref{C:der} show that
        \[[r_{(o(\rho)m+2m,m)}(\rho\times \rho\times \Z([0,o(\rho)-1]_\rho+[0,0]_\rho))]=\]
        \[=2[\rho\times \Z([0,o(\rho)-1]_\rho+[0,0]_\rho)\otimes\rho]+[\rho\times \rho\times \Z([0,o(\rho)-2]_\rho+[0,0]_\rho)]\otimes\rho v_\rho^{-1}].\]
        On the other hand \Cref{L:dervanish} shows that 
        \[[r_{(o(\rho)m+2m,m)} (\Z([0,o(\rho)-1]_\rho+3[0,0]_\rho))]\ge \]
        \[[\Z([0,o(\rho)-1]_\rho+2[0,0]_\rho))\otimes \rho]+[\Z([0,o(\rho)-2]_\rho+3[0,0]_\rho))\otimes \rho v_\rho^{-1}].\]
    By \Cref{L:helpful} \[\Z([0,o(\rho)-1]_\rho+2[0,0]_\rho)\cong \rho\times \Z([0,o(\rho)-1]_\rho+[0,0]_\rho)\] and by \Cref{T:N2} \[\Z([0,o(\rho)-2]_\rho+3[0,0]_\rho))\cong \Z([0,o(\rho)-2]_\rho)\times \rho^3\cong \rho\times \rho\times \Z([0,o(\rho)-2]_\rho+[0,0]_\rho).\]
        Thus if $\pi$ is another subquotient than $\Z([0,o(\rho)-1]_\rho+3[0,0]_\rho)$ of \[\rho\times \rho\times \Z([0,o(\rho)-1]_\rho+[0,0]_\rho),\] it is either cuspidal, which cannot happen, or Frobenius reciprocity implies that $\pi$ is a subrepresentation of 
\[\Z([0,o(\rho)-1]_\rho+2[0,0]_\rho))\times \rho, \]
        implying $\pi\cong \Z([0,o(\rho)-1]_\rho+3[0,0]_\rho)$ by the theory of derivatives. But \[\Z([0,o(\rho)-1]_\rho+3[0,0]_\rho)\] appears with multiplicity $1$ in \[\rho\times \rho\times \Z([0,o(\rho)-1]_\rho+[0,0]_\rho)\] by \Cref{L:sam}, thus proving the claim.
    \end{proof}
\begin{lemma}\label{L:helpful2}
            Let $\ain{b}{1}{o(\rho)-2}$. Then
        \[\rho\times \Z([0,b]_\rho)\times \Z([0,o(\rho)-1]_\rho+[0,0]_\rho)\] is irreducible.
\end{lemma}
\begin{proof}
    The proof proceeds analogously to the one of \Cref{L:helpful} by showing inductively on $b$ that \[[r_{(o(\rho)m+(b+2)m,m)}(\rho\times \Z([0,b]_\rho)\times \Z([0,o(\rho)-1]_\rho+[0,0]_\rho))]=\]
    \[=[r_{(o(\rho)m+(b+2)m,m)}(\Z([0,0]_\rho+[0,b]_\rho+[0,o(\rho)-1]_\rho+[0,0]_\rho))].\] The base case is \Cref{L:helpful2.5}.
For $b>0$
\[[r_{(o(\rho)m+(b+2)m,m)}(\rho\times \Z([0,b]_\rho)\times \Z([0,o(\rho)-1]_\rho+[0,0]_\rho))]=\]
\[=[\Z([0,b]_\rho)\times \Z([0,o(\rho)-1]_\rho+[0,0]_\rho))\otimes \rho]+\]\[+[\rho\times \Z([0,b-1]_\rho)\times \Z([0,o(\rho)-1]_\rho+[0,0]_\rho)\otimes \rho v^b]+\]
\[+[\rho\times \Z([0,b]_\rho)\times \Z([0,o(\rho)-2]_\rho+[0,0]_\rho)\otimes \rho v_\rho^{-1}]\]
and 
\[[r_{(o(\rho)m+(b+2)m,m)}(\Z([0,0]_\rho+[0,b]_\rho+[0,o(\rho)-1]_\rho+[0,0]_\rho))]=\]
\[=[\Z([0,b]_\rho+[0,o(\rho)-1]_\rho+[0,0]_\rho))\otimes \rho]+\]\[+[\Z([0,b-1]_\rho+[0,o(\rho)-1]_\rho+[0,0]_\rho+[0,0]_\rho)\otimes \rho v^b]+\]
\[+[\Z([0,b]_\rho+[0,0]_\rho+[0,o(\rho)-2]_\rho+[0,0]_\rho)\otimes \rho v_\rho^{-1}].\]
We use the induction hypothesis, \Cref{T:N2} and \Cref{L:helpful} to show that the parabolic restrictions agree. We are now done by \Cref{L:sicus}.
\end{proof}
We now come finally to the $\square$-irreducibility of $\Z(\fm)$, $\fm=[0,o(\rho)-1]+[0,0]_\rho$.
Note that by \Cref{T:derseg} and \Cref{L:dervanish} \[r_{(o(\rho)m,m)}(\Z(\fm))=\Z([0,o(\rho)-2]+[0,0]_\rho)\otimes \rho v_\rho^{-1}.\] Thus the Geometric Lemma and \Cref{L:helpful2} imply
that
\[[r_{(2o(\rho)m+m,m)}(\Z(\fm)\times\Z(\fm))]=2[\Z([0,o(\rho)-2]+[0,0]_\rho+\fm)\otimes \rho v_\rho^{-1}].\]
If $\Z(\fm)\times \Z(\fm)$ is not irreducible, it has an irreducible subquotient $\pi$ different from $\Z(\fm+\fm)$ with the same cuspidal support. But by above computation and Frobenius reciprocity, we conclude that $\mathcal{D}_{r,\rho v_\rho^{-1}}(\pi)\cong \mathcal{D}_{r,\rho v_\rho^{-1}}(\Z(\fm+\fm))\neq 0$ and hence $\pi\cong \Z(\fm+\fm)$. Since $\Z(\fm+\fm)$ appears only with multiplicity $1$ in $\Z(\fm)\times \Z(\fm)$, the claim follows.
\end{ex}
\bibliographystyle{abbrv}
\bibliography{References.bib}
\end{document}